\documentclass[twoside,11pt]{article}

\usepackage{jmlr2e}

\usepackage{lineno}
\usepackage{amsopn}
\usepackage{color}
\usepackage{subfig}
\usepackage{amsmath}
\usepackage{enumitem}	
\usepackage{amssymb}
\usepackage{hyperref}
\usepackage{lipsum}
\usepackage{amsfonts}
\usepackage{graphicx}
\usepackage{epstopdf}
\usepackage{algorithm}
\usepackage{algpseudocode}
\ifpdf
\DeclareGraphicsExtensions{.eps,.pdf,.png,.jpg}
\else
\DeclareGraphicsExtensions{.eps}	
\fi

\allowdisplaybreaks[4]

\def\R{\mathbb{R}}

\newcommand{\zz}{^{\top}}

\newcommand{\lb}{\lambda}

\newcommand{\diag}{\operatorname*{diag}}

\newcommand{\st}{\textnormal{s.t.}}
\newcommand{\dist}{\textnormal{dist}}

\ShortHeadings{IALAM for Training the Sparse Leaky ReLU Network}{Liu, Liu and Chen}
\firstpageno{1}

\begin{document}
	
	\title{An Inexact Augmented Lagrangian Algorithm for Training
		Leaky ReLU Neural Network with Group Sparsity}
	
	\author{\name Wei Liu\email liuwei175@lsec.cc.ac.cn
		\AND
		\name Xin Liu \email liuxin@lsec.cc.ac.cn \\
		\addr Institute of Computational Mathematics and Scientific/Engineering Computing \\
		Academy of Mathematics and Systems Science\\ Chinese Academy of Sciences \\
		Beijing 100190, China
		\AND
		\name Xiaojun Chen \email xiaojun.chen@polyu.edu.hk \\
		\addr Department of Applied Mathematics\\
		The Hong Kong Polytechnic University\\
		Hung Hom, Kowloon, Hong Kong}
	
	\editor{}
	\maketitle
	
	\begin{abstract}
				The leaky ReLU network with a group sparse regularization term has been widely used in the recent years.
		However, training such network yields a nonsmooth nonconvex optimization problem and there exists a lack of approaches to compute a stationary point deterministically.
			In this paper, we first resolve the multi-layer composite term in the original optimization problem by introducing auxiliary variables and additional constraints. We show the new model has a nonempty and bounded solution set and its feasible set satisfies the Mangasarian-Fromovitz constraint qualification.  Moreover, we show the relationship between the new model and the original problem.
		{Remarkably}, we propose an inexact augmented Lagrangian algorithm for solving the {new model}, and show the convergence of the algorithm to a KKT point.  Numerical experiments demonstrate that
		our algorithm is more efficient for training sparse leaky ReLU neural networks than some well-known algorithms.
	\end{abstract}
	
	
	\section{Introduction}
	
	In this paper, we focus on the parameter estimation problem of
	the leaky ReLU network~\citep{maas2013rectifier} with the $l_{2,1}$ regularizer, {which pursues the group sparsity}. The problem can be formulated as
	\begin{equation}\label{eq:odnn}
		\min_{w,b}\frac{1}{N}\sum_{n=1}^N\|\sigma(W_{L}\sigma(\cdots\sigma(W_1x_n+b_1)+\cdots)+b_{L})-y_n\|^2+\mathcal{R}_1(w).
	\end{equation}
	Here $\{x_n\in\R^{N_0}\}_{n=1}^N$ and $\{y_n\in\R^{N_L}\}_{n=1}^N$ are the given input and output data, respectively,
	$\sigma$ {stands for} the {component-wise} activation function, variables $W_{\ell} \in \R^{N_{\ell}\times N_{\ell-1}}$ and
		$b_{\ell}\in \R^{N_{\ell}}$ represent
	the weight matrices and the bias vectors for all $\ell\in[L]$,  respectively, $\mathcal{R}_1:\mathbb{R}^{\widetilde{N}} \rightarrow \mathbb{R}$ is the sparse regularizer of $w$. {For conveninece, we denote}
	\begin{align*}
		&
		\sigma(z)=\max\{z,  \alpha z\}, \,\,\mathcal{R}_1(w):=\lb_w\sum_{\ell=1}^L\|W_{\ell}\|_{2,1}=\lb_w\sum_{\ell=1}^L\sum_{j=1}^{N_{\ell-1}}\|(W_{\ell})_{\cdot,j}\|,\\
		&w=\left(\mathrm{vec}(W_1)\zz,\ldots,\mathrm{vec}(W_L)\zz\right)\zz\in\R^{\widetilde{N}},\,\,b=\left(b_1\zz,\ldots,b_L\zz\right)\zz\in\R^{\overline{N}},
	\end{align*}	
	{where $\|\cdot\|$ refers to} the $l_2$ norm,  
	$\lb_w> 0$, $(W_{\ell})_{\cdot,j}$ {stands for} the $j$-th column of $W_{\ell}$, $\mathrm{vec}(W_{\ell}) \in \R^{N_{\ell-1} N_{\ell}}$
	{is} the {column-wise} vectorization of $W_{\ell}$, $\widetilde{N}:=\sum_{\ell=1}^LN_{\ell}N_{\ell-1}$ and $\overline{N}:=\sum_{\ell=1}^LN_{\ell}$, $\max\{z, \alpha z\}= (\max\{z_1,\alpha z_1\},...,\max\{z_K,\alpha z_K\})\zz$ for any $z\in\R^K$ and $0<\alpha<1$ and $[L]$ denotes $\{1,2,\ldots,L\}$.
	
	It is worth noting that the {activation functions} ReLU and the leaky ReLU {get more and more popular in recent applications}, as they can	 {alleviate the overfitting phenomenon and pursue the model (neuron) sparsity, e.g., almost half of the neurons in the ReLU network are zero 
		\citep{jarrett2009best,nair2010rectified,glorot2011deep,maas2013rectifier,dahl2013improving,he2015para,agarap2018deep}.
		Moreover, the performance of the leaky ReLU network is {reported to be} slightly better than that of the ReLU network \citep{maas2013rectifier,pedamonti2018comparison}. As we will show in Theorem \ref{thm:nonem}, the leaky ReLU network with a regularization term has a nonempty and bounded solution set, while the ReLU network with a regularization term does not have the property (see a counterexample given by \cite{liu2021auto}).} For simplicity, we focus on the leaky ReLU network in this paper.
	Our new model, algorithm and theoretical analysis can be generalized to the ReLU network easily (see Remark \ref{rem:leReLU}).
	
	In training a deep neural network (DNN, e.g., the leaky ReLU network),
	regularization techniques play an important role in reducing the generalization error (also called the test error) \citep{goodfellow2016deep}. {  The $l_2$ regularizer  (i.e., $\|\cdot\|^2$, also called the weight decay) has been widely used for training the DNN \citep{goodfellow2016deep}.} Recently, sparse regularizers, such as the lasso regularizer \citep{goodfellow2016deep} and the group sparse regularizer \citep{zhou2010exclusive,wen2016sparse, Feng2017,yoon2017combined, Scardapane2017group},
	{are superior to the  $l_2$ regularizer in pursuing}
	the parameter sparsity and lead to theoretical improvement in efficiency \citep{hoefler2021sparsity}. Moreover, \cite{wen2016sparse} show that by using the gradient descent methods, less training time is required by DNN with a group sparse regularizer compared with that required by DNN with a lasso regularizer.
	{The} group sparse regularizer {also appears in} convolution neural network \citep{bui2021structured} {and other} machine learning {problems} \citep{meier2008group, jenatton2011structured,simon2013sparse}, etc.  
	Hence, {we focus on training the leaky ReLU network with the $l_{2,1}$ regularizer for pursuing the group sparsity.}
	
	The stochastic gradient descent based methods{, including the stochastic gradient descent methods (SGD),} are widely used {in training DNN including} the leaky ReLU network with group sparsity, while they neglect {the fact that the subdifferentials of the objective function at those nondifferentiable points are not available  \citep{abadi2016tensorflow, Paszke2019pytorch}.
		Instead, they calculate the ``gradient'' via the ``chain rule''
		brutely no matter the ``chain rule" applies or not \citep{mjt_dlt,bolte2021conser}). Therefore,
		gradient descent based approaches can not deterministically yield Clarke stationary points (see Definition \ref{def:sta}) and may encounter numerical troubles in extreme cases.}
	Recently, \citet{davis2020ssgd} prove that the stochastic subgradient (SSGD) method for training the nonsmooth network
	{can obtain} Clarke stationary {points} for the leaky ReLU network with probability $1$. 
	However, they {have not explained how to calculate a subgradient practically.}   Moreover, even though a Clarke stationary point is obtained, it may be far away from any local minimizer (see Example \ref{exam:exam0}).
	
	
	{Recently, approaches \citep{carreira2014distributed,taylor2016training,lau2018proximal,zeng2019global,cui2020multicomposite,Evens2021,liu2021auto} based on a new methodology that introduces auxiliary variables and constraints to resolve the multi-layer nonsmoothness, have been proposed, which have chance to find stationary points deterministically.}
	Specifically, \cite{cui2020multicomposite} propose {an} $l_1$ penalty method, {which yields a directional stationary point theoretically, for training} the DNN with
	{piecewise activation functions and} an $l_2$ regularizer. 
	\cite{liu2021auto} propose a smoothing method that finds a Clarke stationary point for the two-layer ReLU network. To the best of our knowledge, algorithms with
	{guaranteed global convergence to} KKT points for a nonsmooth deep neural network with group sparsity have not been {developed yet}.
	\subsection{{Motivation}}\label{sec:motiva}
	
	
	
	{In this paper, we aim to explore efficient approaches for
		solving problem \eqref{eq:odnn} with guaranteed convergence. Hence, we pay our attention to the methods which introduce auxiliary varaibles and constraints to resolve the
		multi-layer nonsmoothness. To peel the complicated composite objective of \eqref{eq:odnn} like bamboo shoot, we first introduce the following varaiables,}
	\begin{equation}
		\label{eq:vdefine}
		v:=\left(v_{1,1}\zz,v_{2,1}\zz,\ldots,v_{1,L}\zz,v_{2,L}\zz,\ldots,v_{N,L}\zz\right)\zz\in\R^m,
	\end{equation}
	where $m:=N\overline{N}$, 
	$v_{n, \ell}:=\sigma(W_{\ell}\sigma(\cdots\sigma(W_{1} x_{n}+b)_{+}+\cdots)+b_{\ell})$, $v_{n,0}:=x_n$ for all $\ell\in[L]$ and $n\in[N]$.
	Specifically, \cite{liu2021auto} introduce a linearly constrained model for training a  two-layer ReLU network with a regularization term.
	{ For solving the sparse leaky ReLU network with more than two layers, 
		we introduce} in this paper 
	a regularization term $\mathcal{R}_2(v): \mathbb{R}^{m}\mapsto \mathbb{R}$ by $$\mathcal{R}_2(v):=\lb_v\|v\|^2,$$  and a new group of variables
	\begin{equation}
		\label{eq:udefine}
		u=\left(u_{1,1}\zz,u_{2,1}\zz,\ldots,u_{1,L}\zz,u_{2,L}\zz,\ldots,u_{N,L}\zz\right)\zz\in\R^m,
	\end{equation}
	where $\lb_v>0$, $u_{n,\ell}=W_{\ell}v_{n,\ell-1}+b_{\ell}$ for all $n\in[N]$ and $\ell\in[L]$.  Then, we derive the following model
	\begin{equation}\label{eq:dnn}
		\tag{P}
		\begin{aligned}
			\min_{w,b,v,u}\, & \, \bar{\mathcal{O}}(w,v):=\frac{1}{N}\sum_{n=1}^N\|v_{n,L}-y_n\|^2+\mathcal{R}_1(w)+\mathcal{R}_2(v)\\
			\st \,\, & \, \sigma(u_{n,\ell})-v_{n,\ell}=0,  \,u_{n,\ell}- (W_{L}v_{n,L-1}+b_{L})=0,  \\
			&\, n\in[N], \, \ell\in[L].
		\end{aligned}
	\end{equation}
	{The regularization teams $\mathcal{R}_1$ and $\mathcal{R}_2$ lead to the level boundedness of} the objective function $\bar{\mathcal{O}}$. Moreover, $\mathcal{R}_1$ imposes
	{column-wise} sparsity of the weight matrices $W_\ell$ for all $\ell\in [L]$. By defining {the} linear operator $\Psi(v):\R^{m}\mapsto\R^{m\times \widetilde{N}}$ and {the} matrix $A\in\R^{m\times \overline{N}}$ as
	\begin{align*}
		\Psi(v)&=	
		\left[
		\begin{array}{cccc}
			X\zz\otimes I_{N_1} & \ldots & \ldots& 0\\
			0 & V_1\zz\otimes I_{N_2} & \ldots& 0\\
			0 & \ldots&\ldots& 0 \\
			0 & \ldots & \ldots& V_{L-1}\zz\otimes I_{N_L}
		\end{array}
		\right]\,\mbox{{and}}\\ A&=
		\left[
		\begin{array}{cccc}
			e_N \otimes I_{N_1} & \ldots & \ldots& 0\\
			0 & e_N \otimes I_{N_2}  & \ldots& 0\\
			0 & \ldots&\ldots& 0 \\
			0 & \ldots & \ldots& e_N \otimes I_{N_L}
		\end{array}
		\right],
	\end{align*}
	{respectively,} where $X:=(x_1,x_2,\ldots,x_n)$, $V_{\ell}:=(v_{1,\ell},v_{2,\ell},\ldots,v_{N,\ell})\in\R^{N_{\ell}\times N}$ for all $\ell\in[L]$, $\otimes$ represents the Kronecker product and $e_N\in\R^N$ denotes the {all one} vector, the constraint set of problem \eqref{eq:dnn} can be simply written as
	
	$$v- \sigma(u)= 0, u= \Psi(v)w+Ab.$$
	
	Due to the nonsmoothness of $\sigma(u)$,
	problem (P) does not satisfy a standard constraint qualification for mathematical programming. We consider to have $v\ge \sigma(u)$ as a constraint and add a penalty term $\beta\zz (v-\sigma(u))$ in the objective function, where $$\beta = (\beta_1 e_{NN_{1}}\zz, \ldots, \beta_L e_{NN_{L}}\zz)\zz\in\R^m$$ with constants $\beta_{\ell}>0$ for all $\ell\in[L]$.
	Note that {the inequality} $v\geq  \sigma(u)$ is equivalent to {the inequalities} $v-u \geq 0$ and $ v-\alpha u\geq 0$, we then {present} the partial $l_1$ penalty model for problem \eqref{eq:dnn} as follows
	\begin{equation}\label{eq:dnn33}\tag{PP}
		\begin{aligned}
			\min_{w,b,v,u}\, & \, \mathcal{O}(w,v,u)=\bar{\mathcal{O}}(w,v)+\beta\zz (v-  \sigma(u))\\
			\st \,\,\, & \, v-u \geq 0, v-\alpha u\geq 0, 
			u= \Psi(v)w+Ab.
		\end{aligned}
	\end{equation}
	For {brevity}, we denote the feasible sets of problems \eqref{eq:dnn} and \eqref{eq:dnn33} by $\Omega_1$ and $\Omega_2$, respectively.

	It is worth noting that problem \eqref{eq:dnn33} is a nonsmooth nonconvex mathematical programming, where {the second term} $\beta \zz (v-  \sigma(u))$ is nonsmooth,
	$\mathcal{R}_1$ is a convex nonsmooth regularizer, the inequality constraints 
	are linear, the equality constraints 
	are nonconvex bilinear. Hence, both the objective function and the feasible region of problem \eqref{eq:dnn33} {are} much more {complicated} than the optimization problem for two-layer network proposed by \cite{liu2021auto}, which is a linearly constrained programming. Hence, {the approaches therein can not be straightforwardly extended to solve}
	problem \eqref{eq:dnn33}.
	
	
	\subsection{Contributions}
	We consider a regularized minimization model (P) {with auxiliary varaibles and nonsmooth constraints for training} the leaky ReLU network with group sparsity.
	We {investigate its partial $l_1$ penalty model (PP) and establish the relationships between these two models with respect to} global minimizers, local minimizers, and stationary points under some mild conditions. Moreover, we show that the solution set of problem (PP) is bounded and any feasible point of problem (PP) satisfies the Mangasarian-Fromovitz constraint qualification. Based on these results, we {theoretically verify the equivalence between the KKT points and the limiting stationary points of (PP), and further prove that any KKT point of (PP)} is an MPCC W-stationary point of problem (P).
	
	{By exploiting the structure of problem (PP), we}
	propose an inexact augmented Lagrangian method, whose subproblem at each iteration is solved by an alternating minimization method (IALAM). Different from the existing inexact augmented Lagrangian methods for nonsmooth nonconvex optimization problems \citep{lu2012augmented,chen2017augmented}, we design a new {rule for updating} the Lagrangian penalty parameter. 
	{We} also prove that {any iterate sequence generated by IALAM has accumulation points, any of which is} a {limiting stationary} point (or equivalently KKT point) of problem (PP) without assuming the the existence of accumulation points. Moreover, any {limiting stationary} point of problem (PP) is a Clarke stationary point of problem (PP).
	
	The numerical experiments demonstrate that {IALAM}, equipped with {prefixed} algorithm parameters, outperforms the popular SGD-based methods (e.g., Adam, Adadelda, and vanilla SGD) and ProxSGD {
		in solving problems arisen from}
	both synthetic and MNIST {data sets}. More specifically, compared with SGD-based methods, {IALAM} achieves lower training error and test error, and obtains sparser solutions. 
	
	By {applying} IALAM  {to training} both the ReLU and the leaky ReLU {networks under} the same settings, we find that the leaky ReLU network with a small positive $\alpha$ (e.g., $\alpha=0.01$) often leads to slightly better performance than that of the ReLU network, {which verifies the observations of}
	\citet{maas2013rectifier,pedamonti2018comparison}. 

	\subsection{Organizations}
	
	The rest of this paper is organized as follows. In Section \ref{sec:preno}, we introduce some notations, preliminary definitions, lemmas, and results. The relationships between the models \eqref{eq:dnn} and \eqref{eq:dnn33} are illustrated in Section \ref{sec:modelana}. In Section \ref{sec:alm}, we propose an augmented Lagrangian method with the alternating minimization for solving problem \eqref{eq:dnn33} and {establish} the global convergence of the algorithm. In Section \ref{sec:numeri}, we illustrate the performance of our proposed algorithm through extensive numerical experiments. Concluding remarks are {drawn} in the last section.
	

	\section{Notations and Preliminaries}\label{sec:preno}
	
	In this section, we introduce some notations, preliminary definitions, examples, and lemmas.
	
	The $m\times m$ identity matrix is denoted by $I_{m}$. We use $\mathbb{N}_+$ to represent the set of positive integers.
	Given a point $z \in \mathbb{R}^{m}$ and $\epsilon>0, \mathcal{B}_{\epsilon}(z)$ denotes a closed ball centered at $z$ with radius $\epsilon$,  $(\mathrm{sign}(z))_i$ denotes the sign function of $z_i$, and $\diag(z)$ denotes the diagonal matrix whose diagonal vector is $z$.  $\dist(z^*,\Omega)=\min_{z\in\Omega}\|z-z^*\|$ represents the distance from a point $z^*$ to a nonempty closed set $\Omega$. We use $\mathrm{int}(\Omega)$, $\mathrm{co}(\Omega)$ to represent the interior and convex hull of $\Omega$, respectively. The indicator function of a set $\Omega$ is denoted by $\delta_{\Omega}$. We let $\nabla_{(z_1,z_2)}f(z)=\nabla_{z_1}f(z) \times \nabla_{z_2}f(z)$ for a smooth function $f$ with respect to $z=(z_1, z_2)$.
	
	Let $H$ be a symmetric positive definite matrix, $\Omega \subseteq \mathbb{R}^{m}$ be a convex set, and $\operatorname{Proj}^{H}_{\Omega}\left(z^{*}\right)=\arg \min \left\{\left\|z-z^{*}\right\|_{H}: z \in \Omega\right\}$ be the orthogonal projection of a vector $z^{*} \in \mathbb{R}^{m}$ onto $\Omega$  \citep{facchinei2003finite}. If  $H$ is the identity matrix, we will use $\operatorname{Proj}_{\Omega}\left(z^{*}\right)$ instead.
	The proximal mapping $\operatorname{Prox}_{f}(\cdot)$ of a proper closed convex function $f$ is defined as
	$
	\operatorname{Prox}_{f}(z^*)=\arg\min_{z \in \R^{n}}\left\{\frac{1}{2}\|z-z^*\|^{2}+f(z)\right\}.
	$
	\subsection{Subdifferentials and Stationarity}\label{sec:1-2}
	
	Let $f: \Omega \rightarrow \mathbb{R}$ be a locally Lipschitz continuous and directionally differentiable function 
	defined on an open set $\Omega \subseteq \mathbb{R}^{n}.$ The directional derivative of $f$ at $z$ along the direction $d$ is defined as
	$$
	f^{\prime}(z ; d)=\lim _{t \downarrow 0} \frac{f(z+t d)-f(z)}{t}.
	$$
	It is worth mentioning that any piecewise smooth and Lipschitz continuous function is directionally differentiable \citep{mifflin}.
	
	Let $
	\bar{z} \in \Omega \text { be given}
	$.  
	The Clarke subdifferential \citep{clarke1990optimization} of $f$ at $\bar{z}$ is defined by
	$$\partial^c f \left(\bar{z}\right)=\operatorname{co}\left\{\lim _{z \rightarrow \bar{z}} \nabla f(z): f \text{ is smooth at } z\right\}.$$  According to \cite[Definition 8.3]{roc1998var}, the limiting subdifferential of $f$ at $\bar{z}$ is defined by ${\partial} f(\bar{z}) :=$
	$$
	\begin{aligned}
		\left\{v : \exists z^{k} \stackrel{f}{\rightarrow} \bar{z}, v^{k} \rightarrow v \text { such that } \liminf _{z \rightarrow z^k} \frac{f(z)-f\left(z^k\right)-\left\langle v^k, z-z^{k}\right\rangle}{\left\|z-z^k\right\|} \geq 0,\,\, \forall  k\right\}, \\
	\end{aligned}
	$$
	where 
	$z^{k} \stackrel{f}{\rightarrow} \bar{z}$ means that $z^{k} \rightarrow \bar{z}$ and $f\left(z^{k}\right) \rightarrow f(\bar{z})$. 
	If $f$ is convex, then $\partial f$ coincides with $\partial^c f$. If $f$ is furthermore smooth, it holds that $\partial f(z)=\partial^c f(z)=\{\nabla f(z)\}$.  In general, one has
	$
	\operatorname{co}(\partial f(\bar{z}))=\partial^{c} f(\bar{z}).
	$
	
	For $z\in\R^n$, we have
	$$\partial\|z\|=\partial^c\|z\|= \begin{cases}\frac{z}{\|z\|}, & \text{if }\|z\| \neq 0, \\ \left\{r : r \in \mathbb{R}^{n},\|r\| \leq 1\right\}, &\text{if } \|z\|=0.\end{cases}$$
	
	Let $\mathcal{T}_{\Omega}(\bar{z})=\left\{d: d=\lim _{z \in \Omega, z \rightarrow \bar{z}, t \downarrow 0} \frac{z-\bar{z}}{t}\right\}$ be the tangent cone of a set $\Omega$ at $\bar{z}\in \Omega$ and $\mathcal{N}_{\Omega}(z)$ be the limiting normal cone at $z \in \Omega$. If $\Omega$ is a convex set, then $\mathcal{N}_{\Omega}(z)$ coincides with the classical (Clarke) normal cone in the convex analysis, where the Clarke normal cone $\mathcal{N}^c_{\Omega}(z)$ is defined by $\mathcal{N}^c_{\Omega}(z)=\mathrm{cl}\mathrm{co}\mathcal{N}_{\Omega}(z)$.
	
	\begin{definition}\label{def:sta}
		Let $\mathcal{Z}$ be a closed set in $\Omega$.  We call $\bar{z}\in\mathcal{Z}$ a d(irectional)-stationary point of $\min_{z\in\mathcal{Z}}f(z)$ if
		$
		f^{\prime}(\bar{z};d)\geq 0\,\,\text{ for all }d\in\mathcal{T}_{\mathcal{Z}}(\bar{z}).
		$
		We say that
		a point $\bar{z} \in \mathcal{Z}$ is a limiting stationary point, a C(larke)-stationary point of $\min_{z\in\mathcal{Z}}f(z)$ if $0 \in \partial f(\bar{z})+\mathcal{N}_{\mathcal{Z}}(\bar{z})$, $0 \in \partial^cf(\bar{z})+\mathcal{N}_{\mathcal{Z}}^c(\bar{z})$, respectively.
	\end{definition}
	Based on Definition \ref{def:sta}, we have the following relationships 
	\begin{equation}
		\label{eq:relasta}
		\text{local minimizer } \Rightarrow \text{ d-stationary }\Rightarrow \text{ limiting stationary }  \Rightarrow \text{ C-stationary}.
	\end{equation}
	
	Furthermore, $0 \in \partial^cf(\bar{z})+\mathcal{N}_{\mathcal{Z}}^c(\bar{z})$ implies
	$$
	f^{\circ}(\bar{z};d):=\limsup _{z \rightarrow \bar{z}, t \downarrow 0} \frac{f(z+t d)-f(z)}{t}  \geq 0,\,\,\text{for all }d\in\mathcal{T}_{\mathcal{Z}}(\bar{z}).
	$$
	
	If a certain constraint qualification condition (see Subsection \ref{sec:pre}) holds at $\bar{z}\in\mathcal{Z}$, then $\bar{z}$ being {a} limiting stationary point is a necessary condition for $\bar{z}$ to be a local minimizer of $f$ (see an example given by \citet{chen2017augmented}).
	
	
	In general, a C-stationary point {  is not a good candidate} for a local minimizer. 
	We end this subsection with an example on the DNN to illustrate that a C-stationary point may not be a
	limiting stationary point, {  which further may not be a local minimizer.}
	\begin{example}\label{exam:exam0}
		Consider
		\begin{equation}\label{eq:exam0}
			\min_{w_1\in\R,w_2\in\R,b_1\in\R,b_2\in\R}\left(\left(w_{2}\sigma\left(w_{1}+b_{1}\right)+b_{2}\right)+1\right)^{2}+\left(\left(w_{2}\sigma\left(2w_{1}+b_{1}\right)+b_{2}\right)-1\right)^{2}.	\end{equation}
		Let $f(w_1,w_2, b_1,b_2)$ be the objective function of \eqref{eq:exam0}, $w_{2}^{*}=1,\, b_{1}^{*}=0,\, w_{1}^{*}=0,\, b_{2}^{*}=0$, we have
		\begin{equation*}
			\begin{aligned}
				&\partial^cf(w^*_1,w^*_2, b^*_1,b^*_2)
				=\left\{(t,0,s,0)^T: t\in [2\alpha-4, 2-4\alpha],
				s\in [-2+2\alpha, 2-2\alpha]\right\},\\
				&\partial\left(f(w^*_1,w^*_2, b^*_1,b^*_2)\right)
				\\&=\left\{(-2\alpha,0,0,0)\zz,(2\alpha-4,0,2\alpha-2,0)\zz,(2-4\alpha,0,2-2\alpha,0)\zz,(-2,0,0,0)\zz\right\},\\
				&f(w^*_1+\epsilon,w^*_2, b^*_1,b^*_2)=5\epsilon^2-2\epsilon+2<2=f(w^*_1,w^*_2, b^*_1,b^*_2) \text{ for some small positive number }\epsilon.
			\end{aligned}
		\end{equation*}
		For some $0<\alpha<\frac{1}{2}$, $(w^*_1,w^*_2, b^*_1,b^*_2)$ is a C-stationary point of \eqref{eq:exam0}, but it is neither a local minimizer nor a limiting stationary point of \eqref{eq:exam0}. Moreover, one can see that $(1,2,-1,-1)$ is a global minimizer of \eqref{eq:exam0}, at which the function value is 0.
	\end{example}
	
	\subsection{Necessary Optimality Conditions.}\label{sec:pre}
	In this subsection, we provide first order necessary optimality conditions for local minimizers of problems \eqref{eq:dnn} and \eqref{eq:dnn33}, respectively. Let
	\begin{equation}\label{eq:omega3}
		\mathcal{C}(v,u):=\left(\begin{matrix}
			u-v \\ \alpha u-v 
		\end{matrix}\right).
	\end{equation}
	\begin{definition}
		We say that $(w^*,b^*,v^*,u^*)$ is a  KKT point of problem \eqref{eq:dnn33} if there exist vectors $\mu\in\R_+^{2m}$ and $\xi\in\R^{m}$ such that
		\begin{align}
			&0= \nabla_w\bar{\mathcal{O}}(w^*,v^*)+ \Psi(v^*)\zz \xi,  \quad 0= A\zz\xi,\label{eq:kktdnn2-1}\\
			&0= \nabla_v \bar{\mathcal{O}}(w^*,v^*)+\beta +\nabla_v \mu\zz \mathcal{C}( v^*,u^*)-\nabla_v \xi\zz  \Psi(v^*)w^*,\label{eq:kktdnn2-2}\\
			&0\in \partial_u (-\beta \zz \sigma(u^*)) +\nabla_u \mu\zz \mathcal{C}( v^*,u^*)+ \xi,\label{eq:kktdnn2-3}\\
			&\mathcal{C}( v^*,u^*)\leq 0,\,\,\mu\zz\mathcal{C}( v^*,u^*)=0, \,\,u^*- \Psi(v^*)w^*-Ab^*=0.\label{eq:kktdnn2-4}
		\end{align}
	\end{definition}
	
	Since $v-\sigma (u)=0$ can be written as the following complementarity problem
	$$
	v-u\geq 0, (v-u)(v-\alpha u)=0, v-\alpha u\geq 0,
	$$
	we can define the Mathematical Programming with Complementarity Constraints (MPCC) W(eakly)-stationary point \citep{Scheel2000mpcc,guo2021mpcc} of problem \eqref{eq:dnn} as follows.
	\begin{definition}
		We say that $(w^*,b^*,v^*,u^*)\in\Omega_1$ is an MPCC W-stationary point of problem \eqref{eq:dnn}
		if there exist vectors $\mu^1\in\R^{m}$, $\mu^2\in\R^{m}$ and $\xi\in\R^{m}$ such that
		\begin{align}
			&0= \nabla_w\bar{\mathcal{O}}(w^*,v^*)+ \Psi(v^*)\zz \xi,  \quad 0= A\zz\xi,\label{eq:mpcckktdnn-1}\\
			&0= \nabla_v \bar{\mathcal{O}}(w^*,v^*)-\mu^1-\mu^2 -\nabla_v \xi\zz  \Psi(v^*)w^*,\label{eq:mpcckktdnn-2}\\
			&0= \mu^1+\alpha \mu^2+ \xi,\label{eq:mpcckktdnn-3}\\
			&\left(\mu^1\right)\zz(v^*- u^*)=0, \,\, \left(\mu^2\right)\zz( v^*- \alpha u^*)=0.\label{eq:mpcckktdnn-5}
		\end{align}
		We say $(w^*,b^*,v^*,u^*)\in\Omega_1$ is an MPCC C(larke)-stationary point of problem \eqref{eq:dnn}, if it is an MPCC W-stationary point of problem \eqref{eq:dnn} and
		$\mu^1_i\mu^2_i\ge 0$  for $u^*_i=v^*_i=0.$
	\end{definition}

	To ensure that
	a local minimizer of problem \eqref{eq:dnn33} is a KKT point, the following Lemma shows that the feasible set of problem \eqref{eq:dnn33} satisfies the Mangasarian–Fromovitz constraint qualification (MFCQ), which is a standard constraint qualification for nonlinear programmings~\citep{Mangasarian1967mfcq}.
	\begin{lemma}\label{lem:mfcqsys}
		The MFCQ holds at $(w^*,b^*,v^*,u^*)\in\Omega_2$ for problem \eqref{eq:dnn33}, i.e.,  there exist no nonzero vectors \(\xi \in \mathbb{R}^{m}, \mu \in \mathbb{R}_{+}^{2m}\)
		such that \(\mu\zz \mathcal{C}( v^*,u^*)=0\) and
		\begin{align}
			&0= \Psi(v^*)\zz \xi,  \quad 0= A\zz\xi,\label{eq:kktdnn3-1}\\
			&0= \nabla_v \mu\zz \mathcal{C}( v^*,u^*)-\nabla_v \xi\zz  \Psi(v^*)w^*,\label{eq:kktdnn3-2}\\
			&0= \nabla_u \mu\zz \mathcal{C}( v^*,u^*)+ \xi.\label{eq:kktdnn3-3}
		\end{align}
	\end{lemma}
	\begin{proof}		
		We prove that the linear system \eqref{eq:kktdnn3-1}--\eqref{eq:kktdnn3-3} only has a zero solution.
		
		Let $\mu=((\mu^1_{1,1})\zz,(\mu^1_{2,1})\zz,\ldots,(\mu^1_{1,L})\zz,\ldots,(\mu^1_{N,L})\zz,(\mu^2_{1,1})\zz,\ldots,(\mu^2_{N,L})\zz)\zz$ and $\xi=\\(\xi_{1,1}^{\top}, \xi_{2,1}^{\top}, \ldots, \xi_{N, 1}^{\top}, \xi_{1,2}^{\top}, \ldots, \xi_{N, L}^{\top})\zz$, where $\mu^1_{n,\ell},\mu^2_{n,\ell}\in\R^{N_{\ell}}_+$, $\xi_{n,\ell}\in\R^{N_{\ell}}$ for all $n\in[N]$ and $\ell\in[L]$. Notice that $u^*= \Psi(v^*)w^*+Ab^*$ is equivalent to $u^*_{n,\ell}- (W^*_{\ell}v^*_{n,\ell-1}+b^*_{\ell})=0$ for all $n\in[N]$ and $\ell\in[L]$, the equalities \eqref{eq:kktdnn3-2} and \eqref{eq:kktdnn3-3} yield
		\begin{align}
			&0 = \nabla_v \left(\sum_{n=1}^{N}\sum_{\ell=1}^{L} \left(\mu^1_{n,\ell}\right)\zz v^*_{n,\ell}+  \left(\mu^2_{n,\ell}\right)\zz v^*_{n,\ell}+ \xi_{n,\ell}\zz W^*_{\ell}v^*_{n,\ell-1}\right),\label{eq:derive1}\\
			&0 = \nabla_u \left(\sum_{n=1}^{N}\sum_{\ell=1}^{L} \left(\mu^1_{n,\ell}\right)\zz u^*_{n,\ell}+  \alpha\left(\mu^2_{n,\ell}\right)\zz  u^*_{n,\ell}+ \xi_{n,\ell}\zz u^*_{n,\ell}\right).\label{eq:derive2}
		\end{align}		
		
		We first consider the coefficient with respect to $v^*_{n,L}$ in \eqref{eq:derive1} for all $n\in[N]$, which yields $0=-(\mu^1_{n,L}+\mu^2_{n,L})$. Together with {the inequalities} $\mu^1_{n,L}\geq 0$ and $\mu^2_{n,L}\geq 0$, we obtain that $\mu^1_{n,L}=\mu^2_{n,L}= 0$ for all $n\in[N]$. We then consider the coefficient with respect to $u^*_{n,L}$ in \eqref{eq:derive2}  for all $n\in[N]$, which implies that $0=\mu^1_{n,L}+\alpha\mu^2_{n,L}+ \xi_{n,L}$. Hence, we have $\xi_{n,L}= 0$ for all $n\in[N]$. Substituting $\xi_{n,L}=\mu^1_{n,L}=\mu^2_{n,L}= 0$ for all $n\in[N]$ into \eqref{eq:derive1} and \eqref{eq:derive2}, we obtain that
		\begin{align*}
			&0 = \nabla_v \left(\sum_{n=1}^{N}\sum_{\ell=1}^{L-1} \left(\mu^1_{n,\ell}\right)\zz v^*_{n,\ell}+  \left(\mu^2_{n,\ell}\right)\zz v^*_{n,\ell}+ \xi_{n,\ell}\zz W^*_{\ell}v^*_{n,\ell-1}\right), \\&0 = \nabla_u \left(\sum_{n=1}^{N}\sum_{\ell=1}^{L-1} \left(\mu^1_{n,\ell}\right)\zz u^*_{n,\ell}+  \alpha\left(\mu^2_{n,\ell}\right)\zz  u^*_{n,\ell}+ \xi_{n,\ell}\zz u^*_{n,\ell}\right).
		\end{align*}		
		Then, we obtain that $\xi_{n,\ell}=\mu^1_{n,\ell}=\mu^2_{n,\ell}= 0$ for all $n\in[N]$ and $\ell\in[L]$ by mathematical induction. 
		This completes the proof.
	\end{proof}
	
	Under MFCQ, we can obtain the following equivalence between
	the limiting stationary points and KKT points of problem \eqref{eq:dnn33}. 
	\begin{theorem}\label{thm:cone1}
		$(w^*,b^*,v^*,u^*)$ is a {limiting stationary} point of problem \eqref{eq:dnn33} if and only if \\$(w^*,b^*,v^*,u^*)$ is a KKT  point of problem \eqref{eq:dnn33}.
	\end{theorem}
	\begin{proof}
		Since the MFCQ holds at $(w^*,b^*,v^*,u^*)\in\Omega_2$ for problem \eqref{eq:dnn33}, then \cite[Theorem 6.14]{roc1998var} yields that $\mathcal{N}_{\Omega_2}(w^*,b^*,v^*,u^*)$ equals to $$ \left\{ \nabla \left(\mu\zz \mathcal{C}( v^*,u^*)+\xi\zz (u^*- \Psi(v^*)w^*-Ab^*)\right)  : \,\, \mu\zz \mathcal{C}( v^*,u^*)=0,\,\,\mu\in\R_+^{2m},\,\,\xi\in\R^{m}\right\}.$$
		
		If $(w^*,b^*,v^*,u^*)$ is a  {limiting stationary} point of problem \eqref{eq:dnn33}, then there exist vectors $\mu^1,\mu^2\in\R_+^{m}$ and $\xi\in\R^{m}$ such that $[(\mu^1)\zz,(\mu^2)\zz] \mathcal{C}( v^*,u^*)=0$ and
		\begin{equation}\label{eq:lstadnn3}
			\begin{aligned}
				&0\in\partial \mathcal{O}(w^*,v^*,u^*)+\nabla \left(\left(\mu^1\right)\zz (u^*-v^*)+\left(\mu^2\right)\zz (\alpha u^*-v^*)+\xi\zz (u^*- \Psi(v^*)w^*)\right),\\
				&0=A\zz\xi.
			\end{aligned}	
		\end{equation}
		Recall the definition of $\mathcal{O}$, the relationships \eqref{eq:kktdnn2-1}--\eqref{eq:kktdnn2-4} hold with $\mu=[(\mu^1)\zz,(\mu^2)\zz]\zz$. Hence $(w^*,b^*,v^*,u^*)$ is a KKT point of problem \eqref{eq:dnn33}.
		
		Conversely, if $(w^*,b^*,v^*,u^*)$ is a KKT point of problem \eqref{eq:dnn33}, then there exist vectors $\mu\in\R_+^{2m}$ and $\xi\in\R^{m}$ such that \eqref{eq:kktdnn2-1}--\eqref{eq:kktdnn2-4} hold. Let $\mu=[(\mu^1)\zz,(\mu^2)\zz]\zz$ with $\mu^1,\mu^2\in\R_+^{m}$. 
		Since the MFCQ holds at $(w^*,b^*,v^*,u^*)\in\Omega_2$ for problem \eqref{eq:dnn33}, we then obtain \eqref{eq:lstadnn3} by the form of $\mathcal{N}_{\Omega_2}(w^*,b^*,v^*,u^*)$. This completes the proof.
	\end{proof}

	Since the feasible set $\Omega_1$ of problem \eqref{eq:dnn} is no longer smooth, then MFCQ does not hold for problem \eqref{eq:dnn}. To ensure that
	a local minimizer of problem \eqref{eq:dnn} is also an MPCC C-stationary point, the following lemma shows that the feasible set of problem \eqref{eq:dnn} satisfies the No Nonzero Abnormal Multiplier Constraint Qualification (NNAMCQ) \citep{jane2013enhanced}. It is worth noting that the MPCC linear independent CQ \citep{Scheel2000mpcc,guo2021mpcc} does not hold for problem \eqref{eq:dnn}.
	\begin{lemma}\label{lem:nnamcq}
		The NNAMCQ holds at $(w^*,b^*,v^*,u^*)\in\Omega_1$ for problem \eqref{eq:dnn}, i.e., there exist no nonzero vectors $\mu\in\R^{m},\,\,\xi\in\R^{m}$
		such that 
		\begin{equation}\label{eq:nnamcqomega1}
			0\in\partial \mu\zz (\sigma(u^*)-v^*)+ \nabla_{(v,u)}\xi\zz (u^*- \Psi(v^*)), 0=\Psi(v^*)\zz \xi, 0=A\zz\xi.
		\end{equation}
	\end{lemma}
	
	\begin{proof}
		We prove that there exist no nonzero vectors \(\xi \in \mathbb{R}^{m}, \mu \in \mathbb{R}^{m}\)
		such that 
		\begin{align}
			&0= -\mu +\nabla_v \xi\zz (u^*- \Psi(v^*)w^*-Ab^*),\label{eq:kktdnn4-2}\\
			&0\in \partial_u \mu\zz\sigma(u^*)+ \xi,\label{eq:kktdnn4-3}
		\end{align}
		since it implies that there is no nonzero vectors $\mu\in\R^{m},\,\,\xi\in\R^{m}$ satisfying \eqref{eq:nnamcqomega1}.
		
		Notice that $u^*= \Psi(v^*)w^*+Ab^*$ is equivalent to $u^*_{n,\ell}- (W^*_{\ell}v^*_{n,\ell-1}+b^*_{\ell})=0$ for all $n\in[N]$ and $\ell\in[L]$, the equality \eqref{eq:kktdnn4-2} yields
		\begin{align}
			&0 = -\mu+\nabla_v \left( \sum_{n=1}^{N}\sum_{\ell=1}^{L}\xi_{n,\ell}\zz W^*_{\ell}v^*_{n,\ell-1}\right)\label{eq:derive3}.
		\end{align}		
		We first consider the coefficient with respect to $v^*_{n,L}$ in \eqref{eq:derive3} for all $n\in[N]$, which yields $0=-\mu_{n,L}$. We then consider the coefficient with respect to $u^*_{n,L}$ in \eqref{eq:kktdnn4-3}  for all $n\in[N]$, which implies $0\in\mu_{n,L}[\alpha,1]+ \xi_{n,L}$. Hence, we have $\xi_{n,L}= 0$ for all $n\in[N]$. Substituting $\xi_{n,L}=\mu_{n,L}= 0$ for all $n\in[N]$ into \eqref{eq:derive3}, we obtain that
		$$
		0 = -\mu+\nabla_v \left( \sum_{n=1}^{N}\sum_{\ell=1}^{L-1}\xi_{n,\ell}\zz W^*_{\ell}v^*_{n,\ell-1}\right).
		$$
		Then, we can derive $\xi_{n,\ell}=\mu_{n,\ell}= 0$ for all $n\in[N]$ and $\ell\in[L]$ by mathematical induction.
		This completes the proof.
	\end{proof}
	
	We end this section with a lemma 
	illustrating
	the MPCC C-stationary point of problem \eqref{eq:dnn} are necessary to be a local minimizer of  problem \eqref{eq:dnn}.
	
	\begin{lemma}\label{lem:kktwrela}
		If $(w^*,b^*,v^*,u^*)$ is a local minimizer of problem \eqref{eq:dnn}, then $(w^*,b^*,v^*,u^*)$ is also an  MPCC C-stationary point of problem \eqref{eq:dnn}.
	\end{lemma}
	\begin{proof}
		Since $(w^*,b^*,v^*,u^*)$ is a local minimizer of problem \eqref{eq:dnn}, Lemma \ref{lem:nnamcq} yields that
		there exist vectors $\mu\in\R^{m}$ and $\xi\in\R^{m}$ satisfying
		\begin{equation}\label{eq:lstadnn4-2}
			\begin{aligned}
				&0\in  \partial \left( \mu\zz (\sigma(u^*)-v^*)\right)+\nabla_{(v,u)}\left(\bar{\mathcal{O}}(w^*,v^*	)+\xi\zz (u^*- \Psi(v^*)w^*)\right), \\
				&0=\nabla_w\bar{\mathcal{O}}(w^*,v^*	)+\Psi(v^*)\zz \xi, 0=A\zz\xi.
			\end{aligned}
		\end{equation}
		
		Notice that $\mu\zz\sigma(u^*)$ is the only nonsmooth term in \eqref{eq:lstadnn4-2}.
		We now analyze the following three cases for all $i\in[m]$.
		
		Case (i): if $u^*_i>0$, we have $v^*_i=u^*_i$ and $\partial \mu_i \sigma(u_i^*)=\{\mu_i\}$. Let ${\mu}^1_i=\mu_i$ and ${\mu}^2_i=0$, then ${\mu}^1_i(v_i^*-u^*_i)=0$ and ${\mu}^2_i(v_i^*-\alpha u^*_i)=0$.
		
		Case (ii): if $u^*_i<0$, we have $v^*_i=\alpha u^*_i$ and $\partial \mu_i \sigma(u_i^*)=\{\alpha\mu_i\} $. Let ${\mu}^1_i=0$ and ${\mu}^2_i=\mu_i$, then ${\mu}^1_i(v_i^*-u^*_i)=0$ and ${\mu}^2_i(v_i^*-\alpha u^*_i)=0$.
		
		Case (iii): if $u^*_i=0$, we have $v^*_i=0$ and
		$$\partial \mu_i \sigma(u_i^*)\subset [\alpha,1]\mu_i=
		\left\{\mu_i^1+\alpha \mu_i^2: \mu^1_i=t_i\mu_i, \mu^2_i=(1-t_i)\mu_i, t_i\in [0,1] \right\}.$$
		In this case, we have ${\mu}^1_i(v_i^*-u^*_i)=0$, ${\mu}^2_i(v_i^*-\alpha u^*_i)=0$ and $\mu^1_i \mu^2_i\ge 0$.
		
		Combining the above three cases, it holds that
		\begin{equation}
			\begin{aligned}
				\partial \mu\zz \sigma(u^*)\subset \bigg\{&\mu^1+\alpha \mu^2:  \mu^1=t \circ \mu, \mu^2=(e_m-t) \circ \mu,\\& \left({\mu}^1\right)\zz(v^*-u^*)=0, \left({\mu}^2\right)\zz(v^*-\alpha u^*)=0, t\in \R^m_{+}, t\le e_m\bigg\}.
			\end{aligned}
		\end{equation}
		
		Together with {the inclusion} $0\in\partial\mu\zz \sigma(u^*)+\xi$, we have
		\begin{equation*}
			\begin{aligned}
				0\in \bigg\{&\mu^1+\alpha \mu^2+\xi:   \mu^1=t \circ \mu, \mu^2=(e_m-t) \circ \mu,\\& \left({\mu}^1\right)\zz(v^*-u^*)=0, \left({\mu}^2\right)\zz(v^*-\alpha u^*)=0, t\in \R^m_{+}, t\le e_m\bigg\}.
			\end{aligned}
		\end{equation*}
		Hence there exist $\bar{\mu}^1$ and $\bar{\mu}^2$ such that
		\begin{align}
			&0= \bar{\mu}^1+\alpha \bar{\mu}^2+\xi,  
			\left(\bar{\mu}^1\right)\zz(v^*-u^*)=0, \left(\bar{\mu}^2\right)\zz(v^*-\alpha u^*)=0,\label{eq:kktwrela-1}\\
			&\bar{\mu}^1=t \circ \mu, \bar{\mu}^2=(e_m-t) \circ \mu, \bar{\mu}^1 \circ \bar{\mu}^2\ge 0, t\in \R^m_{+}, t\le e_m.\label{eq:kktwrela-2}
		\end{align}	
		Combining \eqref{eq:lstadnn4-2}, \eqref{eq:kktwrela-1}, \eqref{eq:kktwrela-2} and $\sigma(u^*)-v^*=0$, we obtain \eqref{eq:mpcckktdnn-1}--\eqref{eq:mpcckktdnn-5}, and $\bar{\mu}^1 \circ \bar{\mu}^2 \ge 0$ with $\bar{\mu}^1, \bar{\mu}^2$ instead of $\mu^1,\mu^2$, respectively. This completes the proof.
	\end{proof}
	

	\section{Model Analysis}\label{sec:modelana}

	In this section, we aim to 
	theoretically investigate the relationship between problems \eqref{eq:dnn} and \eqref{eq:dnn33}.
	
	\subsection{The Existence and Boundedness of the Solution Set}\label{sec:modelana1}
	In this subsection, we show that the solution set of problem \eqref{eq:dnn33} is not empty and bounded.
	%
	First, we define a level set $\Omega_{\theta}$ of the objective function of problem \eqref{eq:dnn33} by
	\begin{equation}\label{eq:theta}
		\Omega_{\theta}=\left\{(w,b,v,u) \in \Omega_2: \mathcal{O}(w,v,u) \leq \theta\right\} \text{ with }\theta> \frac{1}{N}\|Y\|_F^2,
	\end{equation}
	where $Y=(y_1,y_2,\ldots,y_N)$ is the label matrix.
	Clearly $0\in\Omega_{\theta}$.
	For all $(w,b,v,u)\in\Omega_\theta$, it holds that $\mathcal{R}_1(w)=\lambda_w \sum_{\ell=1}^L \|W_{\ell}\|_{2,1}\leq \theta$ and $\mathcal{R}_2(v)=\lambda_v\|v\|^2\leq \theta$, which further implies that $\|w\|$ and $\|v\|$ are bounded.
	For brevity, we let
	$\theta_w:=\frac{\theta}{\lb_w}\sqrt{\overline{N}+N_0}$, $\theta_v:=\sqrt{\frac{\theta}{\lb_v}}$ be the upper bounds of $\|w\|$ and $\|v\|$ over the set $\Omega_{\theta}$, respectively.
	
	\begin{theorem}\label{thm:nonem}
		The set $\Omega_\theta$ is bounded.
		Furthermore, the solution set of problem \eqref{eq:dnn33} is not empty and bounded.
	\end{theorem}
	
	\begin{proof}
		Since $(w,b,v,u) \in \Omega_2$, it holds that $
		b_{\ell}\leq v_{n,\ell}- W_{\ell}v_{n,\ell-1}$
		for all $\ell\in[L]$ and $n\in[N]$.
		Together with {the fact that} $\theta_w$, $\theta_v$ being the upper bounds of $\|w\|$ and $\|v\|$ over the set $\Omega_{\theta}$, respectively, we obtain that $\|b_{+}\|_{\infty}$ is bounded.
		
		Since $\mathcal{O}(w,v,u)\leq \theta$,
		$\mathcal{R}_1(w)$, $\mathcal{R}_2(v)$ are nonnegative, then for all $n\in[N]$, $\ell\in[L]$, $j\in[N_{\ell}]$, it holds that
		$$
		(v_{n,\ell})_j- \sigma({b}_{\ell,j}+W_{\ell,j}v_{n,\ell-1}) \leq \frac{\theta}{\beta_{\ell}},
		$$
		which further implies that either
		$
		(v_{n,\ell})_j-\frac{\theta}{\beta_{\ell}}\leq {b}_{\ell,j}+W_{\ell,j}v_{n,\ell-1}
		$
		or
		$
		(v_{n,\ell})_j-\frac{\theta}{\beta_{\ell}}\leq \alpha({b}_{\ell,j}+W_{\ell,j}v_{n,\ell-1}).
		$
		By $(w,b,v,u)\in\Omega_{\theta}$, the definition of $\theta_w$ and $\theta_v$, we have $\frac{1}{\alpha}((v_{n,\ell})_j-\frac{\theta}{\beta_{\ell}})-W_{\ell,j}v_{n,\ell-1}> -\frac{1}{\alpha}(\theta_v+\frac{\theta}{\beta_{\ell}})-\overline{N}\theta_v\theta_w$ and $(v_{n,\ell})_j-\frac{\theta}{\beta_{\ell}}-W_{\ell,j}v_{n,\ell-1}> -\theta_v-\frac{\theta}{\beta_{\ell}}-\overline{N}\theta_v\theta_w$ for all $\ell\in[L]$ and $n\in[N]$. Since $0<\alpha<1$, it holds that
		\begin{equation}\label{eq:relabb}
			\begin{aligned}
				{b}_{\ell,j}&\geq \min\left\{\frac{1}{\alpha}\left((v_{n,\ell})_j-\frac{\theta}{\beta_{\ell}}\right)-W_{\ell,j}v_{n,\ell-1},(v_{n,\ell})_j-\frac{\theta}{\beta_{\ell}}-W_{\ell,j}v_{n,\ell-1}\right\}\\&> -\frac{1}{\alpha}\left(\theta_v+\frac{\theta}{\beta_{\ell}}\right)-\overline{N}\theta_v\theta_w.
			\end{aligned}
		\end{equation}
		
		It follows from the boundedness of $\|b_{+}\|_{\infty}$ that $\|b\|$ is also bounded. Hence $\|u\|=\|\Psi(v)w+Ab\|$ is bounded, too. These facts imply that  $\Omega_{\theta}$ is a bounded set. Together with {the inclusion} $0\in\Omega_{\theta}$ and the continuity of the objective function of problem \eqref{eq:dnn33},  we obtain that  the solution set of problem \eqref{eq:dnn33} is nonempty and bounded.
	\end{proof}
	
	
	\subsection{Exact Penalization}
	
	In this subsection, we consider problem \eqref{eq:dnn33} with penalty parameter $\beta$ satisfying
	\begin{equation}\label{eq:betacondi}
			\beta_{\ell}>LL_{\bar{\mathcal{O}}} \max\{\theta_w,1\}^{L}+2\sum_{j=\ell+1}^L\beta_j \theta_w\max\{\theta_w,1\}^{j-\ell-1} \text{ for all } \ell\in[L],
	\end{equation}
	and reveal the relationship between problems \eqref{eq:dnn33} and \eqref{eq:dnn}, where $L_{\bar{\mathcal{O}}}$ is the Lipschitz constant of the function $\bar{\mathcal{O}}$ over $\Omega_\theta$. We first present a lemma, which shows that any {limiting stationary} point of problem \eqref{eq:dnn33} is in the feasible set $\Omega_1$ of problem \eqref{eq:dnn}.
	
	\begin{lemma}\label{lem:lsta}
		Let the penalty parameter $\beta$ satisfy \eqref{eq:betacondi}. If $(w^*,b^*,$ $v^*,u^*)\in \{(w,b,v,u) \in \Omega_2: $$\mathcal{O}(w,v,u) < \theta\}$ is {a} {limiting stationary} point of problem~\eqref{eq:dnn33},
		then \((w^*,b^*,v^*,u^*)\) is in the feasible set $\Omega_1$ of problem \eqref{eq:dnn}.
	\end{lemma}
	\begin{proof}
		If 
		\((w^*,b^*,v^*,u^*)\) is a limiting stationary point of problem~\eqref{eq:dnn33}, then \((w^*,b^*,v^*,u^*)\) is a C-stationary point of problem~\eqref{eq:dnn33}, which implies that
		\begin{equation}
			\label{eq:oclar}
			\mathcal{O}^{\circ}(w^*,v^*,u^*;d_w,d_v,d_u) \geq 0,\,\,\text{for all }(d_w,d_v,d_u,d_b)\in\mathcal{T}_{\Omega_2}(w^*,b^*,v^*,u^*).
		\end{equation}
		
		We then prove $v^*_{n,\ell}=\sigma(u^*_{n,\ell})$ for all $n\in[N]$ and $\ell=L,L-1,\ldots,1$ by mathematical induction.
		
		Assume on contradiction that $(w^*,b^*,v^*,u^*) \notin \Omega_{1}$. When $\ell$ equals $L$, let $\mathcal{I}_{L}:=\{n: v^*_{n,L}\geq\sigma(u^*_{n,L}), v^*_{n,L}\neq\sigma(u^*_{n,L})\}$. Without loss of generality, we assume that $\mathcal{I}_{L}$ is not an empty set. We set
		$$d_v=((d_v)_{1,1}\zz,(d_v)_{2,1}\zz,\ldots,(d_v)_{1,L}\zz,(d_v)_{2,L}\zz,\ldots,(d_v)_{N,L}\zz)\zz$$ for all $n\in[N]$, $\ell\in[L]$, $(d_v)_{n,\ell}\in\R^{N_{\ell}}$ and
		\begin{equation*}
			(d_v)_{n, \ell}=\left\{\begin{array}{lr}
				0, &\text { if }\ell<L \text { or } n\not\in\mathcal{I}_L,\hspace{0.35cm} \\
				\sigma(u_{n,\ell}^*)- v^*_{n,\ell}, &\text { if }\,\ell=L \text { and } n\in\mathcal{I}_L. \\
			\end{array}\right.
		\end{equation*}Clearly $d_v\leq 0$ and $(w^*,b^*,v^*+td_v,u^*) \in \Omega_2$ for all $0\leq t<1$. Hence $(0,0,d_v,0)\in\mathcal{T}_{\Omega_2}(w^*,b^*,v^*,u^*)$.
		
		Since $\mathcal{O}(w^*,v^*,u^*)<\theta$ and $\mathcal{O}$ is locally Lipschitz continuous,  there exists $\bar{\epsilon} \in(0,1]$ such that $\mathcal{O}(w,v,u)<\theta$ for all $(w,v,u)\in\mathcal{B}_{\bar{\epsilon}}(w^*,v^*,u^*)$. Furthermore, for any $(w,v,u)\in\mathcal{B}_{\bar{\epsilon}}(w^*,v^*,u^*)$, there exists $\bar{t}\in(0,1]$ such that $\mathcal{O}(w,v+td_v,u)<\theta$ for all $0<t<\bar{t}$.
		
		Together with the inequalities $d_v\leq 0$ and 
		$\bar{\mathcal{O}}(w,v+td_v)<\mathcal{O}(w,v+td_v,u)<\theta$,  it holds that
		\begin{equation}\label{eq:vuconstru2}
			\begin{aligned}
				&\frac{1}{t}\left(\mathcal{O}(w,v+t d_v,u)-\mathcal{O}(w,v,u)\right)\\
				=&\frac{1}{t}\left(\bar{\mathcal{O}}(w,v+td_v)+\beta \zz (v+td_v-\sigma(u))
				-\bar{\mathcal{O}}(w,v)-\beta \zz (v-\sigma(u))\right)\\
				=&\frac{1}{t}\left(\bar{\mathcal{O}}(w,v+td_v)
				-\bar{\mathcal{O}}(w,v)\right)+\beta \zz d_v
				\leq (L_{\bar{\mathcal{O}}}-\beta_L) \sum_{n\in\mathcal{I}_{L}}\left\|v^*_{n,L}-	\sigma(u_{n,L}^*)\right\|_1.
			\end{aligned}
		\end{equation}
		Hence we derive $$\mathcal{O}^{\circ}(w^*,v^*,u^*;0,d_v,0) \leq\limsup _{(w,v,u) \rightarrow (w^*,v^*,u^*), t \downarrow 0} (L_{\bar{\mathcal{O}}}-\beta_L) \sum_{n\in\mathcal{I}_{L}}\left\|v^*_{n,L}-	\sigma(u_{n,L}^*)\right\|_1<0.$$
		
		This leads to a contradiction.  Hence, it holds that $v^*_{n,L}=\sigma(u^*_{n,L})$ for all $n\in[N]$.
		
		Then, we suppose that $v^*_{n,\ell}=\sigma(u^*_{n,\ell})$ for all $n\in[N]$ and $\ell = L,L-1,\ldots,\bar{\ell}+1$. Let $\mathcal{I}_{\bar{\ell}}:=\{n: v^*_{{n},\bar{\ell}}\geq\sigma(u^*_{{n},\bar{\ell}}), v^*_{{n},\bar{\ell}}\neq\sigma(u^*_{{n},\bar{\ell}})\}$. Without loss of generality, we suppose that $\mathcal{I}_{\bar{\ell}}$ is not an empty set. Then for all $n\in[N]$ and $\ell\in[L]$, we set
		\begin{equation}\label{eq:vuconstrut}
			\left\{\begin{aligned}
				&\tilde{v}^{\epsilon}_{n, \ell}=v^*_{n,\ell},\,\,\tilde{u}^{\epsilon}_{n, \ell}=u^*_{n,\ell}, \hspace{2.9cm} \text{ if } \ell<\bar{\ell} \text { or } \ell=\bar{\ell}, n\not\in\mathcal{I}_{\bar{\ell}},\\
				&\tilde{v}^{\epsilon}_{n, \ell}=\epsilon\sigma(u_{n,\ell}^*)+(1-\epsilon) v^*_{n,\ell},\,\,\tilde{u}^{\epsilon}_{n, \ell}={u}^*_{n, \ell},   \text{ if }\ell=\bar{\ell} \text { and } n\in\mathcal{I}_{\bar{\ell}}, \\
				&\tilde{v}^{\epsilon}_{n, \ell}=\sigma(\tilde{u}^{\epsilon}_{n,\ell}), \,\,\tilde{u}^{\epsilon}_{n, \ell}=w^*_{\ell}\tilde{v}^{\epsilon}_{n, \ell-1}+b^*_{\ell},  \hspace{0.75cm}\text{ if }\ell>\bar{\ell}. \\
			\end{aligned}\right.
		\end{equation}
		Clearly $(w^*,b^*, \tilde{v}^{\epsilon}, \tilde{u}^{\epsilon}) \in \Omega_2$, $\lim_{\epsilon\downarrow 0} (\tilde{v}^{\epsilon}-v^*)/\epsilon$ and $\lim_{\epsilon\downarrow 0} (\tilde{u}^{\epsilon}-u^*)/\epsilon$ exist. Let
		\begin{equation}
			\label{eq:dudvdefine}
			d_v=\lim_{\epsilon_k \downarrow 0} d_v^{(k)},\,d_v^{(k)}=\frac{\tilde{v}^{\epsilon_k}-v^*}{\epsilon_{k}},\, d_u=\lim_{\epsilon_k \downarrow 0} d_u^{(k)},\,d_u^{(k)}=\frac{\tilde{u}^{\epsilon_k}-u^*}{\epsilon_{k}},
		\end{equation}
		then we have $(0,0,d_v,d_u)\in\mathcal{T}_{\Omega_2}(w^*,b^*,v^*,u^*)$. Besides, it follows from \eqref{eq:dudvdefine}, the Lipschitz continuity and directional differentiability of $\mathcal{O}$  that
		\begin{align}
			&\limsup _{(w,v,u) \rightarrow (w^*,v^*,u^*), t \downarrow 0} \frac{1}{t}\left(\mathcal{O}(w,v+t d_v,u+td_u)-\mathcal{O}(w,v,u)\right)\notag\\
			=&\limsup _{(w,v,u) \rightarrow (w^*,v^*,u^*), \epsilon_k \downarrow 0} \frac{1}{\epsilon_k}\left(\mathcal{O}\left(w,v+\epsilon_k d_v,u+\epsilon_k d_u\right)-\mathcal{O}(w,v,u)\right)\notag\\
			=&\limsup _{(w,v,u) \rightarrow (w^*,v^*,u^*), \epsilon_k \downarrow 0} \frac{1}{\epsilon_k}\left(\mathcal{O}\left(w,v+\epsilon_k d^{(k)}_v,u+\epsilon_k d^{(k)}_u\right)-\mathcal{O}(w,v,u)\right)\label{eq:cstadudv}\\
			&\hspace{3.5cm}-\frac{1}{\epsilon_k}\left(\mathcal{O}\left(w,v+\epsilon_k d^{(k)}_v,u+\epsilon_k d^{(k)}_u\right)-\mathcal{O}(w,v+\epsilon_k d_v,u+\epsilon_k d_u)\right)\notag\\
			=&\limsup _{(w,v,u) \rightarrow (w^*,v^*,u^*), \epsilon_k \downarrow 0} \frac{1}{\epsilon_k}\left(\mathcal{O}\left(w,v+\epsilon_k d^{(k)}_v,u+\epsilon_k d^{(k)}_u\right)-\mathcal{O}(w,v,u)\right).\notag
		\end{align}
		
		Since $\mathcal{O}(w^*,v^*,u^*)<\theta$ and $\mathcal{O}$ is locally Lipschitz continuous,  there exists $\bar{\epsilon} \in(0,1]$ such that $\mathcal{O}(w,v,u)<\theta$ for all $(w,v,u)\in\mathcal{B}_{\bar{\epsilon}}(w^*,v^*,u^*)$. Furthermore, for any $(w,v,u)\in\mathcal{B}_{\bar{\epsilon}}(w^*,v^*,u^*)$, there exists $\bar{t}\in(0,1]$ such that $\mathcal{O}(w,v+td_v,u+td_u)<\theta$ for all $0<t<\bar{t}$. Together with {the equalities} \eqref{eq:dudvdefine}, there exists $\tilde{\epsilon}\in(0,1]$ such that $\mathcal{O}(w,v+\epsilon d^{(k)}_v,u+\epsilon d^{(k)}_u)<\theta$ for all $k\in\mathbb{N}$ and $0<\epsilon<\tilde{\epsilon}$.  Without loss of generality, we assume that $\epsilon_{k}<\tilde{\epsilon}$ for all $k\in\mathbb{N}$.
		
		Recall the definition of $\theta_w$, $d_u$, $d_v$ and $v^*_{n,\ell}=\sigma_{\ell}(W^*_{\ell}v^*_{n,\ell-1}+b^*_{\ell})$ for all $\ell>\bar{\ell}$ and $n\in[N]$, we obtain  that $\|(d^{(k)}_v)_{n,\ell}\|_1\leq \theta_w\|(d^{(k)}_v)_{n,\ell-1}\|_1$ and $\|(d^{(k)}_u)_{n,\ell}\|_1\leq \theta_w\|(d^{(k)}_v)_{n,\ell-1}\|_1$ for all $k\in\mathbb{N}$, $\ell>\bar{\ell}$ and $n\in[N]$. Hence for all $k\in\mathbb{N}$, we have
		\begin{equation}\label{eq:vv8ineq}
			\begin{aligned}
				\max\left\{\left\|d^{(k)}_v\right\|_1,\left\|d^{(k)}_u\right\|_1\right\}&\leq L\max\left\{\theta_w,1\right\}^{L} \left\|\left(d^{(k)}_v\right)_{n,\bar{\ell}}\right\|_1
				\\&=L\max\{\theta_w,1\}^{L}\sum_{n \in \mathcal{I}_{\bar{\ell}}}\left\|v_{n, \bar{\ell}}^{*}-\sigma\left(u_{n, \bar{\ell}}^{*}\right)\right\|_{1},
			\end{aligned}
		\end{equation}
		where the last equality comes from the definition \eqref{eq:dudvdefine}.
		We also obtain that
			\begin{equation}\label{eq:contradic2}
				\begin{aligned}
					&  \sum_{n=1}^N\sum_{\ell=1}^L \beta_{\ell} e_{N_{\ell}}\zz \left(d^{(k)}_v\right)_{n,\ell}- \frac{1}{\epsilon_k}\sum_{n=1}^N\sum_{\ell=1}^L \beta_{\ell}e_{N_{\ell}}\zz \left(\sigma\left(u_{n,\ell}+\epsilon_k \left(d^{(k)}_u\right)_{n,\ell}\right)-\sigma(u_{n,\ell})\right)
					\\=
					&-\beta_{\bar{\ell}}\sum_{n \in \mathcal{I}_{\bar{\ell}}}\left\|v_{n, \bar{\ell}}^{*}-\sigma\left(u_{n, \bar{\ell}}^{*}\right)\right\|_{1}+\sum_{\ell=\bar{\ell}+1}^L \sum_{n=1}^N \beta_{\ell}\left\|\left(d^{(k)}_v\right)_{n,\ell}\right\|_1\\&\quad- \frac{1}{\epsilon_k}\sum_{n=1}^N\sum_{\ell=\bar{\ell}+1}^L \beta_{\ell}e_{N_{\ell}}\zz \left(\sigma\left(u_{n,\ell}+\epsilon_k \left(d^{(k)}_u\right)_{n,\ell}\right)-\sigma(u_{n,\ell})\right)
					\\
					\leq & -\beta_{\bar{\ell}}\sum_{n \in \mathcal{I}_{\bar{\ell}}}\left\|v_{n, \bar{\ell}}^{*}-\sigma\left(u_{n, \bar{\ell}}^{*}\right)\right\|_{1}+\sum_{\ell=\bar{\ell}+1}^L \sum_{n=1}^N \beta_{\ell}\left(\left\|\left(d^{(k)}_v\right)_{n,\ell}\right\|_1+\left\|\left(d^{(k)}_u\right)_{n,\ell}\right\|_1\right)
					\\\leq &\left(2\sum_{j=\bar{\ell}+1}^L\beta_j \theta_w\max\{\theta_w,1\}^{j-\bar{\ell}-1}-\beta_{\bar{\ell}}\right)\sum_{n \in \mathcal{I}_{\bar{\ell}}}\left\|v_{n, \bar{\ell}}^{*}-\sigma\left(u_{n, \bar{\ell}}^{*}\right)\right\|_{1},
				\end{aligned}
			\end{equation}
			where the first equality comes from $-e_{N_{\bar{\ell}}}\zz (d^{(k)}_v)_{n,\bar{\ell}}
			=\sum_{n \in \mathcal{I}_{\bar{\ell}}}\|v_{n, \bar{\ell}}^{*}-\sigma(u_{n, \bar{\ell}}^{*})\|_{1}$, $(d^{(k)}_u)_{n,\bar{\ell}}=0$ and $(d^{(k)}_v)_{n,{\ell}}=(d^{(k)}_u)_{n,{\ell}}=0$ for all $1\leq \ell<\bar{\ell}$ and $n\in[N]$ by definition \eqref{eq:dudvdefine}, and the last inequality yields from $\|(d^{(k)}_v)_{n,\ell}\|_1\leq \theta_w\|(d^{(k)}_v)_{n,\ell-1}\|_1$ and $\|(d^{(k)}_u)_{n,\ell}\|_1\leq \theta_w\|(d^{(k)}_v)_{n,\ell-1}\|_1$ for all $k\in\mathbb{N}$, $\ell>\bar{\ell}$ and $n\in[N]$.
			
			It then holds that
					\begin{align*}
						& \frac{1}{\epsilon_k}\left(\mathcal{O}\left(w,v+\epsilon_k d^{(k)}_v,u+\epsilon_k d^{(k)}_u\right)-\mathcal{O}(w,v,u)\right)
						\\=&\frac{1}{\epsilon_k}\left(\bar{\mathcal{O}}\left(w,v+\epsilon_k d^{(k)}_v\right)
						-\bar{\mathcal{O}}(w,v)\right)+\sum_{n=1}^N\sum_{\ell=1}^L \beta_{\ell} e_{N_{\ell}}\zz \left(d^{(k)}_v\right)_{n,\ell}\\&\quad- \frac{1}{\epsilon_k}\sum_{n=1}^N\sum_{\ell=1}^L \beta_{\ell}e_{N_{\ell}}\zz \left(\sigma\left(u_{n,\ell}+\epsilon_k \left(d^{(k)}_u\right)_{n,\ell}\right)-\sigma(u_{n,\ell})\right)\\
						\\
						\leq & \left(LL_{\bar{\mathcal{O}}}\max\{\theta_w,1\}^L+2\sum_{j=\bar{\ell}+1}^L\beta_j \theta_w\max\{\theta_w,1\}^{j-\bar{\ell}-1}-\beta_{\bar{\ell}}\right)\sum_{n \in \mathcal{I}_{\bar{\ell}}}\left\|v_{n, \bar{\ell}}^{*}-\sigma\left(u_{n, \bar{\ell}}^{*}\right)\right\|_{1}<0,
					\end{align*}
				where the equality comes from the definitions \eqref{eq:dudvdefine}, 
				and the first inequality yields from {the inequalities} \eqref{eq:vv8ineq}, \eqref{eq:contradic2} 
				and $\bar{\mathcal{O}}(w,v)<\mathcal{O}(w,v,u)<\theta$.

		Together with the relationships \eqref{eq:oclar} and \eqref{eq:cstadudv}, there is a contradiction.
		We then conclude that $v^*_{n,\bar{\ell}}=\sigma(u^*_{n,\bar{\ell}})$ for all $n\in[N]$. The proof is completed by mathematical induction.
	\end{proof}
	
	We next present the main theorem illustrating the fact that problems \eqref{eq:dnn33} and \eqref{eq:dnn} sharing the same global and local minimizers.
	
	\begin{theorem}\label{thm:exactglo}
Let the penalty parameter $\beta$ satisfy \eqref{eq:betacondi}.		 Then the following statements hold.
		
		(a) \((w^*,b^*,v^*,u^*)\) is a global minimizer of problem~\eqref{eq:dnn33}
		if and only if \((w^*,b^*,v^*,u^*)\) is a global minimizer of problem \eqref{eq:dnn}.
		
		
		(b) If \((w^*,b^*,v^*,u^*)\in\left\{(w,b,v,u) \in \Omega_2: \mathcal{O}(w,v,u) < \theta\right\}\) is a local minimizer of problem~\eqref{eq:dnn33}, then
		\((w^*,b^*,v^*,u^*)\) is also  a local minimizer of problem \eqref{eq:dnn}.
		
		(c) \((w^*,b^*,v^*,u^*)\in\left\{(w,b,v,u) \in \Omega_1: \mathcal{O}(w,v,u) < \theta\right\}\) is a local minimizer of problem~\eqref{eq:dnn} if and only if
		\((w^*,b^*,v^*,u^*)\) is  a local minimizer of problem \eqref{eq:dnn33}.
	\end{theorem}
	
	\begin{proof}
		(a) If $(w^*,b^*,v^*,u^*)$ is a global minimizer of problem \eqref{eq:dnn33}, we obtain that $\bar{\mathcal{O}}(w^*,v^*)<\theta$ by $\bar{\mathcal{O}}(0,0)=\frac{1}{N}\|Y\|_F^2<\theta$. Lemma \ref{lem:mfcqsys} yields that $(w^*,b^*,v^*,u^*)$ is a limiting stationary point of problem \eqref{eq:dnn33}. From Lemma \ref{lem:lsta}, we have $(w^*,b^*,v^*,u^*)\in\Omega_1$. 
		Together with {the inclusion} $\Omega_1\subset \Omega_2$, we know that $(w^*,b^*,v^*,u^*)\in\Omega_1$ must be a global minimizer of problem \eqref{eq:dnn}.
		
		Conversely, suppose that $(\bar{w},\bar{b},\bar{v},\bar{u})$ is a global minimizer of problem \eqref{eq:dnn33}. 
		From what we have proved, it holds that $\bar{\mathcal{O}}(\bar{w},\bar{v})<\theta$, $(\bar{w},\bar{b},\bar{v},\bar{u})\in\Omega_1$ and
		\begin{equation}\label{eq:globalexact1}
			\min_{(w,b,v,u)\in\Omega_2} {\mathcal{O}}(w,v,u)={\mathcal{O}}(\bar{w},\bar{v},\bar{u})=\bar{\mathcal{O}}(\bar{w},\bar{v})=\min_{(w,b,v,u)\in\Omega_1} \bar{\mathcal{O}}(w,v).
		\end{equation}
		
		Since $(w^*,b^*,v^*,u^*)\in\Omega_1$ is a global minimizer of problem \eqref{eq:dnn}, we have
		\begin{equation}\label{eq:globalexact2}
			\min_{(w,b,v,u)\in\Omega_1} \bar{\mathcal{O}}(w,v)=\bar{\mathcal{O}}(w^*,v^*)=\mathcal{O}(w^*,v^*,u^*).
		\end{equation}
		Together with {the facts} \eqref{eq:globalexact1} and $(w^*,b^*,v^*,u^*)\in\Omega_1\subset\Omega_2$, we have $(w^*,b^*,v^*,u^*)$ is a global minimizer of problem \eqref{eq:dnn33}.
		
		(b) 
		Lemma \ref{lem:mfcqsys} yields that $(w^*,b^*,v^*,u^*)$ is a limiting stationary point of problem \eqref{eq:dnn33}. From Lemma \ref{lem:lsta}, it holds that $(w^*,b^*,v^*,u^*)\in\Omega_1$. 
		Together with {the inclusion} $\Omega_1\subset \Omega_2$, $(w^*,b^*,v^*,u^*)\in\Omega_1$ must be a local minimizer of problem \eqref{eq:dnn}.
		
		(c) By using a similar method as that in the proof of (a) and (b), we complete the statement (c).
	\end{proof}

	
	
	Since $\mathcal{R}_1$ and $\mathcal{R}_2$ are directionally differentiable, 
	we can show that if \((\bar{w},\bar{b},\bar{v},\bar{u})\in \{(w,b,v,u)\in\Omega_2:\mathcal{O}(w,b,v,u)<\theta\}\) is a d-stationary point of problem~\eqref{eq:dnn33},
	then \((\bar{w},\bar{b},\bar{v},\bar{u})\) is also a d-stationary point of problem~\eqref{eq:dnn} by using Lemma \ref{lem:lsta} and ideas from \cite[Theorem 2.1]{cui2020multicomposite} and \cite[Theorem 2.5]{liu2021auto}. However, computing a d-stationary point is difficult, we will consider a limiting  stationary point of problem \eqref{eq:dnn33}, and show that it is an MPCC W-stationary point of problem \eqref{eq:dnn}.

	\begin{theorem}\label{thm:lsta2}
		Let the penalty parameter $\beta$ satisfy \eqref{eq:betacondi}.
		If \((w^*,b^*,v^*,u^*)\in \{(w,b,v,u) \in \Omega_2: \mathcal{O}(w,v,u) < \theta\}\) is a limiting stationary point of problem~\eqref{eq:dnn33},
		then \((w^*,b^*,v^*,u^*)\) is an MPCC W-stationary point of problem~\eqref{eq:dnn}.
	\end{theorem}
	
	\begin{proof}
		Since it holds that $\mathcal{O}(w^*,v^*,u^*) < \theta$, Lemma \ref{lem:lsta} yields that \((w^*,b^*,v^*,u^*)\in\Omega_1\).
		
		Form Theorem \ref{thm:cone1} and \((w^*,b^*,v^*,u^*)\) being a limiting stationary point of problem~\eqref{eq:dnn33}, there exist vectors $\mu\in\R_+^{2m}$ and $\xi\in\R^{m}$ such that \eqref{eq:kktdnn2-1}-- \eqref{eq:kktdnn2-4} hold.
		
		Let $\mu=[\left(\mu^1\right)\zz,\left(\mu^2\right)\zz]\zz$ with $\mu^1,\mu^2\in\R_+^{m}$. Recall the definition of $\mathcal{C}$, it holds that
		\begin{equation}\label{eq:lsta2-eq3}
			\,\nabla_v \mu\zz \mathcal{C}( v^*,u^*)= -\left({\mu}^1+{\mu}^2\right), \,\nabla_u \mu\zz \mathcal{C}( v^*,u^*)={\mu}^1+\alpha{\mu}^2.
		\end{equation}
		
		Then, we obtain from \eqref{eq:kktdnn2-2}--\eqref{eq:kktdnn2-4} that
		\begin{equation}\label{eq:lsta2-eq0}
			\begin{aligned}
				&0= \nabla_v \bar{\mathcal{O}}(w^*,v^*)+\beta 	-\left({\mu}^1+{\mu}^2\right)+\nabla_v \xi\zz (u^*- \Psi(v^*)w^*),\\
				&0\in \partial_u (-\beta \zz \sigma(u^*)) +{\mu}^1+\alpha{\mu}^2+ \xi,\\
				&\left({\mu}^1\right)\zz (u^*-v^*)=0,\, \left({\mu}^2\right)\zz (\alpha u^*-v^*)=0,\, 		{\mu}^1\geq 0,\, {\mu}^2\geq 0.
			\end{aligned}
		\end{equation}
		
		Now, we prove that there exist $\bar{\mu}^1$ and $\bar{\mu}^2$ such that
		\begin{equation}\label{eq:mu1mu2}
			\begin{aligned}
				&0= \bar{\mu}^1+\alpha\bar{\mu}^2+ \xi, \,\bar{\mu}^1+\bar{\mu}^2={\mu}^1+{\mu}^2-\beta ,\,\left(\bar{\mu}^{1}\right)^{\top}\left(v^{*}-u^{*}\right)=0,\left(\bar{\mu}^{2}\right)^{\top}\left(v^{*}-\alpha u^{*}\right)=0
			\end{aligned}
		\end{equation}
		by analyzing the following cases for all $i\in[m]$.
		
		Case (i): if $u_i^*< 0$, the relation \eqref{eq:lsta2-eq0} together with the definition of $\sigma$ yield $v_i^*=\alpha u_i^*$, $\mu_i^1=0$ and
		$$
		\partial (-(\beta)_i \sigma(u_i^*))+{\mu}_i^1+\alpha{\mu}_i^2=\left\{{\mu}_i^1+\alpha \left({\mu}_i^2-(\beta)_i\right)\right\},
		$$
		where $(\beta)_i$ denotes the $i$-th element of $\beta$.
		In this case, we have ${\mu}_i^1 (v_i^*- u_i^*)=0$ and $({\mu}_i^2-(\beta)_i) (v_i^*- \alpha u_i^*)=0$. Let $\bar{\mu}_i^1=\mu_i^1=0, \bar{\mu}_i^2=\mu_i^2-(\beta)_i$.
		
		Case (ii): if $u_i^*> 0$, the relation \eqref{eq:lsta2-eq0} together with the definition of $\sigma$ yield $v_i^*= u_i^*$, $\mu_i^2=0$ and
		$$
		\partial (-(\beta)_i \sigma(u_i^*))+{\mu}_i^1+\alpha{\mu}_i^2=\left\{\left({\mu}_i^1-(\beta)_i\right)  +\alpha{\mu}_i^2\right\}.
		$$
		In this case, we have (${\mu}_i^1-(\beta)_i) (v_i^*- u_i^*)=0$ and ${\mu}_i^2 (v_i^*- \alpha u_i^*)=0$. Let $\bar{\mu}_i^1=\mu_i^1-(\beta)_i, \bar{\mu}_i^2=\mu_i^2=0$.
		
		Case (iii): if $u_i^*= 0$, we have $v_i^*=\sigma(u_i^*)=0$ and
		\begin{equation}\label{eq:lsta2-eq2}
			\begin{aligned}
				&\partial (-(\beta)_i \sigma(u_i^*))+{\mu}_i^1+\alpha{\mu}_i^2=\left\{ \left({\mu}_i^1-(\beta)_i\right)+\alpha {\mu}_i^2, {\mu}_i^1+\alpha \left({\mu}_i^2-(\beta)_i\right)\right\}.
			\end{aligned}
		\end{equation}
		In this case, we have $v_i^*- u_i^*=0$, $ v_i^*- \alpha u_i^*=0$. It then holds that either $0=({\mu}_i^1-(\beta)_i)+\alpha {\mu}_i^2+\xi$ or $0={\mu}_i^1+\alpha ({\mu}_i^2-(\beta)_i)+\xi$ by \eqref{eq:lsta2-eq0}. If $0={\mu}_i^1+\alpha ({\mu}_i^2-(\beta)_i)+\xi$, let $\bar{\mu}_i^1=\mu_i^1\geq 0, \bar{\mu}_i^2=\mu_i^2-(\beta)_i$. Otherwise, let
		$\bar{\mu}_i^2=\mu_i^2 \geq 0, \bar{\mu}_i^1=\mu_i^1-(\beta)_i$.
		
		Combining the above three cases, we obtain \eqref{eq:mu1mu2}. {Recall} the relation \eqref{eq:lsta2-eq0}, we then derive \eqref{eq:mpcckktdnn-1}-- \eqref{eq:mpcckktdnn-5} with $\bar{\mu}^1$, $\bar{\mu}^2$ instead of $\mu^1$ and $\mu^2$, respectively. The proof is then completed.
	\end{proof}
	
	\begin{remark}
		Notice that the conditions \eqref{eq:betacondi} for the penalty parameter $\beta$
		are recursively define from $L$ to $1$. This coincides the intuition that the inner layer should have larger penalty parameter than the outer one to avoid error accumulation.
	\end{remark}
	
	\begin{remark}
		By simple calculation, we have
		\begin{equation*}
			\bar{\mu}_i^1\bar{\mu}_i^2=\left\{
			\begin{aligned}
				{\mu}^1_i \left({\mu}^2_i-(\beta)_i\right) \text{ or } {\mu}^2_i \left({\mu}^1_i-(\beta)_i\right), &\text{ if }u_i^*=0,\\
				0,\hspace{1.7cm} &\text{ if }u_i^*\neq 0.
			\end{aligned}\right.
		\end{equation*}
		If $\mu^1_i=\mu^2_i=0$ for all $i\in\{i:u_i^*=0\}$, then \((w^*,b^*,v^*,u^*)\) is an MPCC C-stationary point of problem~\eqref{eq:dnn}.
	\end{remark}
	
	\begin{remark}\label{rem:leReLU}
		Our theoretical results can also be extended to ReLU network with $\alpha=0$. Notice that the solution set of problem \eqref{eq:dnn33} with $\alpha=0$ is unbounded (see a counterexample given by \cite{liu2021auto}), we introduce a constrained set $\Omega_b$, and minimize the objective function of problem \eqref{eq:dnn33} over $\Omega_2 \cap \{(w,b,v,u):b\in\Omega_b\}$, where
		$$\Omega_b:=\left\{b : b\ge -e_{\overline{N}}\overline{N}\theta_w\theta_v\right\}.$$ We call the resulted problem $(\mathrm{PP}_{\mathrm{b}})$.
		By using a similar method as that in the proof of Theorems \ref{thm:nonem} and \ref{thm:exactglo}, we can prove that the solution set of problem $(\mathrm{PP}_{\mathrm{b}})$ is nonempty and bounded; the global (local) minimizer of problem $(\mathrm{PP}_{\mathrm{b}})$ is a global (local) minimizer of problem \eqref{eq:dnn}. However, a limiting stationary point of problem $(\mathrm{PP}_{\mathrm{b}})$ may not be a MPCC W-stationary point of \eqref{eq:dnn} with $\alpha=0$. Specifically, if $(w^*,b^*,v^*,u^*)$ is a limiting stationary point of problem $(\mathrm{PP}_{\mathrm{b}})$, then \eqref{eq:mpcckktdnn-1}--\eqref{eq:mpcckktdnn-5} hold with $0\leq A\zz \xi$ instead of $0= A\zz \xi$.
		
		Furthermore, our numerical algorithm for solving problem \eqref{eq:dnn33}, which will be proposed in Section \ref{sec:alm}, can also be applied to solve problem $(\mathrm{PP}_{\mathrm{b}})$.
	\end{remark}
	
	We end this section by summarizing our results for the relationship of problems \eqref{eq:dnn} and \eqref{eq:dnn33} in the following diagram, where the function value at the related points is less than $\theta$.
	\begin{equation*}\label{eq:relation}
		\small
		\boxed{\begin{aligned}
				&\text{(PP)}: \,\,\text{global (local) minimizer}\, \hspace{2.4cm} \text{ limiting stationary point } \Leftrightarrow \text{ KKT point }\\
				&\hspace{2.5cm}\Downarrow\Uparrow\hspace{2.0cm}\text{\scriptsize conditions in Remark 12}\Downarrow\hspace{1.7cm}\Downarrow\\
				&\text{(P)}:  \,\quad\text{global (local) minimizer}\,  \hspace{0.2cm}\text{ MPCC C-stationary point }\hspace{0.2cm}\text{ MPCC W-stationary point }
		\end{aligned}}
	\end{equation*}

	\section{An Inexact Augmented Lagrangian Method with the Alternating Minimization (IALAM)}\label{sec:alm}

	Problem \eqref{eq:dnn33} is to minimize a nonsmooth nonconvex function subject to linear and bilinear constraints. 
	By exploring the structure of problem \eqref{eq:dnn33}, we propose a variation of the inexact augmented
	Lagrangian (IALM) framework in Subsection \ref{sec:ialm}. Then, we present an alternating minimization algorithm to
	solve the augmented Lagrangian
	subproblem in Subsection \ref{sec:am}. Combining these two parts, we call
	our new algorithm IALAM.	
	In Subsections \ref{sec:almconver} and \ref{sec:amconver}, we prove that { any iterate sequence generated by IALAM has at least one accumulation point and any accumulation point is a KKT point of problem \eqref{eq:dnn33}, which is a MPCC W-stationary point of problem \eqref{eq:dnn} according to Theorem \ref{thm:lsta2}. }
	
	\subsection{The Algorithm Framework}\label{sec:ialm}
	By penalizing the equality constraint of problem \eqref{eq:dnn33}, we can obtain
	its augmented Lagrangian (AL) function as follows
	\begin{equation}\label{eq:lagran2}
		\begin{aligned}
			\mathcal{L}_{\rho}(w,b,v,u;\xi):=&{\mathcal{O}}(w,v,u)+\left\langle \xi,u-\Psi(v)w-Ab \right\rangle+ \frac{\rho}{2}\left\|u-\Psi(v)w-Ab\right\|^2, \\
		\end{aligned}
	\end{equation}
	where $\rho>0$ is the penalty parameter and $$\xi:=\left(\xi_{1,1}^{\top}, \xi_{2,1}^{\top}, \ldots, \xi_{N, 1}^{\top}, \xi_{1,2}^{\top}, \ldots, \xi_{N, L}^{\top}\right)^{\top}$$ is the Lagrangian multiplier associate with $u=\Psi(v)w-Ab$, $\xi_{n,\ell}\in\R^{N_{\ell}}$ for all $n\in[N]$ and $\ell\in[L]$. Recall the definition of $\Psi(v)$ and $A$, it holds that
	$$\left\langle \xi,u-\Psi(v)w-Ab \right\rangle=\sum_{\ell=1}^L\sum_{n=1}^{N}\left\langle \xi_{n,\ell},u_{n,\ell}-W_{\ell}v_{n,\ell-1}-b_{\ell} \right\rangle.$$
	
	In {the framework of any} augmented Lagrangian based approach,
	it requires to solve the following subproblem with the dual variables fixed at each iteration to update the prime variables
	\begin{equation}\label{eq:almsub}
		\begin{aligned}
			\min_{(w,b,v,u)\in\Omega_3}\,&\,\mathcal{L}_{\rho}(w,b,v,u;\xi),\\
		\end{aligned}
	\end{equation}
	where $\Omega_3:=\{(w, b,v,u): w\in\R^{\widetilde{N}}, b\in\R^{\overline{N}}, \mathcal{C}(v,u)\leq 0\}$ and $\mathcal{C}(v,u)$ is defined in \eqref{eq:omega3}. 
	We denote $(w^{(k)},b^{(k)},v^{(k)},u^{(k)},\xi^{(k)})$ as the $k$-th iterate tuple. At the $k$-th iteration,
	we inexactly solve \eqref{eq:almsub} with $\rho=\rho^{(k-1)}$ and 
	$\xi=\xi^{(k-1)}$ {to} obtain an approximate solution $(w^{(k)},b^{(k)},v^{(k)},u^{(k)})\in\Omega_3$ satisfying the following {two} conditions,
	\begin{equation}\label{eq:almsubcondi1}
		\mathcal{L}_{\rho^{(k-1)}}\left(w^{(k)},b^{(k)},v^{(k)},u^{(k)};\xi^{(k-1)}\right)
		<\theta,
	\end{equation}
	and
	\begin{equation}\label{eq:almsubcondi2}
		\operatorname{dist}\left(0, \partial \mathcal{L}_{\rho^{(k-1)}}\left(w^{(k)},b^{(k)},v^{(k)},u^{(k)};\xi^{(k-1)}\right)+\mathcal{N}_{\Omega_3}\left(w^{(k)},b^{(k)},v^{(k)},u^{(k)}\right)\right) \leq \epsilon_{k}.
	\end{equation}
	
	We {describe the} IALM framework {with inexact criteria \eqref{eq:almsubcondi1} and \eqref{eq:almsubcondi2}} in Algorithm \ref{alg:alm}. The definitions of $\theta$ and $\Omega_{\theta}$ can be found in Subsection \ref{sec:modelana1}.	
	
	\begin{algorithm}[ht]
		\caption{The inexact augmented Lagrangian method for solving problem \eqref{eq:dnn33}}\label{alg:alm}
		\begin{algorithmic}
			\State 
			{\textbf{Input:} initial point $(w^{(0)},b^{(0)}, v^{(0)}, u^{(0)})\in\Omega_{\theta}$, parameters $\rho^{(0)}>0$,  $\eta_1, \eta_2, \eta_4 \in(0,1)$, $\eta_3>0$, $\xi^{(0)}\in \R^{m}$, $\gamma\in\mathbb{N}_+$, and $\epsilon_0>0$.
				Set $k:=1$.}
			\While {the stop criterion is not met}
			
			\State \textbf{Step 1:}
			{Solve \eqref{eq:almsub} with $\rho=\rho^{(k-1)}$ and $\xi=\xi^{(k-1)}$
				and obtain $(w^{(k)},b^{(k)},v^{(k)},u^{(k)})\in\Omega_3$ satisfying \eqref{eq:almsubcondi1} and \eqref{eq:almsubcondi2}.}

			\State \textbf{Step 2:} Update the Lagrangian multipliers by
			\begin{equation}\label{eq:xiupdate}
				\xi^{(k)}=\xi^{(k-1)}+\rho^{(k-1)} \left(u^{(k)}-\Psi(v^{(k)})w^{(k)}-Ab^{(k)}\right).
			\end{equation}
			\State \textbf{Step 3:} 
			{If $k\leq \gamma$, set $\rho^{(k)}=\rho^{(k-1)}$ and $\epsilon_{k}=\epsilon_{k-1}$. Else if $k>\gamma$, and
				\begin{equation}\label{eq:rhocondition}
					\left\|u^{(k)}-\Psi(v^{(k)})w^{(k)}-Ab^{(k)}\right\|\leq \eta_1 \max_{t=k-\gamma,\ldots,k-1}\left\|u^{(t)}-\Psi(v^{(t)})w^{(t)}-Ab^{(t)}\right\|,
				\end{equation}
				\quad\,\, then set $\rho^{(k)}=\rho^{(k-1)}$ and $\epsilon_{k}=\sqrt{\eta_1}\epsilon_{k-1}$. Otherwise, set
				\begin{equation}\label{eq:rhoupdate}
					\rho^{(k)}=\max\left\{\rho^{(k-1)}/\eta_2, \left\|\xi^{(k)}\right\|^{1+\eta_3}\right\}\text{ and }\epsilon_{k}=\eta_4\epsilon_{k-1}.
			\end{equation}}
			\quad\,\, Set $k:=k+1$.
			\EndWhile
			\State \textbf{Output:} $(w^{(k)},b^{(k)},v^{(k)},u^{(k)})$.
		\end{algorithmic}
	\end{algorithm}

	\subsection{The Alternating Minimization Algorithm}\label{sec:am}
	
	Subproblem \eqref{eq:almsub} is to minimize a nonsmooth nonconvex function subject to linear constraints. We utilize its block structure and
	propose an alternating minimization algorithm. Before we present the detailed algorithm framework, we introduce how to choose an initial point and update the two blocks. We assume to be at the $k$-th iteration of
	Algorithm \ref{alg:alm}.
	
	{\bf{Initialization.}} Let $(w_{\mathrm{init}}^{(k)},b_{\mathrm{init}}^{(k)},v_{\mathrm{init}}^{(k)},u_{\mathrm{init}}^{(k)})$ be the initial point of the Algorithm \ref{alg:subbcd}, which is updated recursively as follows
	\begin{equation}\label{eq:initbcd}
		\left(w_{\mathrm{init}}^{(k)},b_{\mathrm{init}}^{(k)},v_{\mathrm{init}}^{(k)},u_{\mathrm{init}}^{(k)}\right)=\left\{
		\begin{aligned}
			&\left(w_{\mathrm{init}}^{(k-1)},b_{\mathrm{init}}^{(k-1)},v_{\mathrm{init}}^{(k-1)},u_{\mathrm{init}}^{(k-1)}\right), \text{ if }k>1\text{ and} \\&\hspace{3.8cm}
			\mathcal{L}_{\rho^{(k-1)}}\left(w^{(k-1)},b^{(k-1)},\bar{v},\bar{u};\xi^{(k-1)}\right) \geq \theta,\\
			&\left(w^{(k-1)},b^{(k-1)},\bar{v},\bar{u}\right),\hspace{1.6cm}\text{ otherwise,}
	\end{aligned}\right.
\end{equation}
where $\bar{v}_{n,0}=x_n$, $\bar{u}_{n, \ell}=W_{\ell}^{(k-1)} \bar{v}_{n, \ell-1}+b_{\ell}^{(k-1)}$, and $\bar{v}_{n, \ell}=\sigma\left(\bar{u}_{n, \ell}\right)$ for all $n\in[N]$ and $\ell\in[L]$.
Clearly,  $(w_{\mathrm{init}}^{(k)},b_{\mathrm{init}}^{(k)},v_{\mathrm{init}}^{(k)},u_{\mathrm{init}}^{(k)})$ is a feasible point of problem \eqref{eq:dnn33}
for all $k\in\mathbb{N}_+$ by its definition.

The notations $(w^{(k,\jmath)},b^{(k,\jmath)},v^{(k,\jmath)},u^{(k,\jmath)})$ stands for the $\jmath$-th iterate of the alternating minimization algorithm and the $k$-th iterate of Algorithm \ref{alg:alm}.
For brevity, we drop the superscript $(k-1)$ (and $(k)$) and abuse the notations $\rho, \xi, w^{(\jmath)},b^{(\jmath)},v^{(\jmath)},u^{(\jmath)}$ to denote $\rho^{(k-1)}, \xi^{(k-1)}, w^{(k,\jmath)}, b^{(k,\jmath)}$, $v^{(k,\jmath)}, u^{(k,\jmath)}$, respectively. We assume {to be} at the $j$-th iterate of the alternating minimization algorithm.


\textbf{Update of the $(w,b)$ block.}
Once $(v,u)$ block is fixed at $(v^{(j)},u^{(j)})$.
We compute $(w^{(\jmath+1)},b^{(\jmath+1)})$ by solving the following 
convex problem $\min_{w,b}\mathcal{L}_{\rho}(w,b,v^{(\jmath)},u^{(\jmath)};\xi)$, i.e.,
\begin{equation}\label{eq:bcdsubWb1}
	\begin{aligned}
		\min_{w,b}\mathcal{R}_1(w)+\left\langle \xi,u^{(\jmath)}-\Psi\left(v^{(\jmath)}\right)w-Ab \right\rangle+ \frac{\rho}{2}\left\|u^{(\jmath)}-\Psi\left(v^{(\jmath)}\right)w-Ab\right\|^2.
	\end{aligned}
\end{equation}
This can be solved by some existing methods, for example,
the proximal gradient method.

\textbf{Update of the  $v,u$ block:}
After obtaining $(w^{(\jmath+1)},b^{(\jmath+1)})$, we calculate $(u^{(\jmath+1)},v^{(\jmath+1)})$ in the following way. We define an proximal term $\mathcal{P}(u,v;u^{(\jmath)},v^{(\jmath)}, \tau^{(\jmath)})$ by
\begin{equation}\label{eq:sk}
	\mathcal{P}\left(u,v;u^{(\jmath)},v^{(\jmath)}, \tau^{(\jmath)}\right):=\frac{1}{2}\sum_{n=1}^N\sum_{\ell=2}^{L}\left\|\left(\begin{matrix}
		v_{n,\ell-1}\\ u_{n,\ell}
	\end{matrix}\right)-\left(\begin{matrix}
		v_{n,\ell-1}^{(\jmath)} \\ u_{n,\ell}^{(\jmath)}
	\end{matrix}\right)\right\|_{S_{\ell}^{(\jmath)}}^2+\frac{\tau_1}{2}\sum_{n=1}^N \left\|u_{n,1}-u_{n,1}^{(\jmath)}\right\|^2,
\end{equation}
where $\tau_1>0$ is a given parameter, $\tau^{(\jmath)}:=(\tau_2^{(\jmath)},\ldots,\tau^{(\jmath)}_{L})\zz\in\R^{L-1}$, $\tau_{\ell}^{(\jmath)}$ and matrix $S_{\ell}^{(\jmath)}$ are defined by
\begin{align}
	&\tau_{\ell}^{(\jmath)}:=\rho\left\|\left[-{W}^{(\jmath+1)}_{\ell}\,\,\,\, I_{N_{\ell}}\right]\right\|^2 +\tau_1,\label{eq:taudefine}\\
	&S_{\ell}^{(\jmath)}:=\tau_{\ell}^{(\jmath)} I_{N_{\ell}+N_{\ell-1}}-\rho \left[-{W}^{(\jmath+1)}_{\ell}\,\,\,\, I_{N_{\ell}}\right]\zz\left[-{W}^{(\jmath+1)}_{\ell}\quad I_{N_{\ell}}\right],
\end{align}
	respectively for all $\ell=2,3,\ldots,L$. Clearly, $S_{\ell}^{(\jmath)}\succeq \tau_1 I_{N_{\ell}+N_{\ell-1}}$ is a symmetric positive definite matrix, since $\|[-{W}^{(\jmath+1)}_{\ell}\,\,\,\, I_{N_{\ell}}]\|^2$ is the maximal eigenvalue of $[-{W}^{(\jmath+1)}_{\ell}\,\,\,\, I_{N_{\ell}}]\zz[-{W}^{(\jmath+1)}_{\ell}\,\,\,\, I_{N_{\ell}}]$ for any $\ell=2,3,\ldots,L$.
	Then, we arrive at a linearly constrained problem
	\begin{equation}\label{eq:bcdsub1}
		\begin{aligned}
			\arg\min_{v,u}\, & \, \mathcal{L}_{\rho}\left({w}^{(\jmath+1)},{b}^{(\jmath+1)},v,u;\xi\right)+\mathcal{P}\left(u,v;u^{(\jmath)},v^{(\jmath)}, \tau^{(\jmath)}\right)\\
			\st \, & \, v\geq  u, v\geq  \alpha u.
		\end{aligned}
	\end{equation}
	We can calculate its unique solution $({v}^{(\jmath+1)}, {u}^{(\jmath+1)})$ in the following way.
	
	Notice that
	\begin{align*}
		&\frac{\rho}{2}\left\|u-\Psi(v)w^{(\jmath+1)}-Ab^{(\jmath+1)}\right\|^2+\mathcal{P}\left(u,v;u^{(\jmath)},v^{(\jmath)},\tau^{(\jmath)}\right)
		\\
		=&\frac{1}{2}\sum_{n=1}^N\sum_{\ell=2}^L\left(\rho\left\| u_{n,\ell}-\left({W}^{(\jmath+1)}_{\ell}v_{n,\ell-1}+{b}^{(\jmath+1)}_{\ell}\right)\right\|^2+\left\|\left(\begin{matrix}
			v_{n,\ell-1}\\ u_{n,\ell}
		\end{matrix}\right)-\left(\begin{matrix}
			v_{n,\ell-1}^{(\jmath)} \\ u_{n,\ell}^{(\jmath)}
		\end{matrix}\right)\right\|_{S_{\ell}^{(\jmath)}}^2\right)\\
		&+\frac{\rho}{2}\sum_{n=1}^N\left\| u_{n,1}-\left(W^{(\jmath+1)}_{1}x_n+{b}^{(\jmath+1)}_{1}\right)\right\|^2+\frac{\tau_1}{2}\sum_{n=1}^N \left\|u_{n,1}-u_{n,1}^{(\jmath)}\right\|^2,
	\end{align*}
	where
	\begin{align*}
		&\sum_{n=1}^N\sum_{\ell=2}^L\left(\rho\left\| u_{n,\ell}-\left({W}^{(\jmath+1)}_{\ell}v_{n,\ell-1}+{b}^{(\jmath+1)}_{\ell}\right)\right\|^2+\left\|\left(\begin{matrix}
			v_{n,\ell-1}\\ u_{n,\ell}
		\end{matrix}\right)-\left(\begin{matrix}
			v_{n,\ell-1}^{(\jmath)} \\ u_{n,\ell}^{(\jmath)}
		\end{matrix}\right)\right\|_{S_{\ell}^{(\jmath)}}^2\right)\\
		=&\sum_{n=1}^N\sum_{\ell=2}^L\Bigg(\rho\left\| u^{(\jmath)}_{n,\ell}-\left({W}^{(\jmath+1)}_{\ell}v^{(\jmath)}_{n,\ell-1}+{b}^{(\jmath+1)}_{\ell}\right)\right\|^2+\left\|\left(\begin{matrix}
			v_{n,\ell-1}\\ u_{n,\ell}
		\end{matrix}\right)-\left(\begin{matrix}
			v_{n,\ell-1}^{(\jmath)} \\ u_{n,\ell}^{(\jmath)}
		\end{matrix}\right)\right\|_{S_{\ell}^{(\jmath)}}^2\\
		&+2\rho\left(u^{(\jmath)}_{n,\ell}-\left({W}^{(\jmath+1)}_{\ell}v^{(\jmath)}_{n,\ell-1}+{b}^{(\jmath+1)}_{\ell}\right)\right)\zz\left[-{W}^{(\jmath+1)}_{\ell} \,\,\,\, I_{N_{\ell}}\right]\left(\begin{matrix}
			v_{n,\ell-1}-v^{(\jmath)}_{n,\ell-1}\\ u_{n,\ell}-u_{n,\ell}^{(\jmath)}
		\end{matrix}\right)\\
		&+\rho\left(\begin{matrix}
			v_{n,\ell-1}-v^{(\jmath)}_{n,\ell-1}\\ u_{n,\ell}-u_{n,\ell}^{(\jmath)}
		\end{matrix}\right)\zz \begin{bmatrix}
			({W}^{(\jmath+1)}_{\ell})\zz {W}^{(\jmath+1)}_{\ell} & -({W}^{(\jmath+1)}_{\ell})\zz\\
			-{W}^{(\jmath+1)}_{\ell} & I_{N_{\ell}}
		\end{bmatrix}\left(\begin{matrix}
			v_{n,\ell-1}-v^{(\jmath)}_{n,\ell-1}\\ u_{n,\ell}-u_{n,\ell}^{(\jmath)}
		\end{matrix}\right)
		\Bigg)\\
		=&\sum_{n=1}^N\sum_{\ell=2}^L\Bigg(\rho\left\| u^{(\jmath)}_{n,\ell}-\left(W^{(\jmath+1)}_{\ell}v^{(\jmath)}_{n,\ell-1}+{b}^{(\jmath+1)}_{\ell}\right)\right\|^2+\tau_{\ell}^{(\jmath)} \left(\begin{matrix}
			v_{n,\ell-1}-v^{(\jmath)}_{n,\ell-1}\\ u_{n,\ell}-u_{n,\ell}^{(\jmath)}
		\end{matrix}\right)\zz \left(\begin{matrix}
			v_{n,\ell-1}-v^{(\jmath)}_{n,\ell-1}\\ u_{n,\ell}-u_{n,\ell}^{(\jmath)}
		\end{matrix}\right)\\
		&+2\rho\left(u^{(\jmath)}_{n,\ell}-\left({W}^{(\jmath+1)}_{\ell}v^{(\jmath)}_{n,\ell-1}+{b}^{(\jmath+1)}_{\ell}\right)\right)\zz\left[-{W}^{(\jmath+1)}_{\ell} \,\,\,\, I_{N_{\ell}}\right]\left(\begin{matrix}
			v_{n,\ell-1}-v^{(\jmath)}_{n,\ell-1}\\ u_{n,\ell}-u_{n,\ell}^{(\jmath)}
		\end{matrix}\right)
		\Bigg).
	\end{align*}
	
	Then, the objective function of problem \eqref{eq:bcdsub1} can be simplified as
	\begin{align}
		&\frac{1}{N}\sum_{n=1}^N\|v_{n,L}-y_n\|^2+\lambda_v\|v\|^2+\beta \zz (v-  \sigma(u))+\sum_{\ell=1}^L\sum_{n=1}^{N}\left\langle \xi_{n,\ell},u_{n,\ell}-W^{(\jmath+1)}_{\ell}v_{n,\ell-1}\right\rangle\notag\\
		+&\frac{1}{2}\sum_{n=1}^N\sum_{\ell=2}^L\Bigg(
		\tau_{\ell}^{(\jmath)}\left\|u_{n,\ell}-u_{n,\ell}^{(\jmath)}\right\|^2
		+2\rho\left(u^{(\jmath)}_{n,\ell}-\left({W}^{(\jmath+1)}_{\ell}v^{(\jmath)}_{n,\ell-1}+{b}^{(\jmath+1)}_{\ell}\right)\right)\zz \left(u_{n,\ell}-u_{n,\ell}^{(\jmath)}\right)\notag
		\\ +& \tau_{\ell}^{(\jmath)}\left\|v_{n,\ell-1}-v^{(\jmath)}_{n,\ell-1}\right\|^2-2\rho\left(u^{(\jmath)}_{n,\ell}-\left({W}^{(\jmath+1)}_{\ell}v^{(\jmath)}_{n,\ell-1}+{b}^{(\jmath+1)}_{\ell}\right)\right)\zz{W}^{(\jmath+1)}_{\ell}\left(v_{n,\ell-1}-v^{(\jmath)}_{n,\ell-1}\right)
		\Bigg)\notag
		\\+&\frac{\rho}{2}\sum_{n=1}^N\left(\left\| u_{n,1}-\left(W^{(\jmath+1)}_{1}x_n+{b}^{(\jmath+1)}_{1}\right)\right\|^2\right)+\frac{\tau_1}{2}\sum_{n=1}^N \left\|u_{n,1}-u_{n,1}^{(\jmath)}\right\|^2.
	\end{align}
	
	Hence, subproblem \eqref{eq:bcdsub1} can be separated into $m$ independent subproblems of the following
	structure
	\begin{equation}\label{eq:sub332}
		\begin{aligned}
			\min_{r, s\in\R}\, & \, r-\max\{s,\alpha	 s\}+\frac{d_1}{2}\left(r-\frac{d_3}{d_1}\right)^2+\frac{d_2}{2}\left(s-\frac{d_4}{d_2}\right)^2\\
			\st \;\;\;& \, r\geq  s, r\geq  \alpha s,
		\end{aligned}
	\end{equation}
	where the constants $d_1,d_2> 0$ and $d_3,d_4\in\R$ are dependent on
	the parameters of problem \eqref{eq:bcdsub1}.
	
	Restricting problem \eqref{eq:sub332} to $\{(r,s): s\geq 0\}$,  we obtain $\min_{r\geq  s, s\geq 0}\,  \, r-s+d_1(r-d_3/d_1)^2/2+d_2(s-d_4/d_2)^2/2$ and its closed-form solution \citep[page 81]{facchinei2003finite}
	\begin{equation}
		\label{eq:updateuv1}
		(r_1^*,\,\,s_1^*):=\mathrm{Proj}^{\mathrm{diag}(d_1,d_2)}_{\{(r,\,\,s):r\geq  s,\,\,s\geq 0\}}\left(\frac{d_3-1}{d_1},\,\,\frac{d_4+1}{d_2} \right).
	\end{equation}
	On the other hand, restricting problem \eqref{eq:sub332} to $\{(r,s): s\leq 0\}$, we obtain $\min_{r\geq  \alpha s,\,\,s\leq 0}\,  \, r-\alpha s+d_1(r-d_3/d_1)^2/2+d_2(s-d_4/d_2)^2/2$ and its closed-form solution
	\begin{equation}
		\label{eq:updateuv2}
		(r_2^*,\,\,s_2^*)=\mathrm{Proj}^{\mathrm{diag}(d_1,d_2)}_{\{(r,\,\,s):r\geq \alpha s,\,\,s\leq 0\}}\left(\frac{d_3-1}{d_1},\,\,\frac{d_4+\alpha}{d_2} \right).
	\end{equation}
	By comparing the objective function values at $(r_1^*,\,\,s_1^*)$ and $(r_2^*,\,\,s_2^*)$, we obtain the unique solution of \eqref{eq:sub332}.
	
	We next present the framework of the alternating minimization method for solving \eqref{eq:almsub} as follows
	\begin{algorithm}[H]
		\caption{An alternating minimization method for solving \eqref{eq:almsub}}\label{alg:subbcd}
		\begin{algorithmic}
			\State 
			\textbf{Input:} matrix $A$, the vector $\xi$, the parameters $\rho>0$ and $\tau_1>0$. Initialize\\ $(w^{(\jmath)},b^{(\jmath)},v^{(\jmath)},u^{(\jmath)})$ by \eqref{eq:initbcd}. Set  $\jmath=0$.  
			
			\State {\textbf{Step 1:} Update $(w^{(\jmath+1)},b^{(\jmath+1)})$ by solving problem \eqref{eq:bcdsubWb1}.}
			
			\State  \textbf{Step 2:} Update $({u}^{(\jmath+1)},{v}^{(\jmath+1)})$ by solving problem \eqref{eq:bcdsub1}.
			\State \textbf{Step 3:}  Set
			$\jmath:=\jmath+1$.
			If the stop criterion is not met, return to Step 1.
			
			\State \textbf{Output}: $(w^{(\jmath)},b^{(\jmath)},v^{(\jmath)},u^{(\jmath)})$.
		\end{algorithmic}
	\end{algorithm}
	
	\begin{remark}
		{The reasons why we divide subproblem \eqref{eq:almsub} into $(w,b)$ and $(v,u)$ blocks are two-fold.
			Firstly, 
			$w, b$ are the vectorized weight matrices and bias vectors, respectively, meanwhile $v, u$ are the  auxiliary variables. Secondly, subproblem \eqref{eq:almsub} restricted to both of these two blocks are easy to solve. More precisely, the $(w,b)$ subproblem is strongly convex and has one unique solution, meanwhile the $(v,u)$ subproblem has a closed-form unique solution.}
	\end{remark}

	\subsection{Convergence Analysis of Algorithm \ref{alg:alm}}\label{sec:almconver}
	
	In this subsection, we establish the convergence of 
	Algorithm \ref{alg:alm}.
	\begin{theorem}
		Let $\{(w^{(k)},b^{(k)},v^{(k)},u^{(k)})\}$ be the sequence generated by Algorithm \ref{alg:alm} with $\eta_3>1$. Then the following statements hold.
			
			
			(a) $\liminf_{k\rightarrow \infty}\|u^{(k)}-\Psi(v^{(k)})w^{(k)}-Ab^{(k)}\| = 0$ and the sequence $\{(w^{(k)},b^{(k)},v^{(k)},u^{(k)})\}$ has at least one accumulation point.
			
			(b) $\liminf_{k\rightarrow \infty} \dist((w^{(k)},b^{(k)},v^{(k)},u^{(k)}), \mathcal{Z}^*)=0$, where $\mathcal{Z}^*$ is the set of KKT points of problem \eqref{eq:dnn33}.
			
			(c) If in addition that $\gamma=1$, then $\lim_{k\rightarrow \infty}\|u^{(k)}-\Psi(v^{(k)})w^{(k)}-Ab^{(k)}\| = 0.$ Furthermore, any accumulation point $(w^{*},b^{*},v^{*},u^{*})$ of $\{(w^{(k)},b^{(k)},v^{(k)}$, $u^{(k)})\}$  is a KKT point of problem \eqref{eq:dnn33}.
	\end{theorem}

	\begin{proof}
		(a) 
		Case (i): the sequence $\{\rho^{(k)}\}$ is bounded. In this case, the formula \eqref{eq:rhoupdate} can be called at most finite times.
		That is for some $k_{0}>\gamma$, it holds that $\rho^{(k)}=\rho^{(k_0)}$ and
		\begin{equation}\label{eq:ineqc2-2}
			\left\|u^{(k)}-\Psi\left(v^{(k)}\right)w^{(k)}-Ab^{(k)}\right\|\leq \eta_1 \max_{t=k-\gamma,\ldots,k-1}\left\|u^{(t)}-\Psi\left(v^{(t)}\right)w^{(t)}-Ab^{(t)}\right\|.
		\end{equation}
		for all $k > k_{0}$.
		By simple calculation, it then follows that for all $k> k_0$,
		\begin{equation}\label{eq:ineqc3}
			\max_{t=k,k+1,\ldots,k+\gamma-1}\left\|u^{(t)}-\Psi\left(v^{(t)}\right)w^{(t)}-Ab^{(t)}\right\|\leq \eta_1 \max_{t=k-\gamma,\ldots,k-1}\left\|u^{(t)}-\Psi\left(v^{(t)}\right)w^{(t)}-Ab^{(t)}\right\|,
		\end{equation}
		and
		hence
		$
		\lim _{k \rightarrow \infty} \max_{t=k-\gamma,\ldots,k-1}\|u^{(t)}-\Psi(v^{(t)})w^{(t)}-Ab^{(t)}\|=0.
		$
		This yields
		\begin{equation}\label{eq:kinfty}
			\liminf_{k\rightarrow \infty}\left\|u^{(k)}-\Psi\left(v^{(k)}\right)w^{(k)}-Ab^{(k)}\right\|=\lim_{k\rightarrow \infty}\left\|u^{(k)}-\Psi\left(v^{(k)}\right)w^{(k)}-Ab^{(k)}\right\| = 0.
		\end{equation}
		We next show that for all $k> k_0$,
		\begin{equation}\label{eq:xiineq-2}
			\left\|\xi^{(k)}\right\| \leq \left\|\xi^{(k_0)}\right\|+\rho^{(k_0)}\sum_{i=1}^{k-k_0}\left\|u^{(k_0+i)}-\Psi\left(v^{(k_0+i)}\right)w^{(k_0+i)}-Ab^{(k_0+i)}\right\|.
		\end{equation}
		In view of $\rho^{(k)}=\rho^{(k_0)}$ for all $k> k_0$ and the updating rule \eqref{eq:xiupdate} of $\xi^{(k)}$ that
		\begin{equation*}
			\left\|\xi^{(k_0+i)}\right\| \leq \left\|\xi^{(k_0+i-1)}\right\|+\rho^{(k_0)} \left\|u^{(k_0+i)}-\Psi\left(v^{(k_0+i)}\right)w^{(k_0+i)}-Ab^{(k_0+i)}\right\|
		\end{equation*}
		for all $i \in \mathbb{N}_+$. Summing up the above inequalities
		for all $i \in \mathbb{N}_+$ yields \eqref{eq:xiineq-2}.
		
		Combining {inequalities} \eqref{eq:ineqc2-2}, \eqref{eq:ineqc3} {with} \eqref{eq:xiineq-2}, we obtain that
		\begin{equation*}
			\begin{aligned}
				\left\|\xi^{(k)}\right\| &\leq \left\|\xi^{(k_0)}\right\|+\rho^{(k_0)}\sum_{i=1}^{k-k_0}\eta_1^{\lceil i/\gamma\rceil}\max_{t=k_0-\gamma+1,\ldots,k_0}\left\|u^{(t)}-\Psi\left(v^{(t)}\right)w^{(t)}-Ab^{(t)}\right\| \\
				&\leq \left\|\xi^{(k_0)}\right\|+\frac{\gamma \eta_1 \rho^{(k_0)}}{1-\eta_1}\max_{t=k_0-\gamma+1,\ldots,k_0}\left\|u^{(t)}-\Psi\left(v^{(t)}\right)w^{(t)}-Ab^{(t)}\right\|.
			\end{aligned}
		\end{equation*}
		Hence $\{\xi^{(k)}\}$ is bounded.
		
		Furthermore, 
		we obtain by the inequality \eqref{eq:almsubcondi1} and the definition of $\mathcal{L}_{\rho}$ that for all $k\in\mathcal{K}$,
		\begin{equation}\label{eq:xiineq3oo}
			\begin{aligned}
				&{\mathcal{O}}\left(w^{(k+1)},v^{(k+1)},u^{(k+1)}\right)+ \frac{\rho^{(k)}}{2}\left\|\frac{\xi^{(k)}}{\rho^{(k)}}+u^{(k+1)}-\Psi\left(v^{(k+1)}\right)w^{(k+1)}-Ab^{(k+1)}\right\|^2 \\
				\leq &\theta+\frac{\left\|\xi^{(k)}\right\|^2}{2\rho^{(k)}}.
			\end{aligned}
		\end{equation}
		
		It follows the inclusion $\left\{\left(w^{(k+1)}, b^{(k+1)},v^{(k+1)},u^{(k+1)}\right)\right\} \subseteq \Omega_3$ and the definition of $\mathcal{O}$ that
		$\left\{ {\mathcal{O}}\left(w^{(k+1)},v^{(k+1)},u^{(k+1)}\right)\right\}$ is bounded away from $0$.
		Using this, $\{\xi^{(k)}\}$ being bounded, $\rho^{(k)}=\rho^{(k_0)}$ for all $k>k_0$ and the relation \eqref{eq:xiineq3oo} yield that $\left\{ {\mathcal{O}}\left(w^{(k+1)},v^{(k+1)},u^{(k+1)}\right)\right\}$ is bounded. Using a similar method as that in the proof of Theorem \ref{thm:nonem}, we obtain that the sequence $\{(w^{(k)},b^{(k)},v^{(k)},u^{(k)})\}$ is bounded.
		Hence, the sequence $\{(w^{(k)},b^{(k)},v^{(k)}$, $u^{(k)})\}$ has at least one accumulation point. 

		Case (ii): the sequence $\{\rho^{(k)}\}$ is unbounded. In this case, the set
		\begin{equation}\label{eq:kdefine}
			\mathcal{K}:=\left\{k: \rho^{(k)}=
			\max\left\{\rho^{(k-1)}/\eta_2, \left\|\xi^{(k)}\right\|^{1+\eta_3}\right\}\right\}
		\end{equation} is infinite. Together with {the inclusion} $\eta_2\in (0,1)$, it follows that
		$\left\{\rho^{(k)}\right\} \rightarrow \infty$ as $k \rightarrow \infty, k\in\mathcal{K}$.
		For all $k\in\mathcal{K}$, we have $\|\xi^{(k)}\|^{1+\eta_3}\leq \rho^{(k)}$, which further yields
		$
		\frac{\left\|\xi^{(k)}\right\|}{\rho^{(k)}}\leq (\rho^{(k)})^{-\eta_3/(1+\eta_3)}.
		$
		Together with {the fact} $\left\{\rho^{(k)}\right\} \rightarrow \infty$ as $k \rightarrow \infty, k\in\mathcal{K}$, we derive
		\begin{equation}\label{eq:xirela}
			\lim_{k \rightarrow \infty, k\in\mathcal{K}} \frac{\left\|\xi^{(k)}\right\|}{\rho^{(k)}}=0.
		\end{equation}
		Similarly, we have $
		\frac{\left\|\xi^{(k)}\right\|^2}{\rho^{(k)}}\leq (\rho^{(k)})^{(1-\eta_3)/(1+\eta_3)}
		$ and
		\begin{equation}\label{eq:xirela22}
			\lim_{k \rightarrow \infty, k\in\mathcal{K}} \frac{\left\|\xi^{(k)}\right\|^2}{\rho^{(k)}}=0
		\end{equation}
		by $\eta_3>1$ and $\left\{\rho^{(k)}\right\} \rightarrow \infty$ as $k \rightarrow \infty, k\in\mathcal{K}$.
		
		Similarly, we obtain \eqref{eq:xiineq3oo} with $k\in\mathbb{N}_+$ replaced by $k\in\mathcal{K}$.
		Dividing both sides of the above inequality by $\rho^{(k)}/2$, we obtain that for all $k\in\mathcal{K}$,
		\begin{equation}\label{eq:xiineq3}
			\begin{aligned}
				&\left\|\frac{\xi^{(k)}}{\rho^{(k)}}+u^{(k+1)}-\Psi\left(v^{(k+1)}\right)w^{(k+1)}-Ab^{(k+1)}\right\|^2\\ \leq& \frac{2}{\rho^{(k)}}\left(\theta-{\mathcal{O}}\left(w^{(k+1)},v^{(k+1)},u^{(k+1)}\right)\right)+\frac{\left\|\xi^{(k)}\right\|^2}{(\rho^{(k)})^2}.
			\end{aligned}
		\end{equation}
		Using this, the relation \eqref{eq:xirela}, $\left\{ {\mathcal{O}}\left(w^{(k+1)},v^{(k+1)},u^{(k+1)}\right)\right\}$ is bounded away from $0$, and $\left\{\rho^{(k)}\right\} \rightarrow \infty$ as $k \rightarrow \infty$ and $k\in\mathcal{K}$, we then derive that
		\begin{equation}\label{eq:xiinfty}
			\lim_{k \rightarrow \infty, k\in\mathcal{K}}\left\|\frac{\xi^{(k)}}{\rho^{(k)}}+u^{(k+1)}-\Psi\left(v^{(k+1)}\right)w^{(k+1)}-Ab^{(k+1)}\right\|=0,
		\end{equation}
		which together with \eqref{eq:xirela} yield
		\begin{equation}\label{eq:feasviinfty}
			\begin{aligned}
				&\liminf_{k\rightarrow \infty}\left\|u^{(k)}-\Psi\left(v^{(k)}\right)w^{(k)}-Ab^{(k)}\right\|\\=&\lim_{k\rightarrow \infty, k\in\mathcal{K}}\left\|u^{(k+1)}-\Psi\left(v^{(k+1)}\right)w^{(k+1)}-Ab^{(k+1)}\right\| = 0.
			\end{aligned}
		\end{equation}
		
		Combining  the relations \eqref{eq:xiineq3oo}, \eqref{eq:xirela22}, $\rho^{(k)}>0$ {with the sequence} $\left\{ {\mathcal{O}}\left(w^{(k+1)},v^{(k+1)},u^{(k+1)}\right)\right\}$ being bounded away from 0, there exists $\bar{k}>0$ such that $\mathcal{O}\left(w^{(k+1)},v^{(k+1)},u^{(k+1)}\right)<2\theta$ for all $k>\bar{k}$, $k\in\mathcal{K}$. Using a similar method as that in the proof of Theorem \ref{thm:nonem}, we obtain that the sequence $\{\left(w^{(k+1)},b^{(k+1)},v^{(k+1)},u^{(k+1)}\right)\}_{k\in\mathcal{K}}$ is bounded.
		Hence, the sequence $\{(w^{(k+1)},b^{(k+1)},v^{(k+1)}$, $u^{(k+1)})\}_{k\in\mathcal{K}}$ has at least one accumulation point.
		
		Combining the above two cases, the proof is completed.

		(b) Let $(w^{*},b^{*},v^{*},u^{*})$ be an accumulation point of $\{(w^{(k+1)}, b^{(k+1)}, v^{(k+1)}$, $u^{(k+1)})\}_{k\in\mathcal{K}}$, where $\mathcal{K}$ is defined in \eqref{eq:kdefine}.
		By what we have proved in the statement (a), the closedness of $\Omega_2$, we obtain that the point $(w^{*},b^{*},v^{*},u^{*})$ belongs to the feasible set of problem \eqref{eq:dnn33}.
		
		Since $(w^{*},b^{*},v^{*},u^{*})$ is an accumulation point of $\left\{\left(w^{(k+1)},b^{(k+1)},v^{(k+1)},u^{(k+1)}\right)\right\}_{k\in\mathcal{K}}$, there exists a subsequence $\left\{(w^{(j_k)},b^{(j_k)},v^{(j_k)},u^{(j_k)})\right\}$ of $\left\{\left(w^{(k+1)},b^{(k+1)},v^{(k+1)},u^{(k+1)}\right)\right\}_{k\in\mathcal{K}}$ such that
		$$\lim_{k\rightarrow\infty}\left(w^{(j_k)},b^{(j_k)},v^{(j_k)},u^{(j_k)}\right)= (w^{*},b^{*},v^{*},u^{*}).$$
		
		Since $\Omega_3\subset\Omega_2$, then Lemma \ref{lem:mfcqsys} yields that the MFCQ holds at any feasible point $(w,b,v,u)$ for problem \eqref{eq:almsub}. Together with \cite[Theorem 6.14]{roc1998var}, we obtain that $\mathcal{N}_{\Omega_3}(w,b,v,u)= \{ \nabla_{(v,u)} (\mu\zz \mathcal{C}( v,u))  : \,\, \mu\zz \mathcal{C}(v,u)=0,\,\,\mu\in\R_+^{2m}\}$.
		Together with {the inequality} \eqref{eq:almsubcondi2},
		and the definition of $\mathcal{L}_{\rho}(w,b,v,u;\xi)$, there exist $\mu^{(j_k)}\geq 0$ and $\zeta^{(j_k)}$ satisfying $(\mu^{(j_k)})\zz \mathcal{C}( v^{(j_k)},u^{(j_k)})=0$, $\|\zeta^{(j_k)}\|\leq \epsilon_{j_k}$ such that
		\begin{equation}\label{eq:zetakini}
			\begin{aligned}
				\zeta^{(j_k)}\in \partial_{(w,b,v,u)} &{\mathcal{O}}\left(w^{(j_k)},v^{(j_k)},u^{(j_k)}\right) + \nabla\bigg( \left(\xi^{(j_k-1)}\right)\zz \left(u^{(j_k)}-\Psi\left(v^{(j_k)}\right)w^{(j_k)}-Ab^{(j_k)}\right)\\+&\frac{\rho^{(j_k-1)}}{2}\left\|u^{(j_k)}-\Psi\left(v^{(j_k)}\right)w^{(j_k)}-A^{(j_k)}b^{(j_k)}\right\|^2+\left(\mu^{(j_k)}\right)\zz \mathcal{C}\left( v^{(j_k)},u^{(j_k)}\right)\bigg).
			\end{aligned}
		\end{equation}
		In view of $\xi^{(j_k)}=\xi^{(j_k-1)}+\rho^{(j_k-1)} \left(u^{(j_k)}-\Psi(v^{(j_k)})w^{(j_k)}-Ab^{(j_k)}\right)$ and $\nabla_z \frac{1}{2}\|z\|^2=z\nabla_z z$ for $z\in\R^m$, we then obtain that
		\begin{equation}\label{eq:zetak}
			\begin{aligned}
				\zeta^{(j_k)}\in &\partial_{(w,b,v,u)} {\mathcal{O}}\left(w^{(j_k)},v^{(j_k)},u^{(j_k)}\right) \\&+ \nabla\left( \left(\xi^{(j_k)}\right)\zz \left(u^{(j_k)}-\Psi\left(v^{(j_k)}\right)w^{(j_k)}-Ab^{(j_k)}\right)+\left(\mu^{(j_k)}\right)\zz \mathcal{C}\left( v^{(j_k)},u^{(j_k)}\right)\right).
			\end{aligned}
		\end{equation}
		
		It holds from the update rule in Algorithm \ref{alg:alm} that $
		\epsilon_{k}\leq \max\{\sqrt{\eta_1}, \eta_4\}\epsilon_{k-1}$ for all $k>\gamma$.
		Together with {the relationship} $0<\eta_1,\eta_4<1$, we derive $\lim_{k \rightarrow \infty}\epsilon_{k}=0.$
		It then follows that $\zeta^{(j_k)}\rightarrow 0$ as $k\rightarrow\infty$.
		
		Let $r_k:=\max\left\{\|\xi^{(j_k)}\|_{\infty}, \|\mu^{(j_k)}\|_{\infty}\right\}$. Suppose that $\left\{r_{k}\right\}$ is unbounded. 
		Without loss of generality, we assume that as $k\rightarrow\infty$, it holds that
		\begin{equation}
			\frac{\xi^{(j_k)}}{r_k}\rightarrow \xi^*,\text{ and } \frac{{\mu}^{(j_k)}}{r_k}\rightarrow{\mu}^*.
		\end{equation}
		It then holds that $\max\{\|\xi^*\|_{\infty}, \|{\mu}^*\|_{\infty}\}=1$ and ${\mu}^*\geq 0$, since ${\mu}^{(j_k)}\geq 0$ for all $k\ge 0$.
		
		Dividing by $r_k$ and taking the limit $ k \rightarrow\infty$ on the both sides of \eqref{eq:zetak}, we obtain
		\begin{equation*}
			0 =  \nabla_{(w,b,v,u)} \left((\xi^{*})\zz (u^*-\Psi(v^*)w^*-Ab^*)+ (\mu^*)\zz \mathcal{C}( v^*,u^*)\right).
		\end{equation*}
		This together with Lemma \ref{lem:mfcqsys} and {the equality} $\max\{\|\xi^*\|_{\infty}, \|{\mu}^*\|_{\infty}\}=1$ lead to a contradiction. 
		$\left\{r_{k}\right\}$ is hence bounded. Without loss of generality, we assume that as $ k \rightarrow \infty$,
		$$
		\xi^{(j_k)}\rightarrow \xi^*,\text{ and } {\mu}^{(j_k)}\rightarrow{\mu}^*.
		$$
		Since ${\mu}^{(j_k)}\geq 0$ for all $k\ge 0$, we have
		$
		{\mu}^{*} \geq 0.
		$
		
		Taking the limit $k \rightarrow\infty$ on both sides of \eqref{eq:zetak},
		we obtain that
		\begin{equation*}
			\begin{aligned}
				0 \in \partial_{(w,b,v,u)} \left({\mathcal{O}}(w^*,v^*,u^*)+(\xi^{*})\zz (u^*-\Psi(v^*)w^*-Ab^*)+ (\mu^*)\zz \mathcal{C}( v^*,u^*)\right),
			\end{aligned}
		\end{equation*}
		which together with  {the inequality} 
		$ 	{\mu}^{*} \geq 0$ yield that $(w^{*},b^{*},v^{*},u^{*})$ is a KKT point of problem \eqref{eq:dnn33}. Hence $\liminf_{k\rightarrow \infty} \dist((w^{(k)},b^{(k)},v^{(k)},u^{(k)}), \mathcal{Z}^*)=0$ by the definition of $(w^{*},b^{*},v^{*},u^{*})$, which completes the statement (b).
		
		(c)
		Case (i): the set $\mathcal{K}$ defined in \eqref{eq:kdefine} is finite. By what we have proved in the statement (a), we have $\lim_{k\rightarrow \infty}\left\|u^{(k)}-\Psi\left(v^{(k)}\right)w^{(k)}-Ab^{(k)}\right\| = 0$ in this case. Using a similar method as that in the proof of statement (b), any accumulation point $(w^{*},b^{*},v^{*},u^{*})$ of $\{(w^{(k)},b^{(k)},v^{(k)},u^{(k)})\}$ is a KKT point of problem \eqref{eq:dnn33}.
		
		Case (ii): the set $\mathcal{K}$ defined in \eqref{eq:kdefine} is infinite.
		
		For any given $k\in\mathbb{N}_+$, let $t_k$ be the largest element in $\mathcal{K}$ satisfying $t_k\leq k$. We then show that
		\begin{equation}\label{eq:xiineq}
			\frac{\left\|\xi^{(k)}\right\|}{\rho^{(k)}} \leq \frac{\left\|\xi^{(t_k)}\right\|}{\rho^{(t_k)}}+\sum_{i=1}^{k-t_k}\left\|
			u^{(t_k+i)}-\Psi \left(v^{(t_k+i)}\right)w^{(t_k+i)}-Ab^{(t_k+i)}\right\|.
		\end{equation}
		Clearly, the inequality \eqref{eq:xiineq} holds when $k=t_k.$ We now suppose $k>t_k.$ In view of {the fact that} $\rho^{(t_k+i)}=\rho^{(t_k)}$ for all $0<i \leq k-t_k$ and the updating rule \eqref{eq:xiupdate} of $\xi^{(k)}$, we have
		\begin{equation*}
			\frac{\left\|\xi^{(t_k+i)}\right\|}{\rho^{(t_k+i)}}=\frac{\left\|\xi^{(t_k+i)}\right\|}{\rho^{(t_k+i-1)}}\leq \left\|\frac{\xi^{(t_k+i-1)}}{\rho^{(t_k+i-1)}}\right\|+\left\|
			u^{(t_k+i)}-\Psi\left(v^{(t_k+i)}\right)w^{(t_k+i)}-Ab^{(t_k+i)}
			\right\|
		\end{equation*}
		for all $i \in [k-t_k]$. Summing up the above inequalities
		for all $i \in [k-t_k]$ yields \eqref{eq:xiineq}.
		
		For all $i \in [k-t_k]$, we obtain from $\gamma=1$, \eqref{eq:rhocondition} and the definition of $t_k$ that
		\begin{equation}\label{eq:ineqc2}
			\begin{aligned}
				&\left\|u^{(t_k+i)}-\Psi\left(v^{(t_k+i)}\right)w^{(t_k+i)}-Ab^{(t_k+i)}\right\|
				\\\leq& \eta_1\left\|u^{(t_k+i-1)}-\Psi\left(v^{(t_k+i-1)}\right)w^{(t_k+i-1)}-Ab^{(t_k+i-1)}\right\|.
			\end{aligned}	
		\end{equation}
		Together with the inequality \eqref{eq:xiineq}, we derive
		\begin{equation}\label{eq:xiineq2}
			\begin{aligned}
				&\frac{\left\|\xi^{(k)}\right\|}{\rho^{(k)}}
				\leq \frac{\left\|\xi^{(t_k)}\right\|}{\rho^{(t_k)}}+
				\sum_{i=1}^{k-t_k}\eta^{i-1}\left\|
				u^{(t_k+1)}-\Psi\left(v^{(t_k+1)}\right)w^{(t_k+1)}-Ab^{(t_k+1)}\right\|\\
				\leq	&	 \frac{\left\|\xi^{(t_k)}\right\|}{\rho^{(t_k)}}
				+\frac{1}{1-\eta_1}\left\|
				u^{(t_k+1)}-\Psi\left(v^{(t_k+1)}\right)w^{(t_k+1)}-Ab^{(t_k+1)}\right\|.\\
			\end{aligned}
		\end{equation}
		
		Together with {the equalities} \eqref{eq:xirela}, \eqref{eq:feasviinfty}, $t_k\in\mathcal{K}$ and $\mathcal{K}$ being infinite, we can conclude that
		\begin{equation}\label{eq:xirela2}
			\lim_{k \rightarrow \infty} \frac{\left\|\xi^{(k)}\right\|}{\rho^{(k)}}=0.
		\end{equation}
		
		Similarly, we obtain  \eqref{eq:xiineq3} with $k\in\mathcal{K}$ replaced by $k\in\mathbb{N}_+$.
		This together with  {the fact} \eqref{eq:xirela2} and the lower boundedness of ${\mathcal{O}}\left(w^{(k+1)},v^{(k+1)},u^{(k+1)}\right)$ imply that
		$$\lim_{k \rightarrow \infty}\left\|u^{(k)}-\Psi\left(v^{(k)}\right)w^{(k)}-Ab^{(k)}\right\|=\lim_{k \rightarrow \infty}\left\|u^{(k+1)}-\Psi\left(v^{(k+1)}\right)w^{(k+1)}-Ab^{(k+1)}\right\|= 0.$$
		
		Using a similar method as that in the proof of statement (b), any accumulation point $(w^{*},b^{*},v^{*},u^{*})$ of $\{(w^{(k)},b^{(k)},v^{(k)},u^{(k)})\}$ is a KKT point of problem \eqref{eq:dnn33}.
		
		This proof is then completed by summarizing the above two cases.
	\end{proof}
	
	\subsection{Convergence Analysis of Algorithm \ref{alg:subbcd}}\label{sec:amconver}
	
	In this subsection, we prove the global convergence of Algorithm \ref{alg:subbcd}
	for solving subproblem \eqref{eq:almsub} and show that the condition \eqref{eq:almsubcondi1} 
	{always} hold.
	
	\begin{theorem}\label{thm:bcd}
		Let $\{(w^{(\jmath)},b^{(\jmath)},v^{(\jmath)},u^{(\jmath)})\}$ be the sequence generated by Algorithm \ref{alg:subbcd}. Then we have the following statements.
		
		(a) It holds that
		\begin{equation}\label{eq:ineqjmath}
			\begin{aligned}
				&\mathcal{L}_{\rho}\left(w^{(\jmath+1)},b^{(\jmath+1)},v^{(\jmath+1)},u^{(\jmath+1)};\xi\right)-\mathcal{L}_{\rho}\left(w^{(\jmath)},b^{(\jmath)},v^{(\jmath)},u^{(\jmath)};\xi\right)\\
				\leq& -\frac{\lb_w}{2}\left\|w^{(\jmath+1)}-w^{(\jmath)}\right\|^2-\frac{\tau_1}{2}\left\|u^{(\jmath+1)}-u^{(\jmath)}\right\|^2-\frac{\tau_1}{2}\sum_{n=1}^N\sum_{\ell=1}^{L-1}\left\|v_{n,\ell}^{(\jmath+1)}-v_{n,\ell}^{(\jmath)}\right\|^2.
			\end{aligned}
		\end{equation}
		
		(b) The sequence $\{\mathcal{L}_{\rho}(w^{(\jmath)},b^{(\jmath)},v^{(\jmath)},u^{(\jmath)};\xi)\}$ is convergent.
		
		(c) The sequence $\{(w^{(\jmath)},b^{(\jmath)},v^{(\jmath)},u^{(\jmath)})\}$ is bounded.
		
		(d) It holds that \begin{equation}\label{eq:WVuBbound}
			\lim_{\jmath\rightarrow \infty}\left\|w^{(\jmath+1)}-w^{(\jmath)}\right\|^2+\left\|b^{(\jmath+1)}-b^{(\jmath)}\right\|^2+\left\|v^{(\jmath+1)}-v^{(\jmath)}\right\|^2+\left\|u^{(\jmath+1)}-u^{(\jmath)}\right\|^2=0.
		\end{equation}
		
		(e) The sequence $\{(w^{(\jmath)},b^{(\jmath)},v^{(\jmath)},u^{(\jmath)})\}$ has at least one accumulation point, and
		any accumulation point $(w^*,b^*,v^*,u^*)$ of  $\{(w^{(\jmath)},b^{(\jmath)},v^{(\jmath)},u^{(\jmath)})\}$ is  a KKT point of \eqref{eq:almsub}.
	\end{theorem}
	
	\begin{proof}
		(a) Since 
		we have $S_{\ell}^{(\jmath)}\succeq \tau_1 I_{N_{\ell}+N_{\ell-1}}$ for all $\ell=2,3,\ldots,L$, we obtain that
		\begin{equation*}
			\begin{aligned}
				&\mathcal{P}\left(u^{(\jmath+1)},v^{(\jmath+1)};u^{(\jmath)},v^{(\jmath)}, \tau^{(\jmath)}\right)\\=&\frac{1}{2}\sum_{n=1}^N\sum_{\ell=2}^{L}\left\|\left(\begin{matrix}
					{v}^{(\jmath+1)}_{n,\ell-1}\\ {u}^{(\jmath+1)}_{n,\ell}
				\end{matrix}\right)-\left(\begin{matrix}
					{v}_{n,\ell-1}^{(\jmath)} \\ {u}_{n,\ell}^{(\jmath)}
				\end{matrix}\right)\right\|_{S_{\ell}^{(\jmath)}}^2+\frac{\tau_1}{2}\sum_{n=1}^N \left\|{u}^{(\jmath+1)}_{n,1}-u_{n,1}^{(\jmath)}\right\|^2
				\\
				\geq &\frac{\tau_1}{2}\left\|u^{(\jmath+1)}-u^{(\jmath)}\right\|^2+\frac{\tau_1}{2}\sum_{n=1}^N\sum_{\ell=1}^{L-1}\left\|v_{n,\ell}^{(\jmath+1)}-v_{n,\ell}^{(\jmath)}\right\|^2.
			\end{aligned}
		\end{equation*}
		Together with {the fact that} $({u}^{(\jmath+1)},{v}^{(\jmath+1)})$ being the global minimizer of \eqref{eq:bcdsub1}, we obtain \begin{equation}\label{eq:bcdineq}
			\begin{aligned}
				&\mathcal{L}_{\rho}\left(w^{(\jmath+1)},{b}^{(\jmath+1)},{v}^{(\jmath+1)},{u}^{(\jmath+1)};\xi\right)- \mathcal{L}_{\rho}\left(w^{(\jmath+1)},b^{(\jmath+1)},v^{(\jmath)},u^{(\jmath)};\xi\right)\\\leq&-\frac{\tau_1}{2}\left\|u^{(\jmath+1)}-u^{(\jmath)}\right\|^2-\frac{\tau_1}{2}\sum_{n=1}^N\sum_{\ell=1}^{L-1}\left\|v_{n,\ell}^{(\jmath+1)}-v_{n,\ell}^{(\jmath)}\right\|^2.
			\end{aligned}
		\end{equation}

		Since $\mathcal{R}_1$ is $\lb_w$-strongly convex with respect to $w$ \citep{beck2017first}, we derive from the definition of $\mathcal{L}_{\rho}$ and the updating rule \eqref{eq:bcdsubWb1} that 	
		$$\mathcal{L}_{\rho}\left(w^{(\jmath+1)},b^{(\jmath+1)},v^{(\jmath)},u^{(\jmath)};\xi\right)-\mathcal{L}_{\rho}\left(w^{(\jmath)},b^{(\jmath)},v^{(\jmath)},u^{(\jmath)};\xi\right)
		\leq -\frac{\lb_w}{2} \left\|w^{(\jmath+1)}-w^{(\jmath)}\right\|^2,$$
		which together with {the inequality} \eqref{eq:bcdineq} complete the statement (a).
		
		(b) The statement (a) together with {the inequality} $\mathcal{L}_{\rho}(w_{\mathrm{init}}^{(k)},b_{\mathrm{init}}^{(k)},v_{\mathrm{init}}^{(k)},u_{\mathrm{init}}^{(k)};\xi) \leq \theta$ and the definition of  $\mathcal{L}_{\rho}$, we obtain that
		\begin{equation}\label{eq:boundl}
			-\frac{1}{2\rho}\|\xi\|^2\leq \mathcal{L}_{\rho}\left(w^{(\jmath)},b^{(\jmath)},v^{(\jmath)},u^{(\jmath)};\xi\right)\leq \theta, \,\,\text{ for all }\jmath\in\mathbb{N}.
		\end{equation} Then, the non-increasing sequence $\{\mathcal{L}_{\rho}\left(w^{(\jmath)},b^{(\jmath)},v^{(\jmath)},u^{(\jmath)};\xi\right)\}$ is bounded and hence convergent.
		
		(c) Recall the definition of $\mathcal{L}_{\rho}$ and \eqref{eq:boundl}, we  obtain that $\lb_w\|w\|_{2,1}+\lb_v\|v\|^2\leq \frac{1}{2\rho}\|\xi\|^2+\theta$, then {the sequence} $\{(w^{(\jmath)},v^{(\jmath)})\}$ is bounded. Together with the definition of $\mathcal{L}_{\rho}$ and \eqref{eq:boundl}, it holds that $\{(b^{(\jmath)})\}$ is bounded
		.  Since $\mathcal{O}$ is bounded away from 0,  \eqref{eq:boundl} also yields that
		$$\frac{\rho}{2}\left\| \frac{\xi}{\rho}+u^{(\jmath)}-\Psi(v^{(\jmath)})w^{(\jmath)}-Ab^{(\jmath)}\right\|^2\leq \theta+\frac{1}{2\rho}\|\xi\|^2 $$
		for all $\jmath\in\mathbb{N}_+$, which together with {the boundedness of} $\{(w^{(\jmath)},b^{(\jmath)},v^{(\jmath)})\}$ imply that $\{u^{(\jmath)}\}$ is bounded. Hence, the sequence $\{(w^{(\jmath)},b^{(\jmath)},v^{(\jmath)},u^{(\jmath)})\}$ is bounded, which completes the statement (c).
		
		(d) From the statements (a), (b), (c), and $\tau_1>0$, we obtain that
		\begin{equation}\label{eq:WVubound}
			\lim_{\jmath\rightarrow \infty}\left\|w^{(\jmath+1)}-w^{(\jmath)}\right\|^2=0, \lim_{\jmath\rightarrow \infty}\left\|u^{(\jmath+1)}-u^{(\jmath)}\right\|^2=0, \lim_{\jmath\rightarrow \infty}\sum_{n=1}^N\sum_{\ell=1}^{L-1}\left\|v_{n,\ell}^{(\jmath+1)}-v_{n,\ell}^{(\jmath)}\right\|^2=0.
		\end{equation}
		
		From the updating rule of the $(w,b)$, we obtain from the KKT condition of \eqref{eq:bcdsubWb1} and the definition of $A,\Psi$ that for all $\ell\in[L]$, it holds that
		$$b^{(\jmath+1)}_{\ell}= \frac{1}{N}\sum_{n=1}^N \left(\frac{\xi_{n,\ell}}{\rho}+u_{n,\ell}^{(\jmath)}-W_{\ell}v_{n,\ell-1}\right).$$
		This together with the statement (c) and \eqref{eq:WVubound} imply that
		\begin{equation}
			\label{eq:bbound}
			\lim_{\jmath\rightarrow \infty}\left\|b^{(\jmath+1)}-b^{(\jmath)}\right\|^2=0.
		\end{equation}
		
		From the updating rule of the $(v,u)$, \eqref{eq:updateuv1} and \eqref{eq:updateuv2}, we obtain that $v^{(\jmath+1)}_{n,L}$ has a closed-form associated with $\xi_{n,L}$, $\rho$, $\tau_{L}^{(\jmath)}$, $u^{(\jmath)}_{n,L}$,  ${W}^{(\jmath+1)}_{L}$, $v^{(\jmath)}_{n,L-1}$, and  ${b}^{(\jmath+1)}_{L}$ for all $n\in[N]$. Together with {the facts} \eqref{eq:taudefine} and \eqref{eq:WVubound}, it holds that
		$$
		\lim_{\jmath\rightarrow \infty}\sum_{n=1}^N \left\|v_{n,L}^{(\jmath+1)}-v_{n,L}^{(\jmath)}\right\|^2=0.
		$$
		
		Using this and relations \eqref{eq:WVubound}, \eqref{eq:bbound}, we complete the statement (d).
		
		(e) The statement (c) yields that the sequence $\{(w^{(\jmath)},b^{(\jmath)},v^{(\jmath)},u^{(\jmath)})\}$ has at least one accumulation point. Let $\mathcal{J}$ be a index set of $\{(w^{(\jmath)},b^{(\jmath)},v^{(\jmath)},u^{(\jmath)})\}$ such that
		$$\lim_{\jmath\rightarrow\infty, \jmath\in\mathcal{J}}\left(w^{(\jmath)},{b}^{(\jmath)},v^{(\jmath)},u^{(\jmath)}\right)=(w^*,b^*,v^*,u^*).$$
		From the first-order optimality conditions for the updating schemes in Steps 1-2 of Algorithm \ref{alg:subbcd} and the constraint set $\mathcal{C}$ being separable with respect to $(v,u)$ block  and $(w,b)$ block, there exists $\mu^{(\jmath)}\in\R_+^{2m}$ such that
		\begin{equation}\label{eq:kktsubbcd}
			\left\{
			\begin{aligned}
				&0\in \partial_{(w,b)} \mathcal{L}_{\rho}\left(w^{(\jmath+1)},{b}^{(\jmath+1)},v^{(\jmath)},u^{(\jmath)};\xi\right)\\
				&0\in \partial_{(v,u)} \mathcal{L}_{\rho}\left(w^{(\jmath+1)},{b}^{(\jmath+1)},v^{(\jmath+1)},u^{(\jmath+1)};\xi\right)\\&\quad\quad+\nabla_{(v,u)} \left(\mathcal{P}\left(u^{(\jmath+1)},v^{(\jmath+1)};u^{(\jmath)},v^{(\jmath)}, \tau^{(\jmath)}\right) +\left(\mu^{(\jmath)}\right)\zz\mathcal{C}\left(v^{(\jmath+1)},u^{(\jmath+1)}\right)\right)\\
				&	\left(\mu^{(\jmath)}\right)\zz  \mathcal{C}\left(	v^{(\jmath+1)},u^{(\jmath+1)}\right)=0.
			\end{aligned}\right.
		\end{equation}
		
		Let $r_{\jmath}:=\|\mu^{(\jmath)}\|_{\infty}$. Suppose that $\left\{r_{\jmath}\right\}$ is unbounded.
		Without loss of generality, we assume that as $\jmath\rightarrow\infty, \jmath\in\mathcal{J}$, it holds that
		$
		\frac{\mu^{(\jmath)}}{r_{\jmath}}\rightarrow \bar{\mu}^*.
		$
		It then follows the fact $\mu^{(\jmath)}\geq 0$ for all $\jmath\in\mathbb{N}$ that $\| \bar{\mu}^*\|_{\infty}=1$ and $\bar{\mu}^*\geq 0$.
		
		Dividing by $r_{\jmath}$ and taking the limit $\jmath\rightarrow\infty, \jmath\in\mathcal{J}$ on both sides of \eqref{eq:kktsubbcd}, we obtain $0=\nabla_{(v,u)} (\mu^*)\zz \mathcal{C}( v^*,u^*)$ which results from
		the locally Lipschitz continuity of $\mathcal{L}_{\rho}, \mathcal{C}, \mathcal{P}, \nabla \mathcal{C}$ and the statements (a)--(d). Together with {the equality} $\| \bar{\mu}^*\|_{\infty}=1$ and Lemma \ref{lem:mfcqsys}, it leads to contradiction. Thus, $\left\{r_{\jmath}\right\}$ is bounded as desired. Without loss of generality, we assume that as $\jmath\rightarrow\infty, \jmath\in\mathcal{J}$, it holds that
		\begin{equation}
			\mu^{(\jmath)}\rightarrow \bar{\mu}^*.
		\end{equation}
		Similarly, we can obtain that $\bar{\mu}^*\geq 0$, since $\mu^{(\jmath)}\geq 0$ for all $\jmath\in\mathbb{N}$.
		
		Again, taking the limit $\jmath\rightarrow\infty, \jmath\in\mathcal{J}$  on both sides of \eqref{eq:kktsubbcd}, we finally arrive at \\$(\bar{\mu}^*)\zz  \mathcal{C}( v^*,u^*)=0$ and
		\begin{equation*}
			\begin{aligned}
				&0\in \partial_{(w,b)} \mathcal{L}_{\rho}(w^*,b^*,v^*,u^*;\xi),\,0\in \partial_{(v,u)} \left(\mathcal{L}_{\rho}(w^*,b^*,v^*,u^*;\xi)+ (\bar{\mu}^*)\zz\mathcal{C}( v^*,u^*)\right).
			\end{aligned}
		\end{equation*}
		This completes the proof.
	\end{proof}

	\begin{remark}
		The statement (a) of Theorem \ref{thm:bcd} shows that Algorithm \ref{alg:subbcd}
		yields a monotonic nonincreasing function value sequence $\{\mathcal{L}_{\rho}\left(w^{(\jmath)},b^{(\jmath)},v^{(\jmath)},u^{(\jmath)};\xi\right)\}$ for fixed $\rho$  and $\xi$. Together with the selected initial guess, we can conclude that condition \eqref{eq:almsubcondi1} always holds. Meanwhile, the statement (e) of Theorem \ref{thm:bcd} guarantees an inexact stationarity condition \eqref{eq:almsubcondi2} can hold by certain iterate. Therefore, the inner iteration, Algorithm \ref{alg:subbcd}, {is 
			qualified to be
			Step 1 of} the outer iteration, {namely,} Algorithm \ref{alg:alm}.
	\end{remark}
	
		

	\section{Numerical Experiments}\label{sec:numeri}
	
	In this section, we evaluate the numerical
	performance of IALAM for training
	the sparse leaky ReLU network with group sparsity 
	through comparing with some state-of-the-art SGD-based approaches. All the numerical experiments are conducted under MATLAB R2018b with windows 7 on a desktop with 3.4 GHz Inter Core i7-6700 CPU and 16 GB RAM.
	
	\subsection{Implementation Details}
	
	{\bf{Algorithm parameters.}}
	For our IALAM, we set $\eta_1=0.99$, $\eta_2=\frac{5}{6}$, $\eta_3=0.01$, $\eta_4=\frac{2}{3}$, $\epsilon_0=0.1$, $\rho^{(0)}=1/N$, $\xi^{(0)}=0$, and $\gamma=2L$.
	{It is worthy of mentioning that although $\eta_3=0.01$ does not satisfy the requirement
		in Theorem \ref{thm:bcd} to guarantee the global convergence of
		Algorithm \ref{alg:alm}, this choice always yields better performance than $\eta_3>1$ in practice.}
	In the inner iteration, {subproblem} \eqref{eq:bcdsubWb1}
	is solved by the proximal gradient method
	\citep{dai2005bb}. {
		We set the initial proximal parameter in
		subproblem \eqref{eq:bcdsub1} as $\tau_1=\frac{1}{10N}$,
		and update $\tau_{\ell}^{(\jmath)}$, for all $\jmath=2,3,\ldots,L$, by formulation (54).}

{\bf{Stopping criterion and initial guess.}}
Except for otherwise mentioned, we terminate our algorithm whenever $\epsilon_{k}<10^{-6}$ or $\rho^{(k)}>10^{3} \rho^{(0)}$. 
For all $\ell\in[L]$, the variables $W_{\ell}^{(0)}$ are randomly generated
by $W_{\ell}^{(0)} = \mathrm{randn}(N_{\ell},N_{\ell}-1)/N$, where $\mathrm{randn}(n, p)$
stands for an $n \times p$ randomly generated matrix under the standard Gaussian distribution. 		
Let
{$b^{(0)}=0$}, 
$v^{(0)}_{n,0}=x_n$, $u^{(0)}_{n,\ell}=W_{\ell}^{(0)}v^{(0)}_{n,\ell-1}$ and {$v^{(0)}_{n,\ell}=\sigma(u^{(0)}_{n,\ell})$} for all $n\in[N]$ and $\ell\in[L]$.

{\bf{Algorithms in Comparison.}}
For comparison, we choose a few state-of-the-art
SGD-based approaches, including the Adam~\citep{kingma2014adam}, the Adamax~\citep{kingma2014adam}, the Adadelata~\citep{zeiler2012adadelta}, the Adagrad~\citep{duchi2011adaptive}, the AdagradDecay~\citep{duchi2011adaptive}, and the Vanilla SGD~\citep{cramir1946mathematical} with batch-size (Vanilla SGD (batch)). The MATLAB codes of these SGD-based approaches are downloaded from the SGD Library~\citep{JMLR:v18:17-632}. We also include ProxSGD
\citep{yang2020proxsgd}.
These approaches directly solve problem \eqref{eq:odnn}. All of these algorithms are run under their defaulting settings. The batch-size of these methods is set to $\lceil\sqrt{N}\rceil$. We terminate these methods whenever the epoch (i.e., ``Iteration$\times$batch-size$/N$'') reaches 1000 unless otherwise stated.

{  {\bf{Model parameters (hyperparameters).}}
	
	We introduce the model parameters of problem \eqref{eq:dnn33} in the tests unless otherwise statement, which include $\alpha=0.01$, $\lb_w=\frac{1}{N}$ $\lb_v=\frac{1}{100N}$, and $\beta=\frac{1}{N}e_m$. Specifically, results with various values of constant $\alpha$ and vector $\beta$ are shown in Figures \ref{fig:alphacomparsion} and \ref{fig:perpro}, and Table \ref{tab:betacomparsion}, respectively.
}

{\bf{Test problems.}}
The number of test samplings $N_{\text {test}}$ is set to be $\lceil N/5 \rceil$.

There are two classes of test problems. 
The first class of test problems are generated randomly. We 
construct the training data sets and test data sets with a similar way as that proposed by \cite{cui2020multicomposite}, i.e., 	
$$y_n=\sigma(W_L\sigma({\cdots\sigma(W_1 x_n+b_1)+b_2\cdots})+b_L)+\tilde{y}_n,$$
for all $n\in[N+N_{\text {test}}]$, where $x_n \sim \mathcal{N}(\zeta, \Sigma_{0}^{T} \Sigma_{0})$, $\tilde{y}_n= \epsilon_y \mathrm{randn}(1,1)$. Here, the parameter $\epsilon_y=0.05$ is to control the noise level, $\zeta=\mathrm{randn}(N_0,1)$, and $\Sigma_{0}=\mathrm{randn}(N_0,1)$. 

The second class of test problems is the classification problem on the MNIST \citep{LeCun1998mnist} data set, consisting of 10-classes handwritten digits with the size $28 \times 28$, namely, $N_{0}=784.$ In practice, we randomly pick up data entries from each class of MNIST under uniform distribution. Since there are ten classes in the MNIST, we take $N_L=10$.

\textbf{Output evaluation.}
Finally, we introduce how to evaluate the performance of various approaches. We record the measurements including the training error, the test error, 
the first feasibility violation, the second feasibility violation, and the KKT violation, which are
{denoted} by
$$\text{TrainErr}=\frac{1}{N}\sum_{n=1}^{N} \left\|\sigma(W_L\sigma({\cdots\sigma(W_1 x_n+b_1)+b_2\cdots})+b_L)-y_n\right\|^2,$$
$$\text{TestErr}=\frac{1}{N}\sum_{n=N+1}^{N+N_{\mathrm{test}}} \left\|\sigma(W_L\sigma({\cdots\sigma(W_1 x_n+b_1)+b_2\cdots})+b_L)-y_n\right\|^2,$$
$$\text{FeasVi1}=\frac{1}{N}\sum_{n=1}^N\sum_{\ell=1}^L\left\|v_{n,\ell}-\sigma(u_{n,\ell})\right\|^2, \text{ FeasVi2 }=\frac{1}{N}\sum_{n=1}^N\sum_{\ell=1}^L\left\|u_{n,\ell}-(W_{\ell}v_{n,\ell-1}+b_{\ell})\right\|^2,$$
and
\begin{equation*}
	\text{KKTVi}=
	\operatorname{dist}(0, \partial \mathcal{L}_{\rho^{(k-1)}}(w^{(k)},b^{(k)},v^{(k)},u^{(k)};\xi^{(k-1)})+\mathcal{N}_{\Omega_3}(w^{(k)},b^{(k)},v^{(k)},u^{(k)})) +\frac{1}{2}\text{FeasVi2},
\end{equation*}
respectively, and the average feasibility violation FeasVi=(FeasVi1+FeasVi2)$/\overline{N}$. Time is the CPU time in (minutes: seconds).  For the classification task, we also record  the classification accuracy for the training data, ``Accuracy'', and test data, ``TestAcc'', respectively.  

\subsection{Numerical Performance of IALAM}

In this subsection, we investigate the numerical performance of IALAM in solving problems with both randomly generated data sets and MNIST.

\subsubsection{Solving Problems with Different Layers}

We test IALAM in solving problem \eqref{eq:dnn33} with different layers. Figure \ref{fig:normalialam} shows the  performance of IALAM in solving problem with synthetic data set, where we set $N=500$, $L=4$, $N_0=5$, $N_1=4$, $N_2=4$, $N_3=3$ and $N_4=1$.  We can learn from Figure \ref{fig:normalialam} that (i) the training error and the test error decrease in the same order; (ii) the feasibility violations and KKT violation reduce {oscillatorily to zero.} 

Next, we demonstrate the numerical behavior of IALAM in solving problems with the MNIST data set, where $N=60000$, $N_{\text {test}}=10000$ 
and the number of hidden layers up to four (averaged over 100 simulations). 
We can learn from Table \ref{tab:mnistialam} and Table \ref{tab:mnistialam-normal} that IALAM
works well in dealing with classification data set MNIST with different layers. 

\begin{figure}[H]	
	\centering
	\setcounter{subfigure}{0}
	\subfloat[TrainErr/TestErr]{\includegraphics[width=50mm]{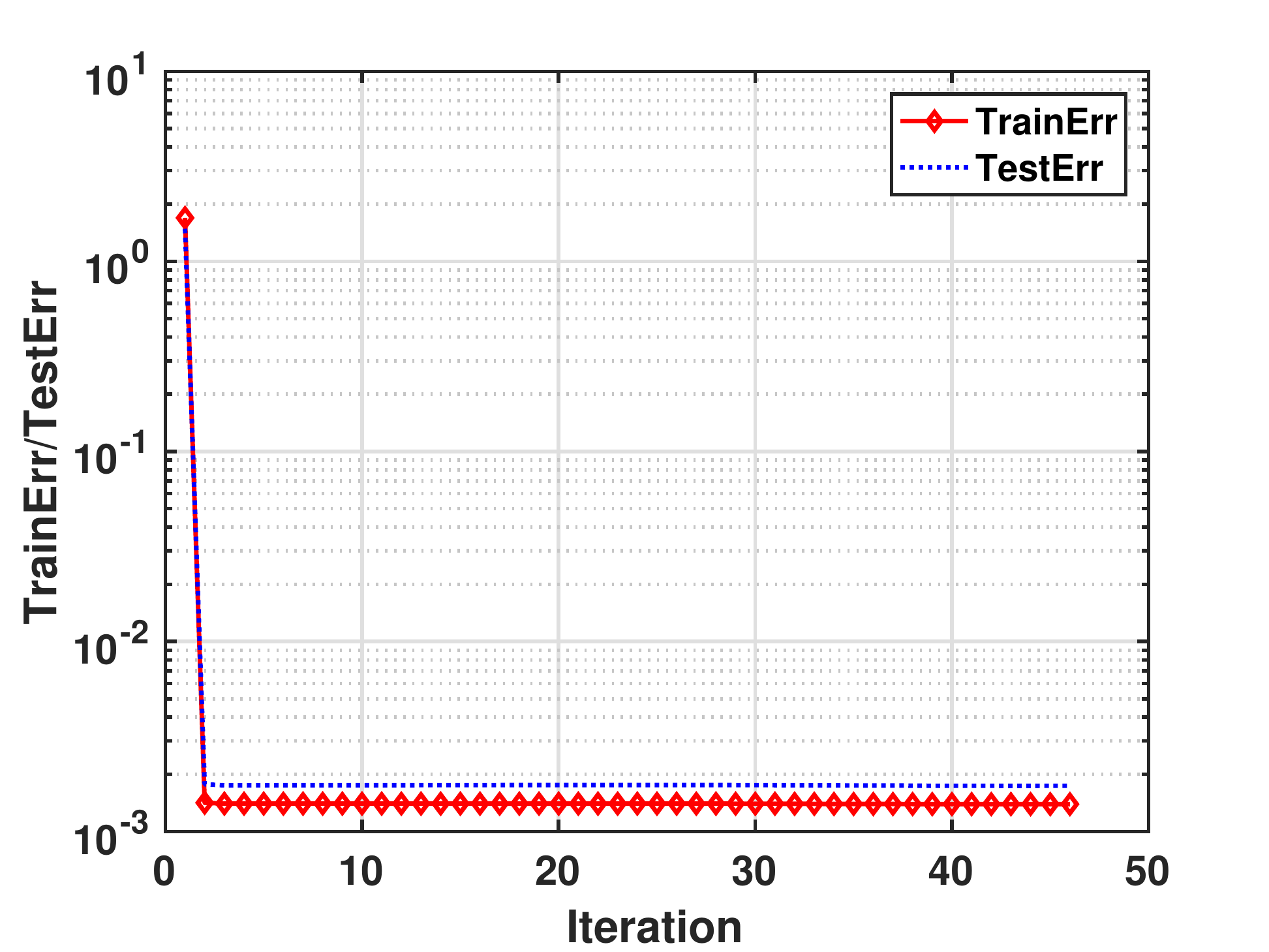}}
	\subfloat[FeasVi1/FeasVi2]{\includegraphics[width=50mm]{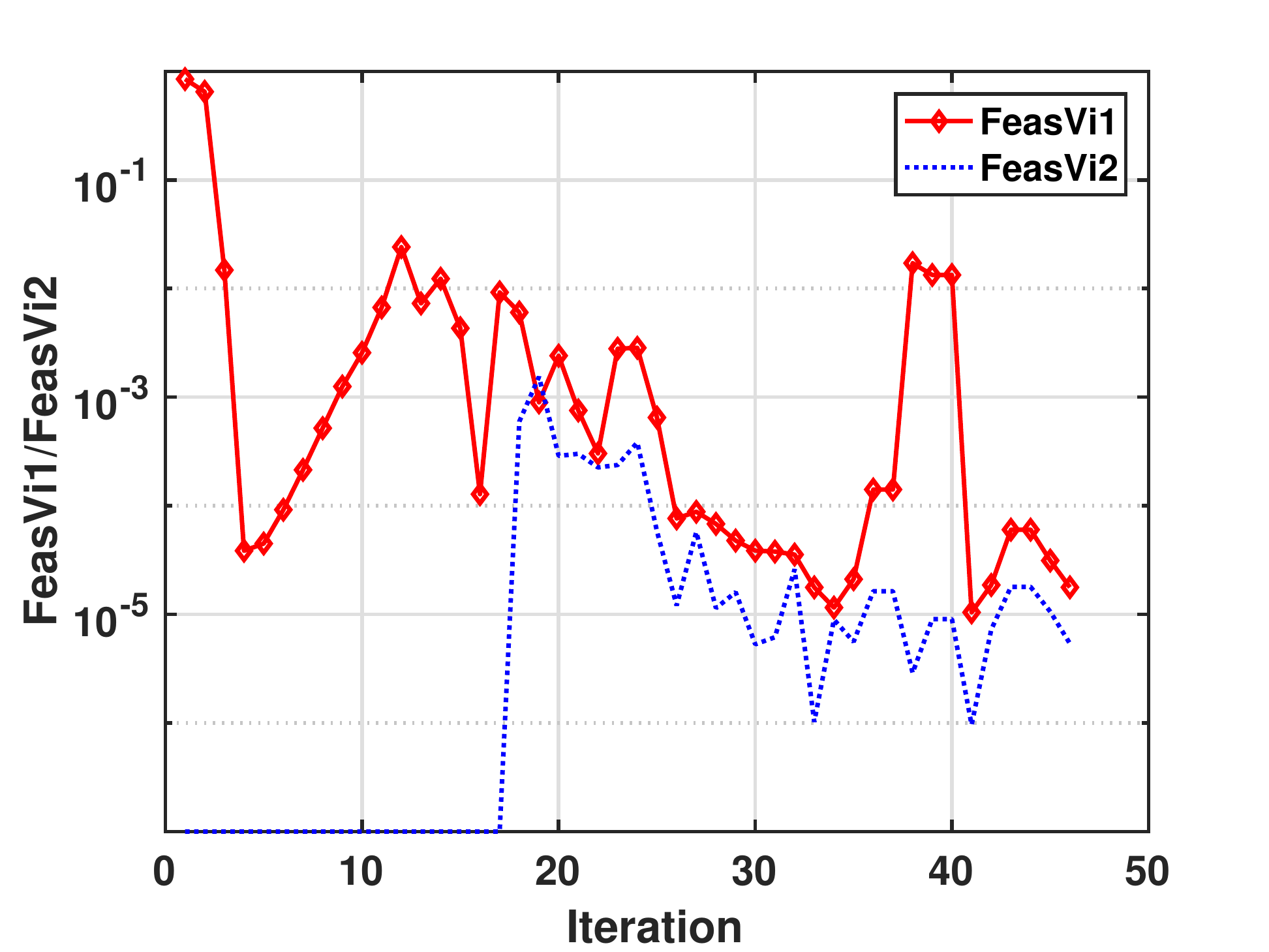}}
	\subfloat[KKTVi]{\includegraphics[width=50mm]{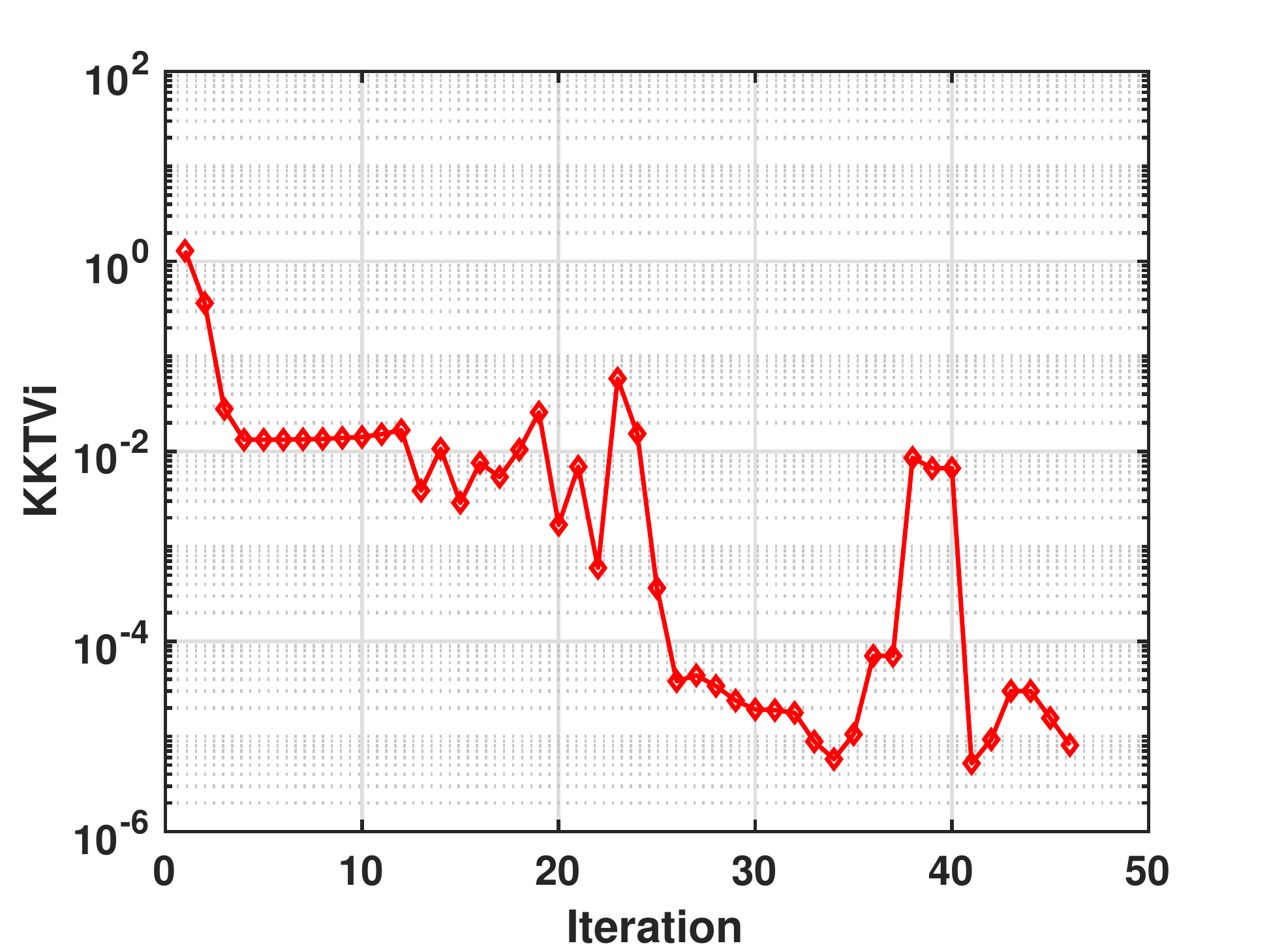}}
	\caption{Algorithm performance of IALAM on the synthetic data set.}\label{fig:normalialam}
\end{figure}

\begin{table}[htbp]
	\setlength\tabcolsep{3.5pt}
	\centering
	\caption{Numerical Results of IALAM with up to four layers.}\label{tab:mnistialam}
	\begin{tabular}{ccc|ccccccccc}
		$N_1$ & $N_2$ & $N_3$ & Iter & Time & FeasVi1 & FeasVi2 & TrainErr & Accuracy &  TestAcc  \\
		\hline
		500 & --&  --  &  138 & 47:06 & 2.00e-5 & 9.95e-9 & (3.98$\pm$0.24)e-2 & 0.981$\pm$0.002 &0.974$\pm$0.002\\
		200 & 100& -- & 152 & 21:58 & 6.97e-5 & 4.06e-5& (3.35$\pm$0.41)e-2 & 0.980$\pm$0.004 &0.976$\pm$0.004 \\					
		500 & 200& --  &  125 & 21:33 & 1.07e-4 & 1.97e-5  & (1.94$\pm$0.62)e-2 & 0.989$\pm$0.003 &0.983$\pm$0.003 \\					
		100 & 100& 100 & 109 & 6:37 & 1.19e-4 & 7.49e-5 & (5.23$\pm$0.86)e-2 & 0.947$\pm$0.006   &0.943$\pm$0.005 \\
		200 & 400& 200 &  112 & 20:10 & 8.68e-5 & 5.66e-5  & (4.42$\pm$1.27)e-2 & 0.960$\pm$0.010 & 0.959$\pm$0.008 \\
		800 & 400& 200 &  95 & 31:39 & 9.82e-5 & 3.74e-5 & (5.21$\pm$1.31)e-2 & 0.960$\pm$0.012 &0.958$\pm$0.008 \\	
	\end{tabular}
\end{table}

\begin{table}[htbp]
	\setlength\tabcolsep{4mm}
	\centering
	\caption{Numerical Results of IALAM with up to four layers among 100 times.}\label{tab:mnistialam-normal}
	\begin{tabular}{c|ccc|ccc|ccc|}
		& $N_1$ & $N_2$ & $N_3$ & $N_1$ & $N_2$ & $N_3$& $N_1$ & $N_2$ & $N_3$\\
		& 800 & 400& 200 & 500 & 200& --  &	500 & --&  --  \\
		\hline
		FeasVi1 & \multicolumn{3}{|c|}{[7.89e-7, 8.78e-4]}  &  \multicolumn{3}{|c|}{[1.67e-6, 3.09e-4] } &  \multicolumn{3}{|c|}{ [7.32e-7, 1.25e-4]}   \\
		FeasVi2 & \multicolumn{3}{|c|}{[1.06e-8, 4.29e-4]} &  \multicolumn{3}{|c|}{[3.86e-9, 1.05e-4]}&  \multicolumn{3}{|c|}{[7.75e-9, 1.87e-8]}\\
		FeasVi& \multicolumn{3}{|c|}{[6.67e-10, 9.33e-7] } &  \multicolumn{3}{|c|}{[1.20e-9, 9.42e-5]}&  \multicolumn{3}{|c|}{[5.30e-10, 8.90e-8]}\\
		TrainErr& \multicolumn{3}{|c|}{ [1.94e-2, 6.63e-2] } &  \multicolumn{3}{|c|}{[1.10e-2, 3.97e-2]}&  \multicolumn{3}{|c|}{[3.60e-2, 4.61e-2] }\\
		Accuracy& \multicolumn{3}{|c|}{ [0.946, 0.986] } &  \multicolumn{3}{|c|}{[0.980, 0.996]}&  \multicolumn{3}{|c|}{[0.976, 0.985]}\\
		TestAcc& \multicolumn{3}{|c|}{[0.943, 0.982]} &  \multicolumn{3}{|c|}{[0.978, 0.990]}&  \multicolumn{3}{|c|}{[0.969, 0.979]} \\
	\end{tabular}
\end{table}

\subsubsection{Investigating the Model Parameters in Sparse leaky ReLU Network}

In this subsection, we first study the numerical performance of IALAM in solving
problem \eqref{eq:dnn33} with {various leaky ReLU parameters $\alpha$}.
Our test is based on the MNIST data set with $N=60000$ and $N_{\mathrm{test}}=10000$ and a fixed initialization point. 
We can learn from Figure \ref{fig:alphacomparsion} that (i) IALAM can be extended to {training the ReLU network,}
i.e. $\alpha=0$;  (ii) a small positive $\alpha$ often leads to better performance than $\alpha=0$, but {further increasing of $\alpha$ yields worse and worse performance.}

%


%

\begin{figure}[htbp]	
	\centering
	\setcounter{subfigure}{0}
	\subfloat[TrainErr]{\includegraphics[width=50mm]{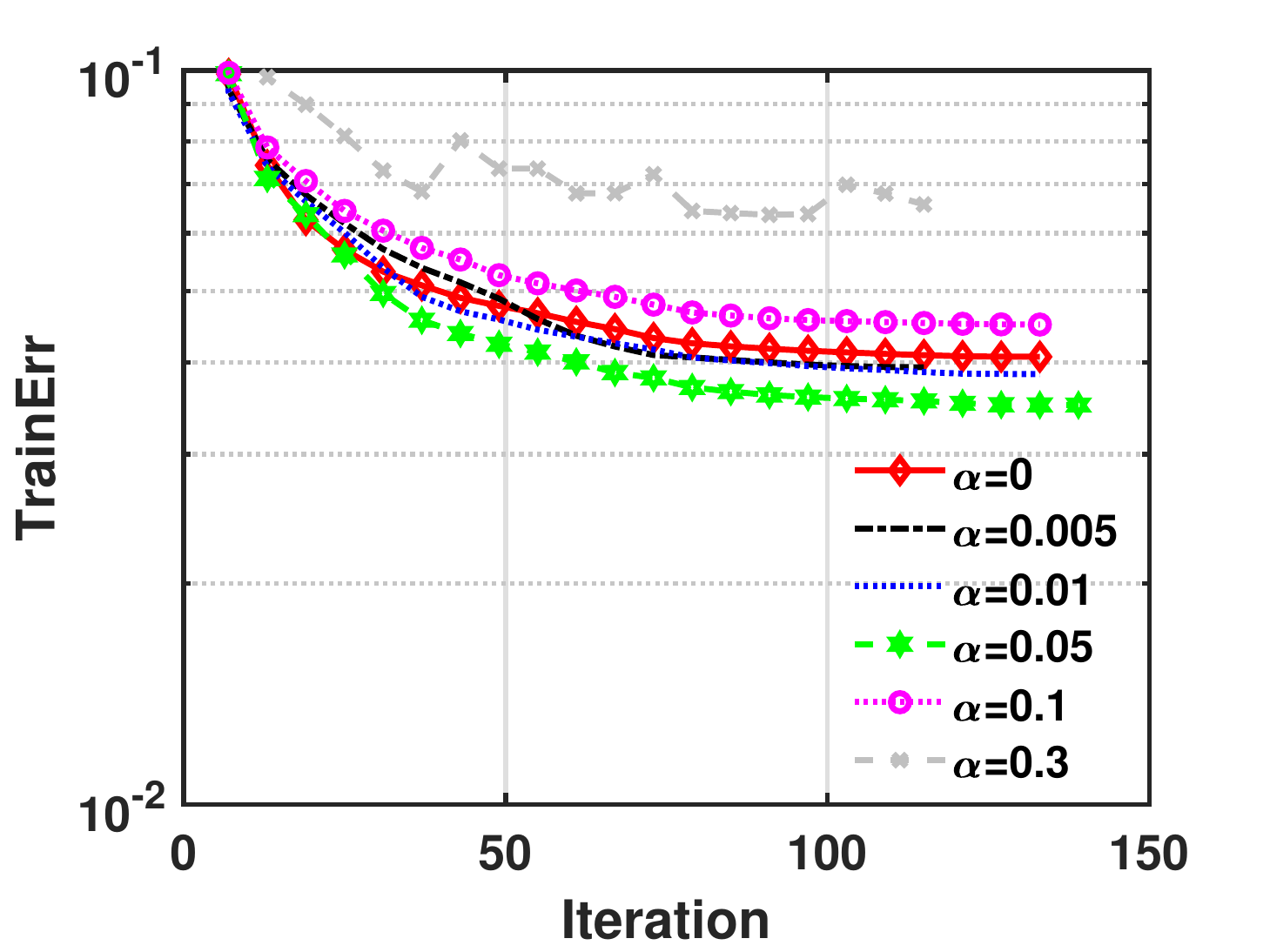}}
	\subfloat[Accuracy]{\includegraphics[width=50mm]{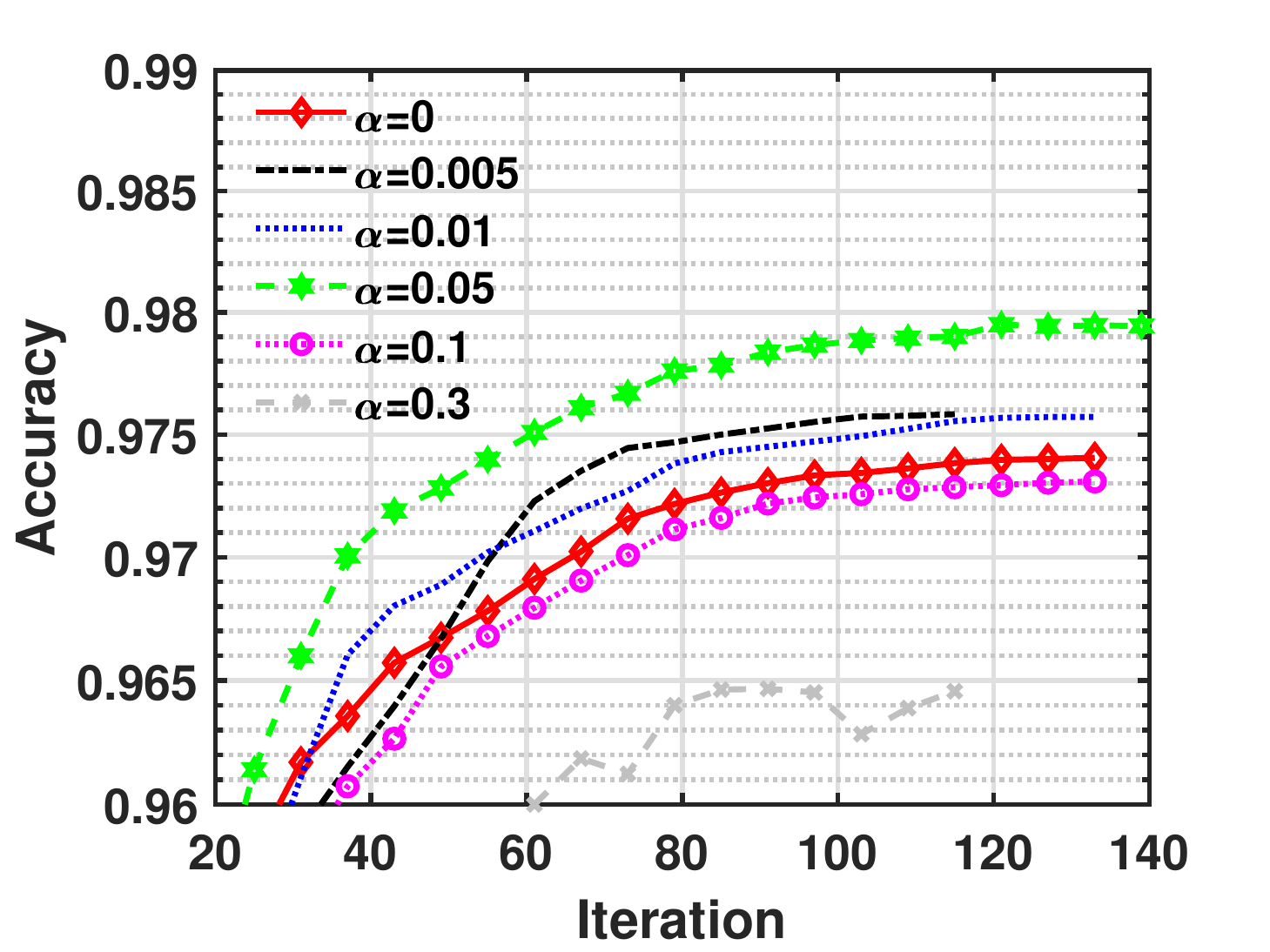}}
	\subfloat[TestAcc]{\includegraphics[width=50mm]{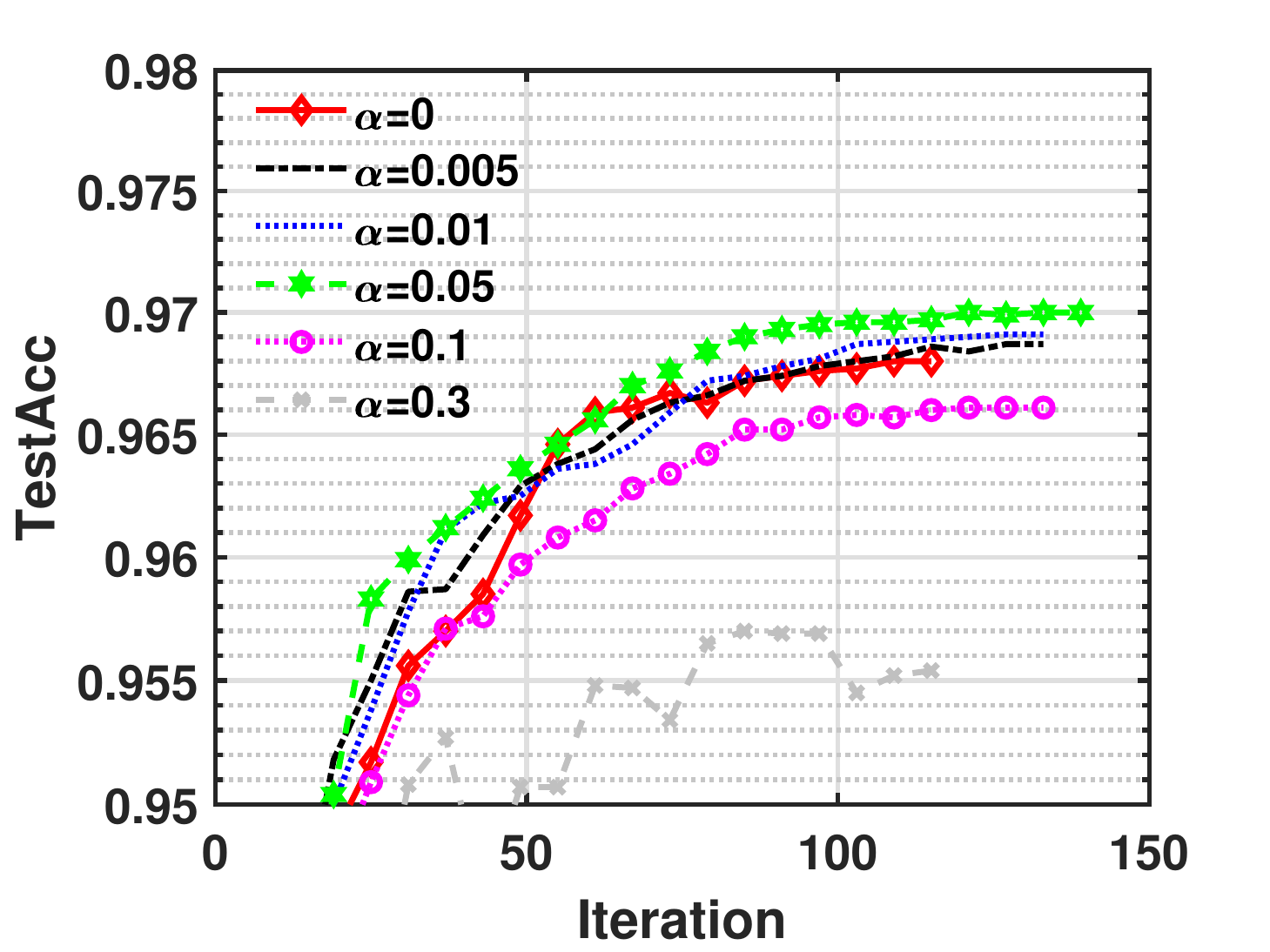}}\\
	\subfloat[TrainErr]{\includegraphics[width=50mm]{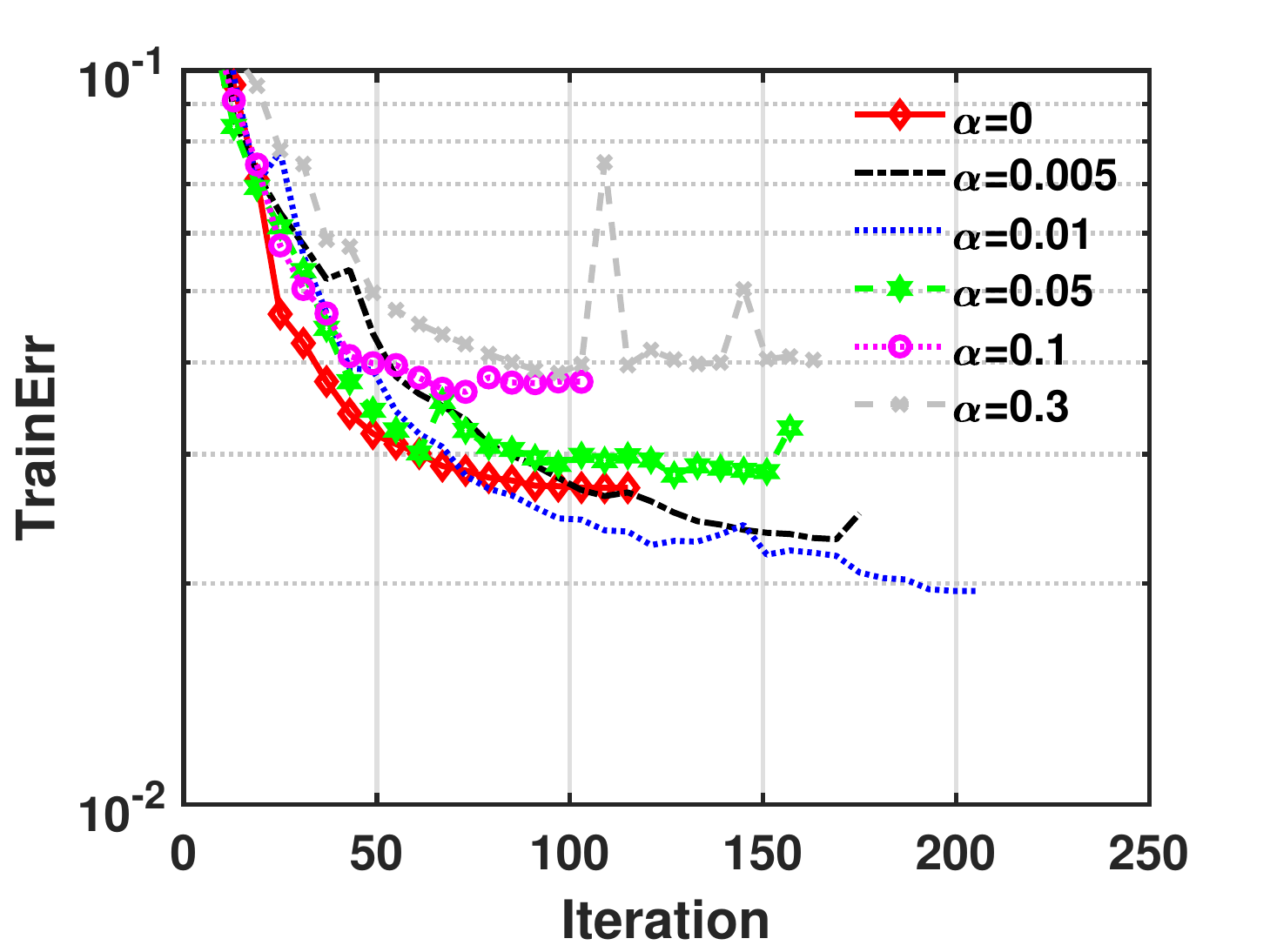}}
	\subfloat[Accuracy]{\includegraphics[width=50mm]{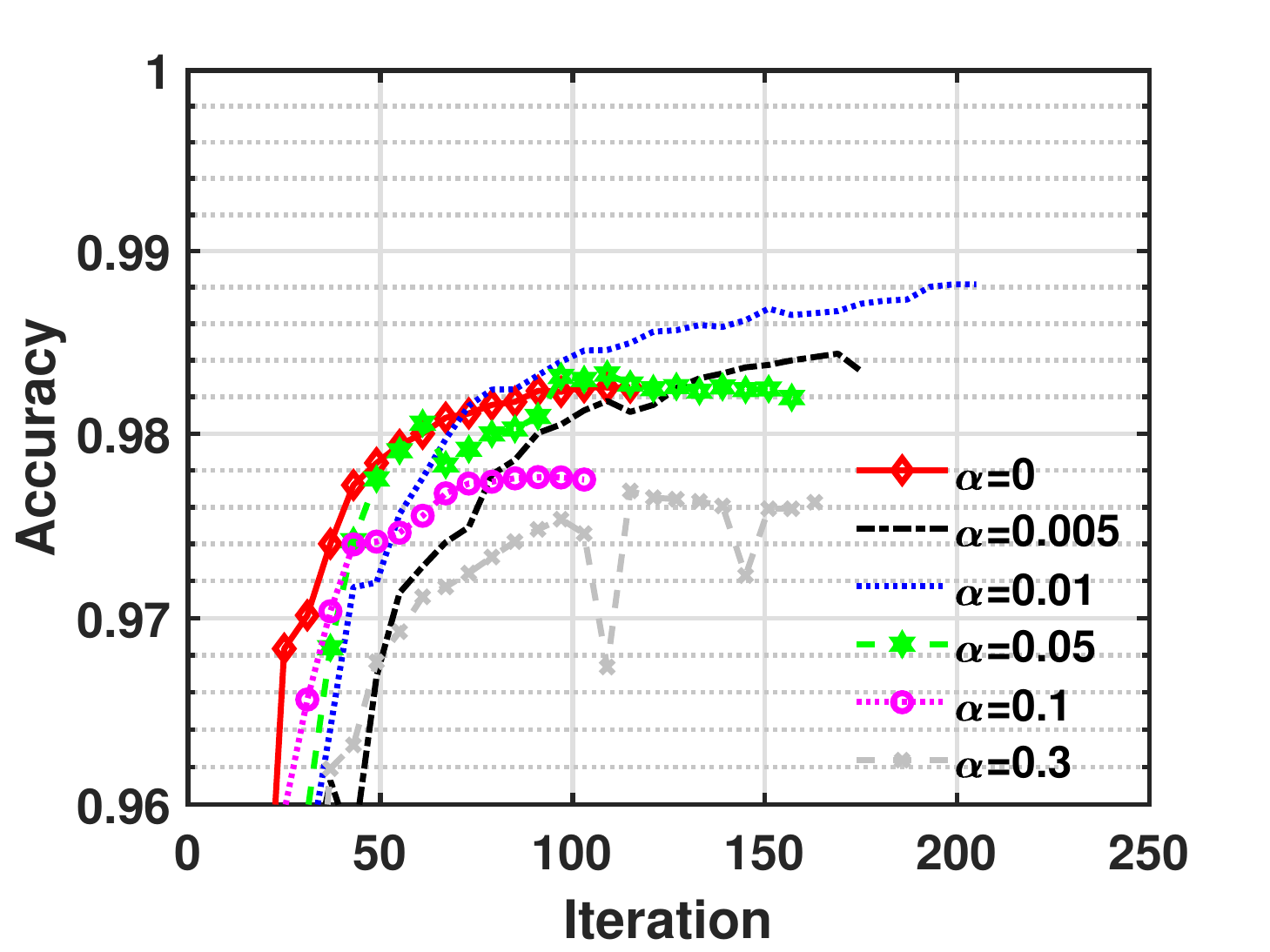}}				
	\subfloat[TestAcc]{\includegraphics[width=50mm]{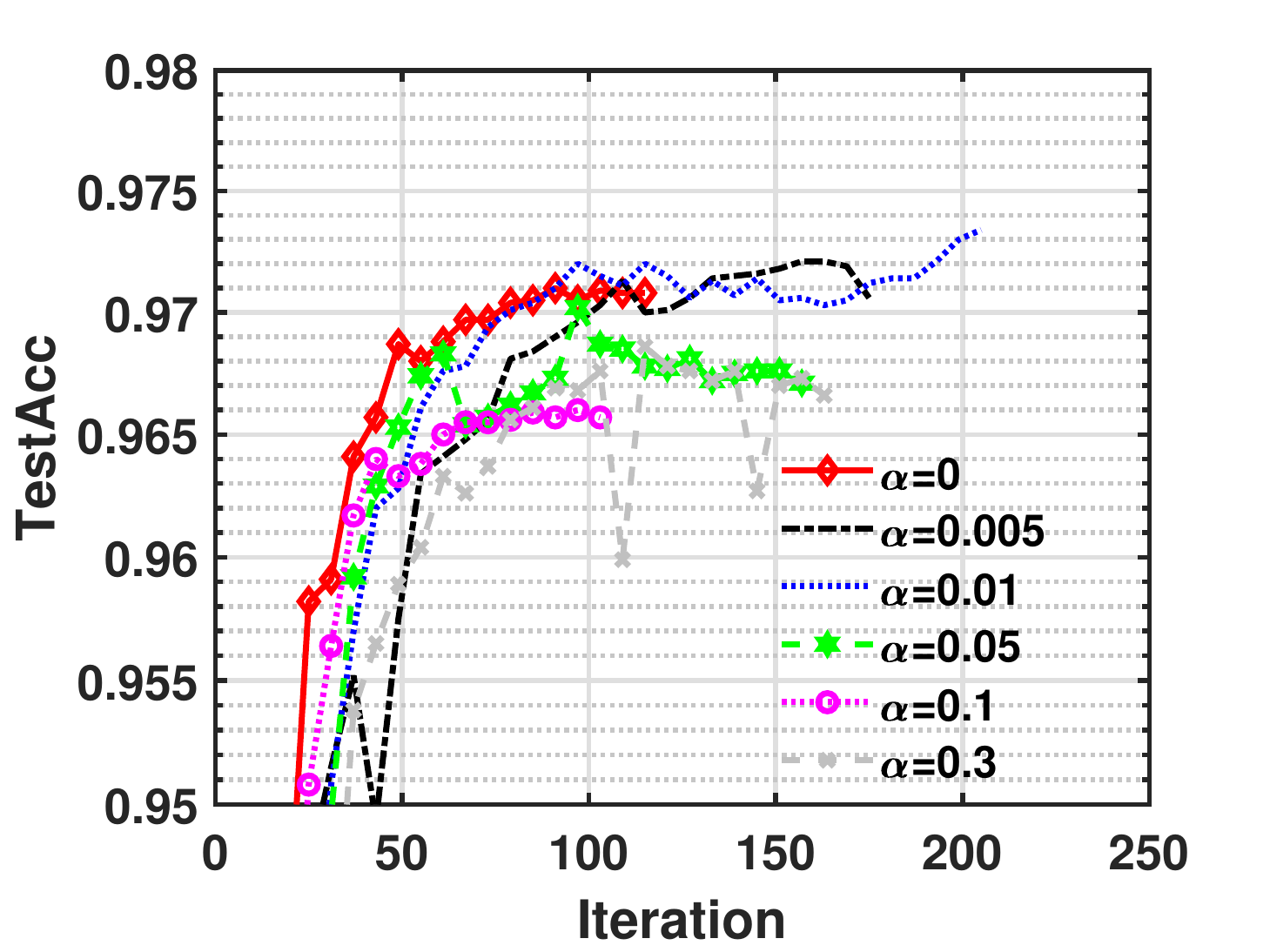}}\\
	\subfloat[TrainErr]{\includegraphics[width=50mm]{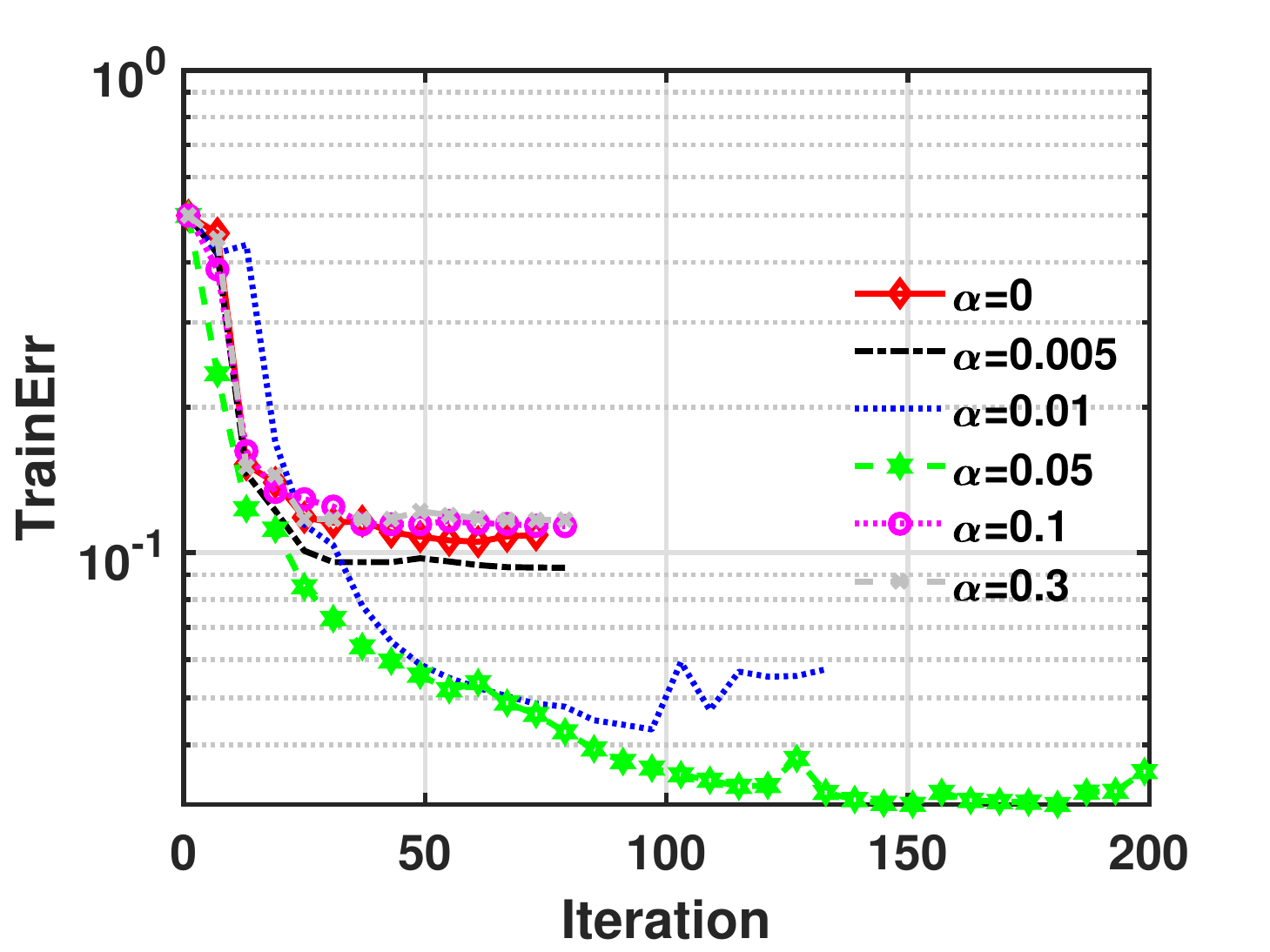}}
	\subfloat[Accuracy]{\includegraphics[width=50mm]{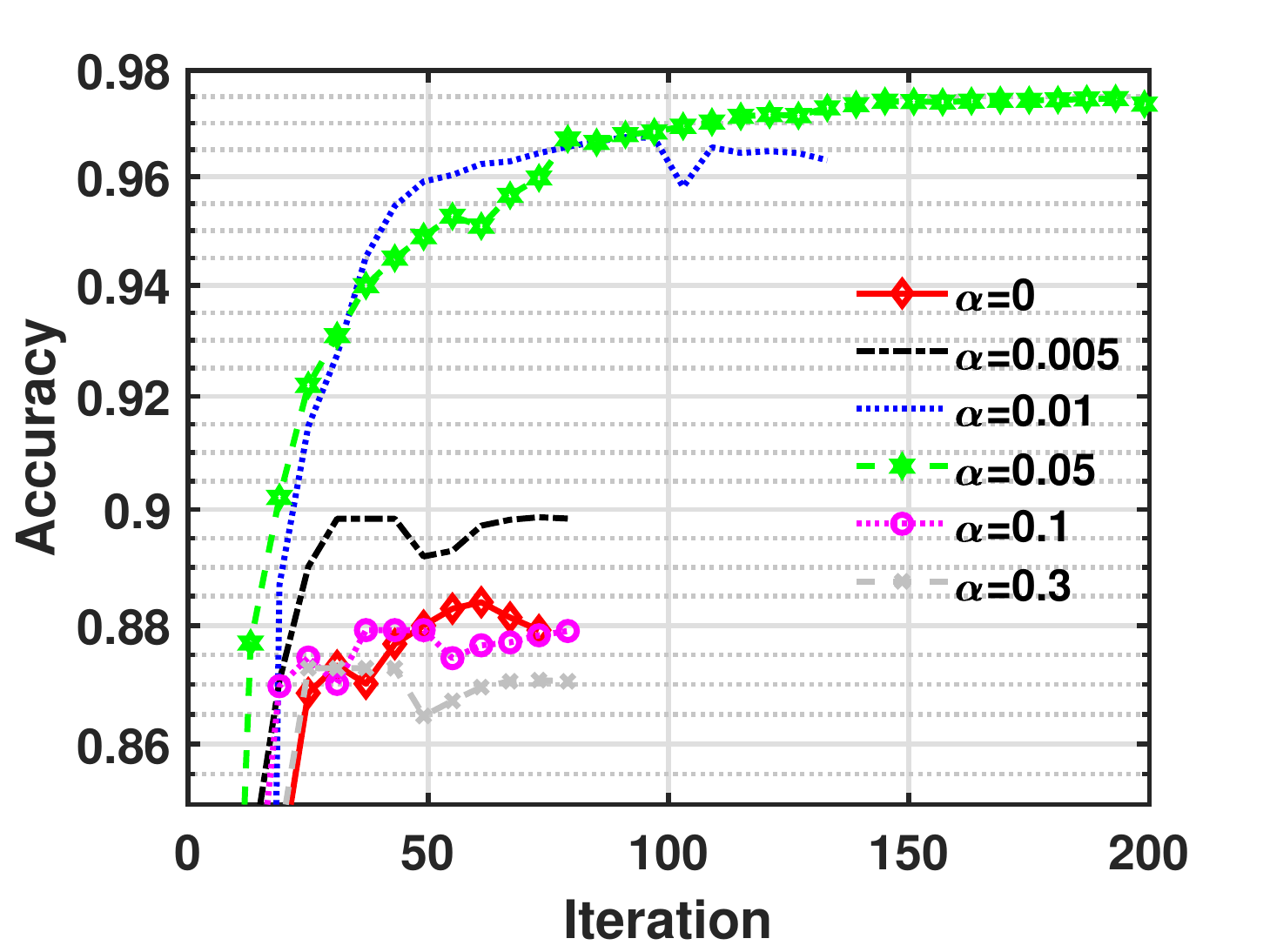}}
	\subfloat[TestAcc]{\includegraphics[width=50mm]{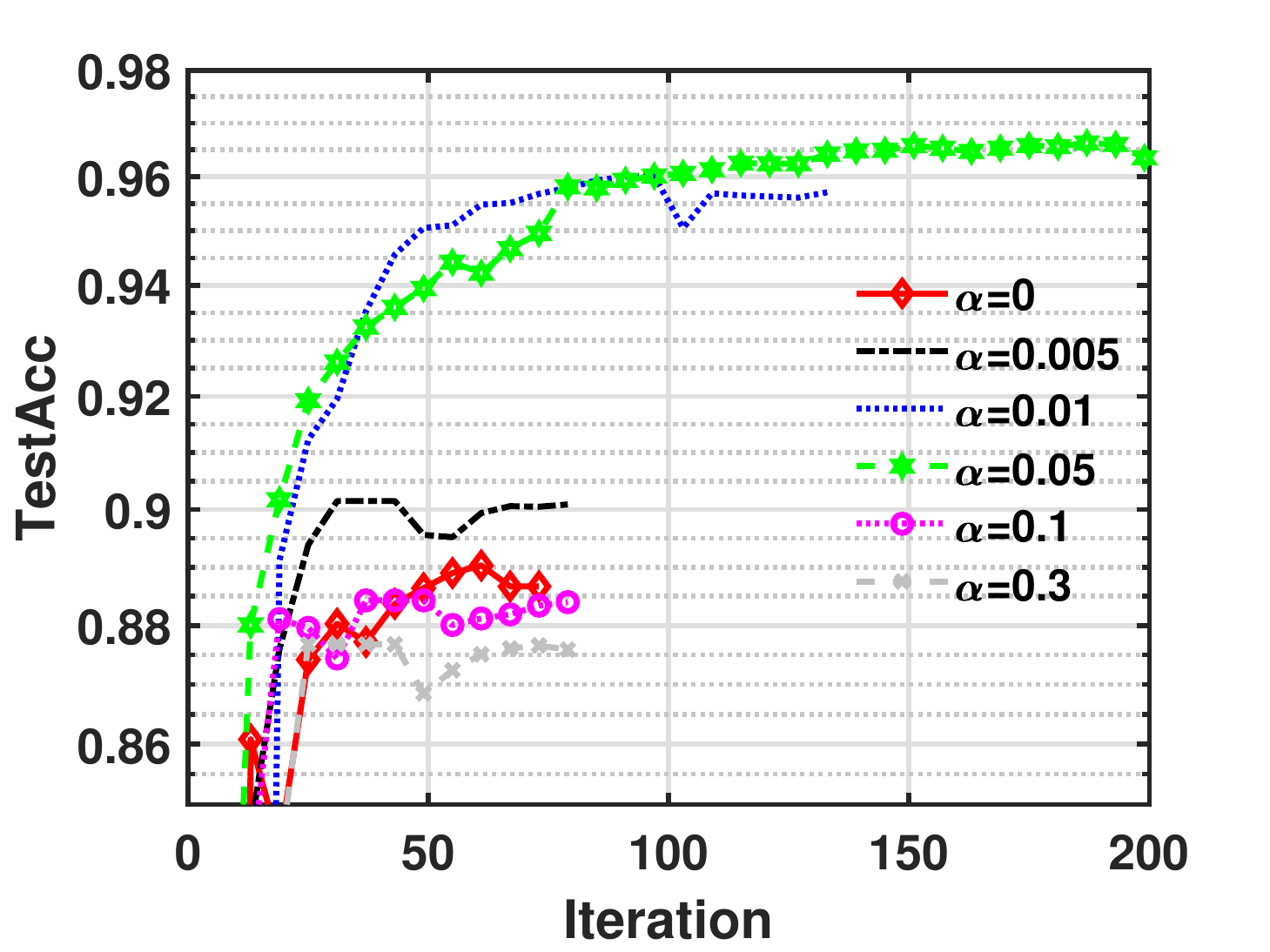}}\\
	\caption{Comparisons among IALAM with $N=60000$, different networks, $\alpha$ and (a)--(c): $N_1=500, L=2$; (d)--(f): $N_1=200,N_2=100, L=3$; (g)--(i): $N_1=100$, $N_2=100$, $N_3=100$,  $L=4$.}
	\label{fig:alphacomparsion}
	
\end{figure}

{ Then, we study the numerical performance of IALAM in solving
	problem \eqref{eq:dnn33} with different penalty parameters $\beta:=\bar{\beta} e_m$, $\bar{\beta}>0$ and a fixed initialization point.
	Our test is also based on the MNIST data set with $N=60000$ and $N_{\mathrm{test}}=10000$. 
	We can learn from Table \ref{tab:betacomparsion}  that (a) the bigger $\bar{\beta}$ always leads to slower convergence; (b) $\bar{\beta}=\frac{1}{N}$ performs the best among $\{1/N,1/10N,10/N,100/N\}$.} Hence, {we choose $\frac{1}{N} e_m$ as the default value of penalty parameter $\beta$ in IALAM.}

\begin{table}[H]
	\setlength\tabcolsep{3.5pt}
	\centering
	\caption{Comparisons among IALAM with vector ${\beta}=\bar{\beta} e_m$.}\label{tab:betacomparsion}
	\begin{tabular}{cccc|cccccccc}
		$N_1$ & $N_2$ & $N_3$ & $\bar{\beta}$ & Iter & FeasVi1 & FeasVi2 & FeasVi & TrainErr & Accuracy &  TestAcc  \\
		\hline
		500 & -- & -- & 1/10N & 121 & 7.64e-6 & 1.51e-11& 1.50e-8& 3.66e-3 & 0.975 &0.970\\
		500 & -- & -- & 1/N & 147  & 6.02e-9 & 8.89e-9& 2.92e-11& 3.88e-3 & 0.981 &0.974\\
		500 & --& -- & 10/N & 168 & 9.05e-9 & 4.80e-14& 1.77e-11& 3.71e-3 & 0.977 &0.969\\
		500 & --& -- & 100/N & 155 & 1.21e-8 & 0& 2.37e-11& 3.91e-3  & 0.975 &0.967\\
		\hline
		200 & 100 & -- & 1/10N & 113  & 7.38e-5 & 6.52e-5& 4.48e-7& 3.43e-2 & 0.970 & 0.961 \\		
		200 & 100 & -- & 1/N & 146  & 6.97e-5 & 4.06e-5& 3.56e-7& 2.76e-2 & 0.984 &0.969 \\
		200 & 100 & -- & 10/N & 133 & 4.21e-6 & 5.99e-11& 1.36e-8& 2.98e-2 & 0.980 &0.968 \\
		200 & 100 & -- & 100/N & 149 & 6.66e-6 & 1.61e-12& 2.15e-8& 2.80e-2 & 0.981 &0.969 \\
		\hline
		100 & 100 &100 & 1/10N & 88  & 9.53e-4 & 4.53e-4& 4.54e-6& 1.95e-1 & 0.772 & 0.776 \\		
		100 & 100& 100 & 1/N & 135 & 1.73e-4 & 7.39e-5& 7.96e-7& 4.04e-2 & 0.959 &0.955 \\
		100 & 100 &100 & 10/N & 87 & 3.94e-5 & 0 & 1.27e-7& 6.63e-2 & 0.935 &0.927 \\
		100 & 100 &100 & 100/N &  99 & 9.10e-6 & 0& 2.94e-8& 6.60e-2 & 0.934 &0.932 \\
	\end{tabular}
\end{table}



%
	

\subsection{Comparisons with the State-of-the-art Approaches}

In this subsection, we compare IALAM with
the existing SGD-based approaches including ProxSGD in solving problem \eqref{eq:dnn} through different ways.

{\subsubsection{Testing on Synthetic Data Sets}
	
	The synthetic data sets are generated with $N=500$, $N_0=5$ and $N_L=1$.
	We compare our IALAM with vanilla SGD, Adam, Adammax, AdaGrad, AdaGradDecay
	and Adadelta. We depict the ``TrainErr" and the ``TestErr'' with the x-axis varying on CPU time.
	We can learn from Figure \ref{fig:syncompare}  that (i) IALAM converges faster than the other approaches; (ii) IALAM can always reach comparable TrainErr and TestErr with the other approaches.  }

\begin{figure}[htbp]	
	
	\centering
	\setcounter{subfigure}{0}
	\subfloat[$L=2$, $N_1=10$]{\includegraphics[width=50mm]{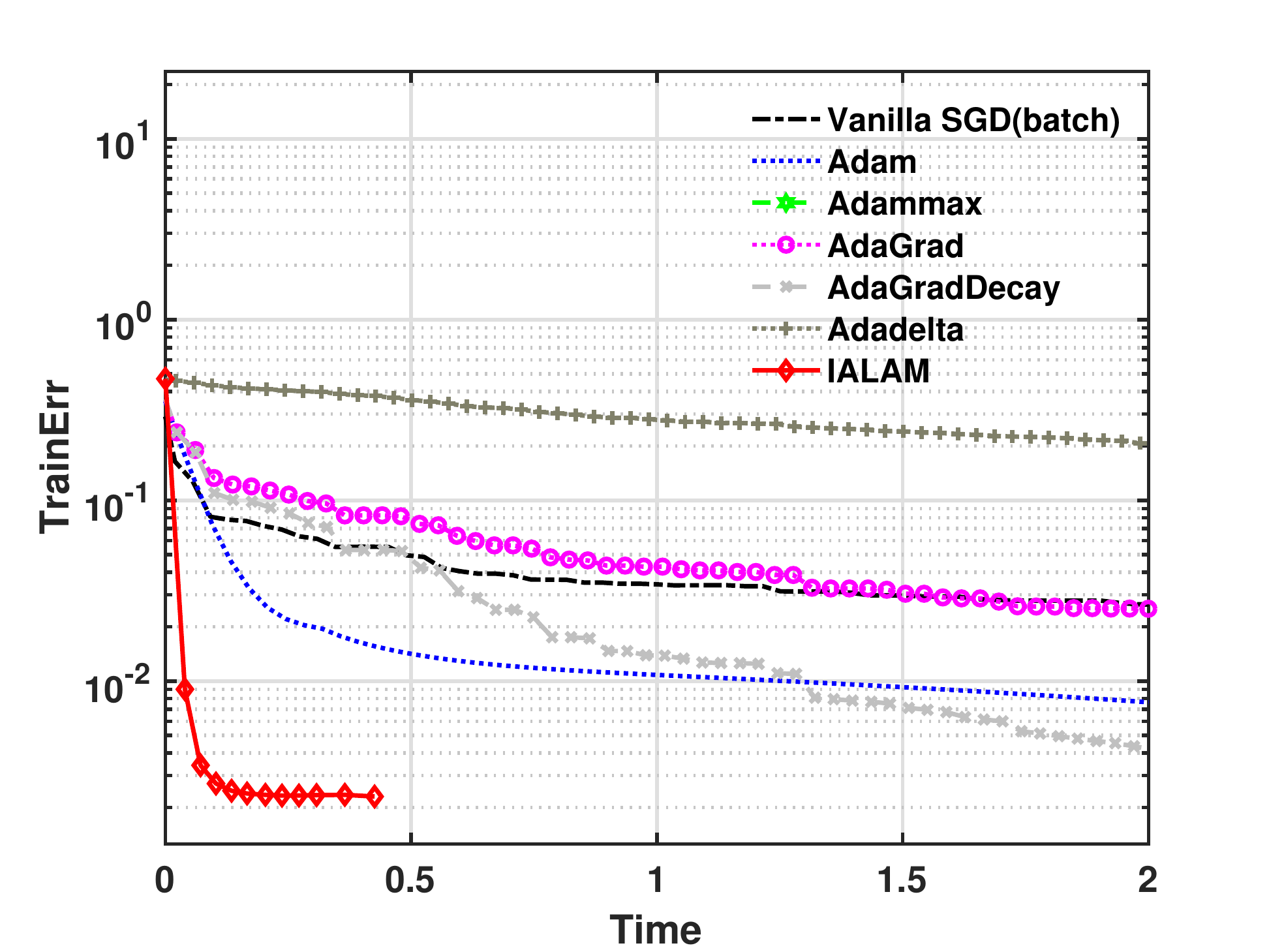}}
	\subfloat[$L=3$, $N_1=N_2=5$]{\includegraphics[width=50mm]{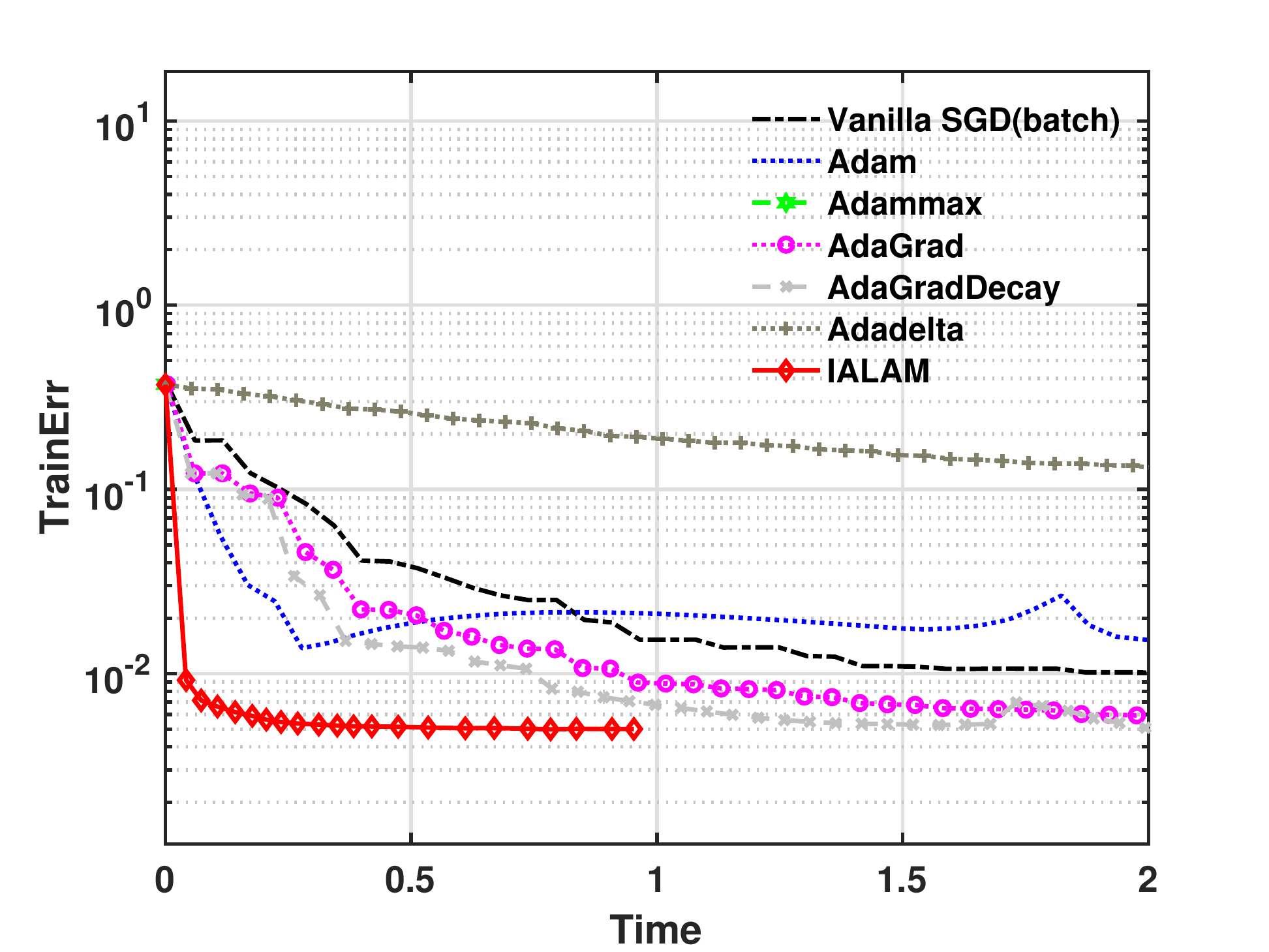}}
	\subfloat[$L=4$, $N_1=4, N_2=N_3=3$]{\includegraphics[width=50mm]{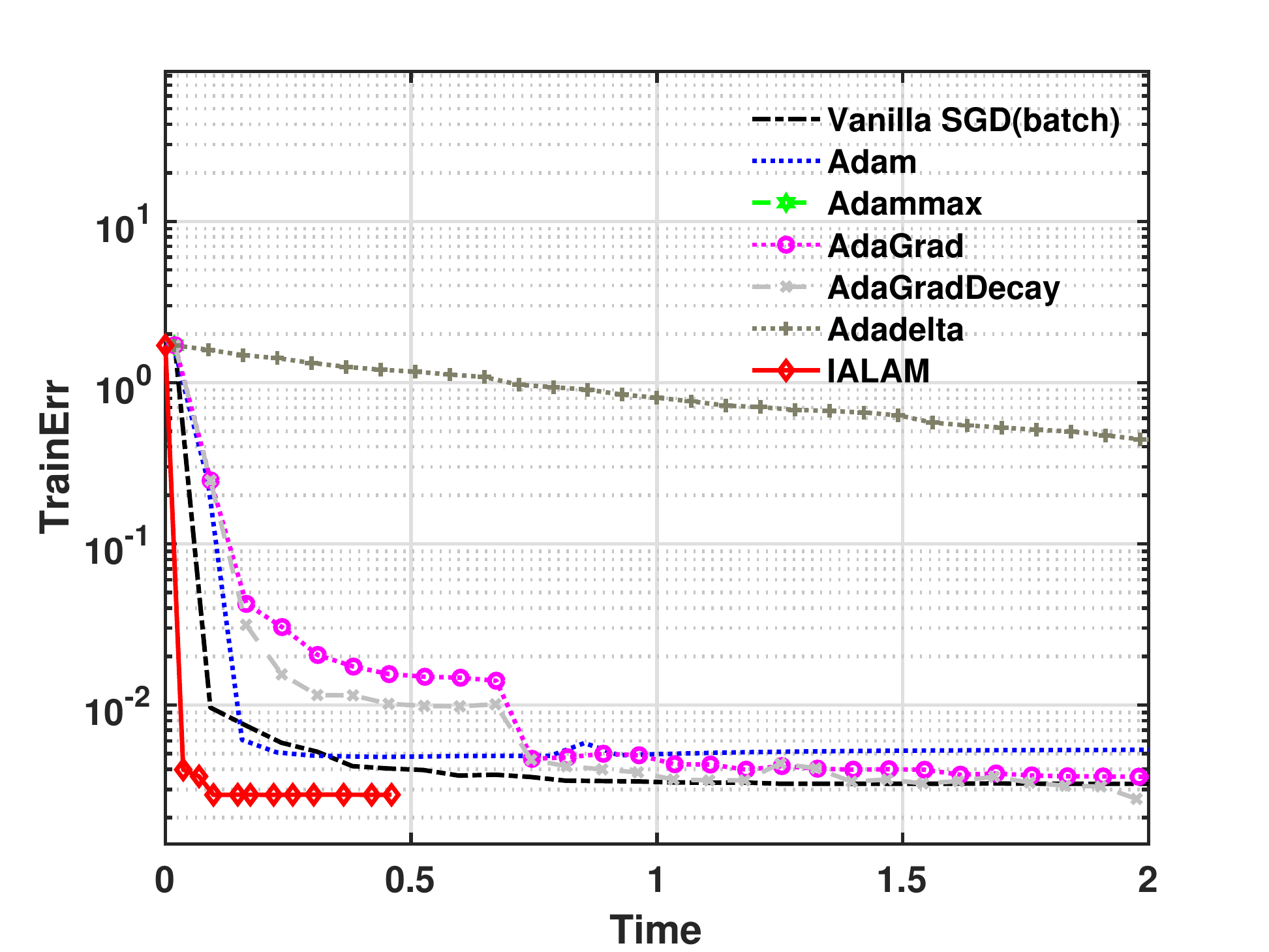}}\\
	\subfloat[$L=2$, $N_1=10$]{\includegraphics[width=50mm]{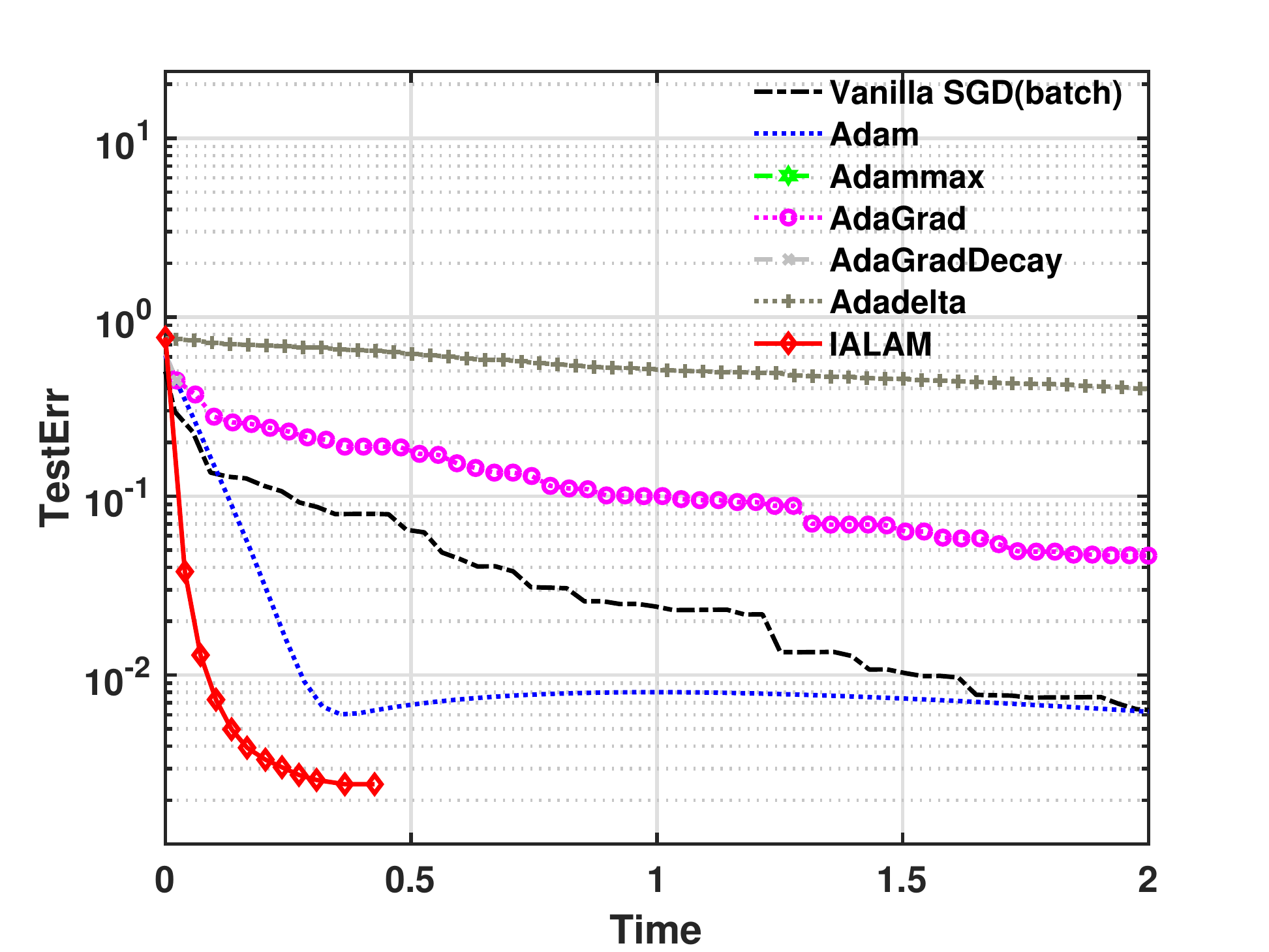}}
	\subfloat[$L=3$, $N_1=N_2=5$]{\includegraphics[width=50mm]{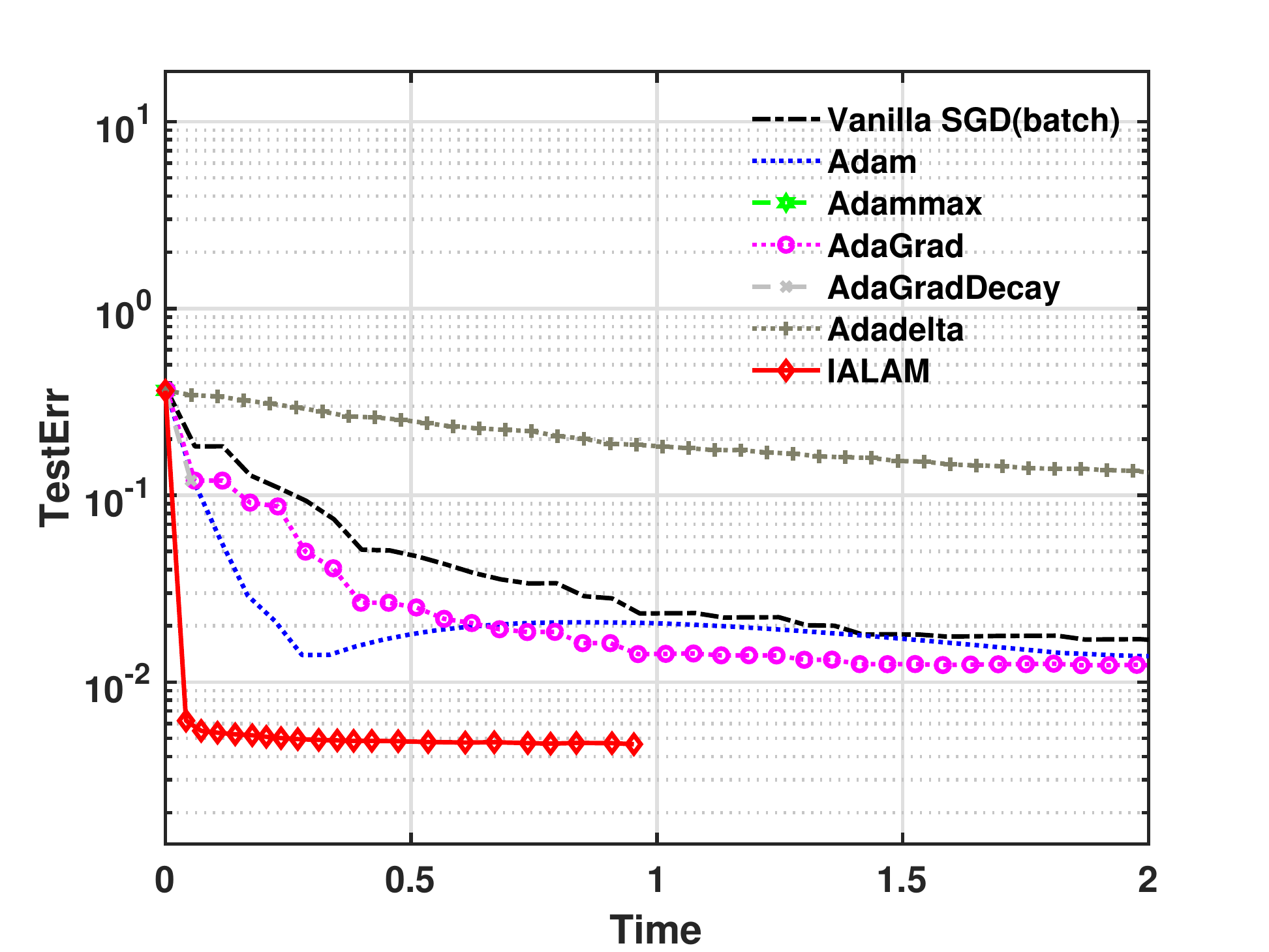}}
	\subfloat[$L=4$, $N_1=4, N_2=N_3=3$]{\includegraphics[width=50mm]{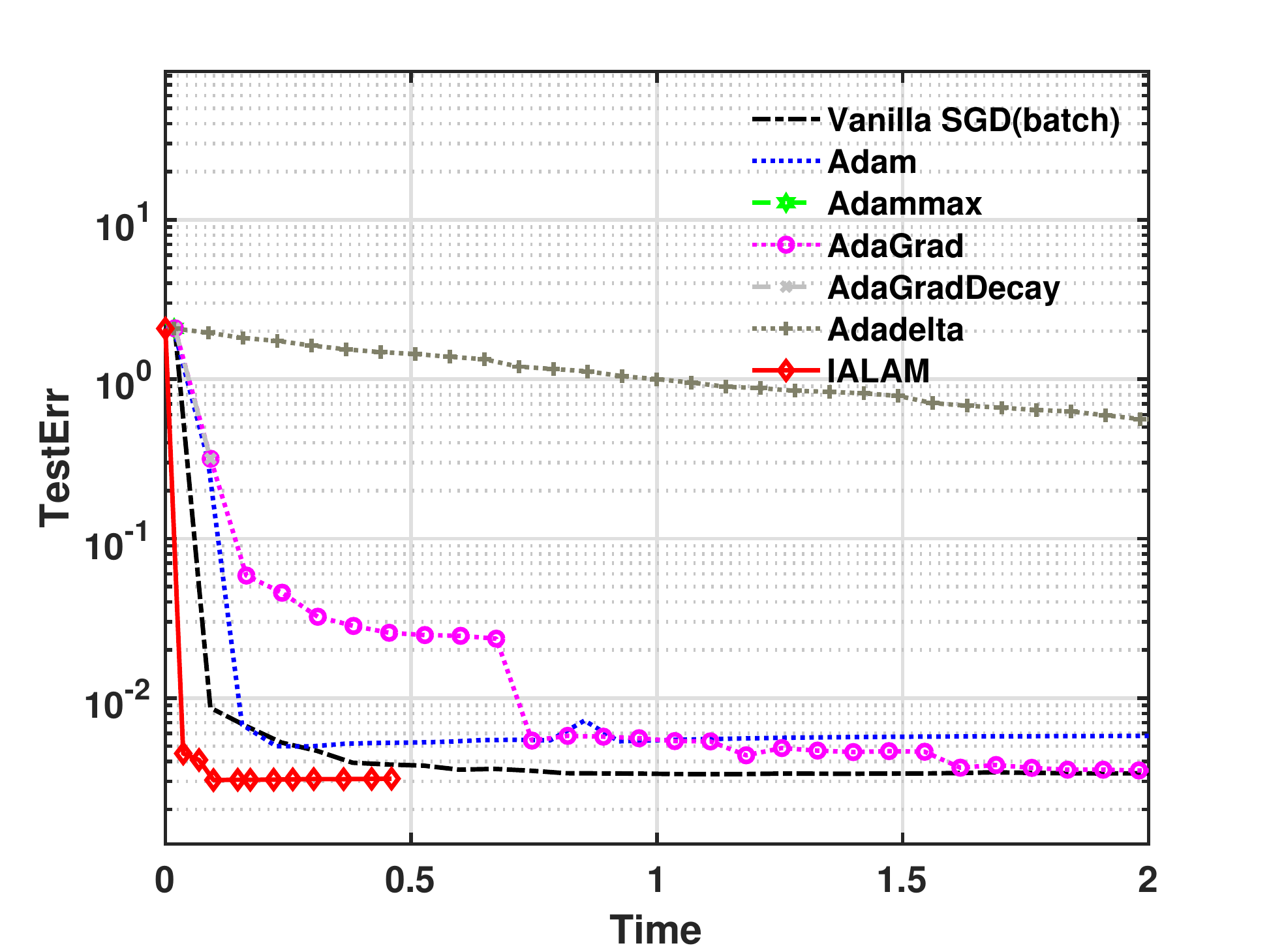}}
	\caption{Comparisons among IALAM and SGD-based approaches on the synthetic data set.}\label{fig:syncompare}
\end{figure}

\subsubsection{Testing on MNIST Data Set}		

Now we {consider} the test on MNIST data set. We {first investigate how} the ``TrainErr" and ``Accuracy"
with the x-axis varying on ``Iteration$\times$batch-size$/N$",
which is equivalent to ``iteration" for IALAM and ``epoch" for the SGD-based approaches. We also {display} the ``Column Sparsity Ratio" with the x-axis varying on ``Tolerance".
{Here,  
	let $(W_1,\ldots,W_{L})$ be the derived weight matrix of solver $s$, we denote $t^s_{\ell, j}= ||(W_{\ell})_{\cdot,j}||_2$ for all $\ell\in[L]$ and $j\in[N_{\ell-1}]$ and for a given tolerance $\omega$, the ``Column Sparsity Ratio" $r^s_{\omega}$ of solver $s$ is defined by
	$$
	r^s_{\omega}:= \sum_{\ell=1}^{L}\sum_{j=1}^{N_{\ell-1}}\delta(t^s_{\ell, j} \leq \omega) \left/\sum_{\ell=0}^{L-1}N_{\ell}\right.,
	$$
	where 
	$\delta(\Gamma)=1$ if the statement ``$\Gamma$" is true, otherwise $\delta(\Gamma)=0$.}


We can conclude from Figures \ref{fig:mnistcompare}--\ref{fig:mnistcompare2} that (i) IALAM can find {sparser} solution than the SGD-based approaches; (ii) IALAM can yield comparable TrainErr and Accuracy with other approaches, if not better; (iii) ProxSGD can find {as
	sparse solutions as those of} IALAM but much worse behavior on ``TrainErr" and ``Accuracy".

\begin{figure}[htbp]	
	\centering
	\setcounter{subfigure}{0}
	\subfloat[TrainErr]{\includegraphics[width=50mm]{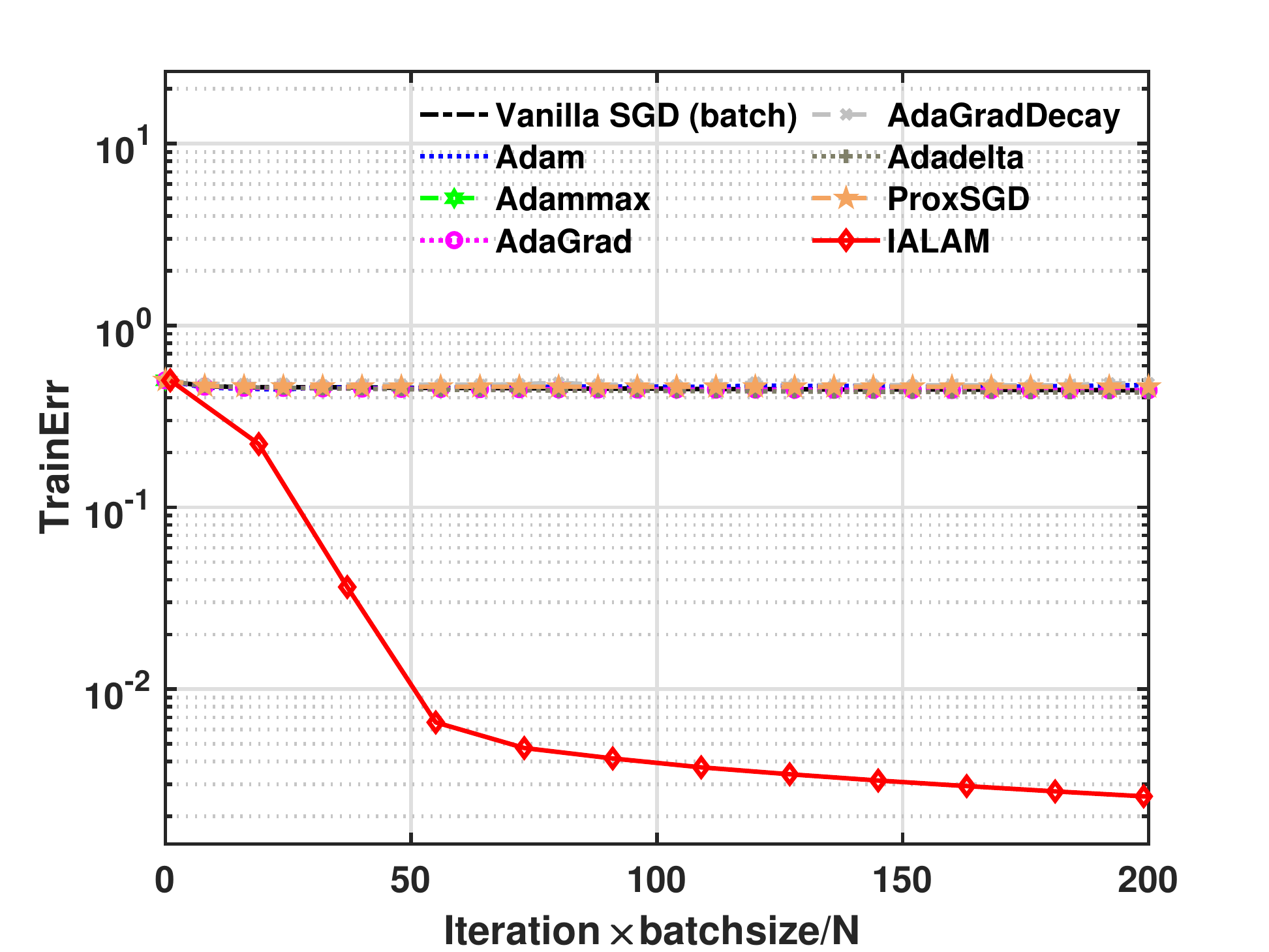}}
	\subfloat[Accuracy]{\includegraphics[width=50mm]{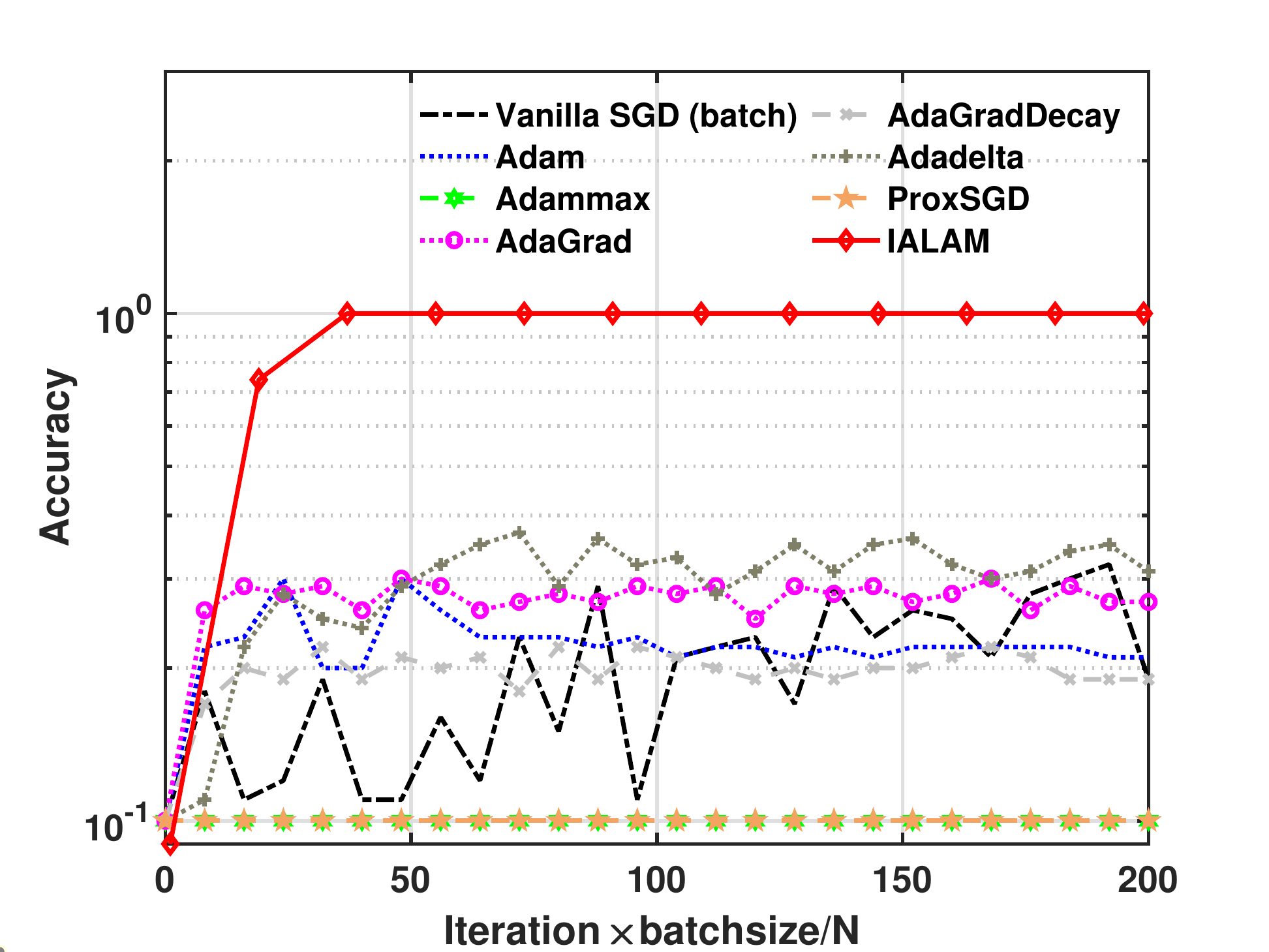}}
	\subfloat[Column Sparse Ratio]{\includegraphics[width=50mm]{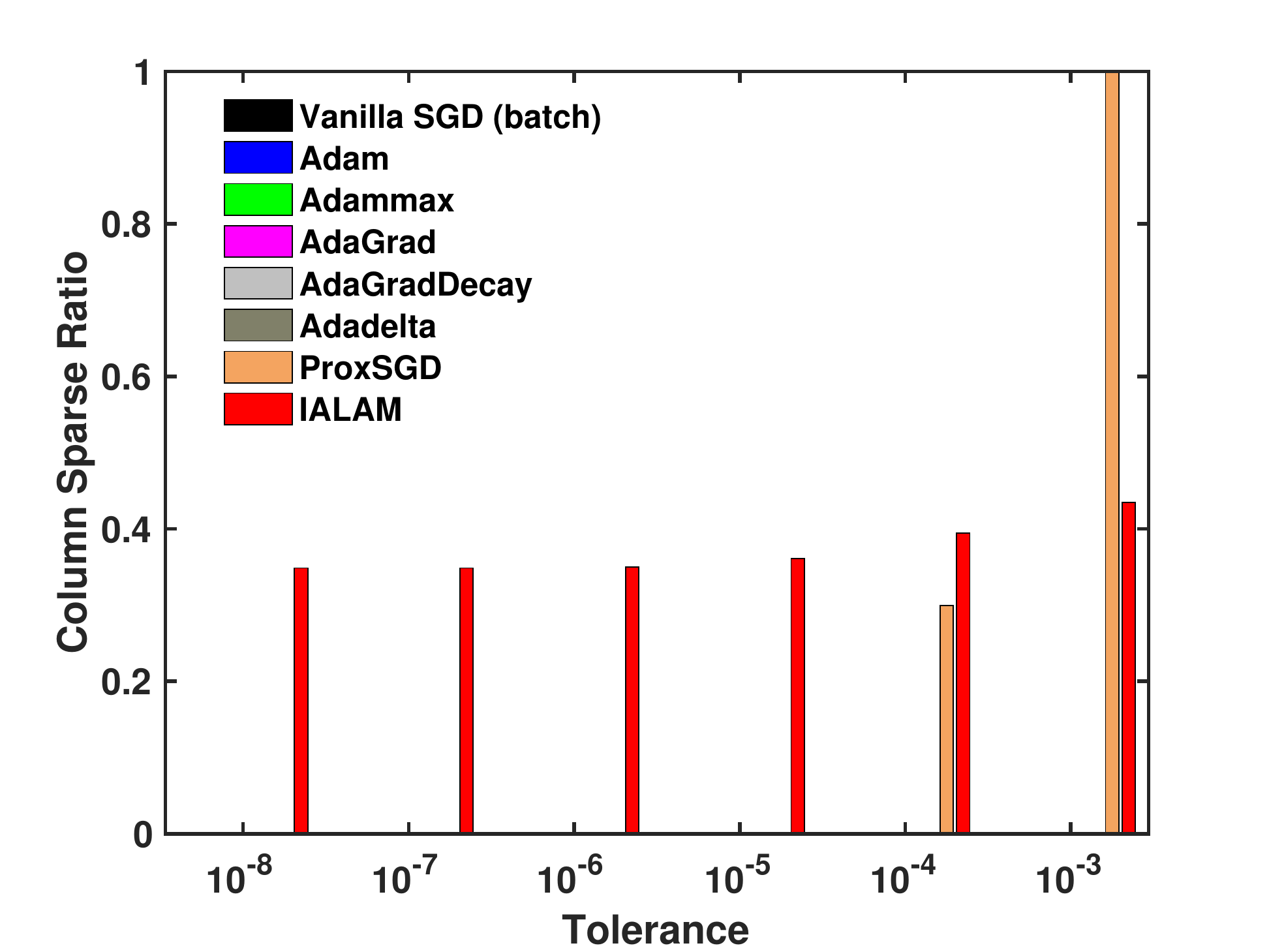}}\\
	\subfloat[TrainErr]{\includegraphics[width=50mm]{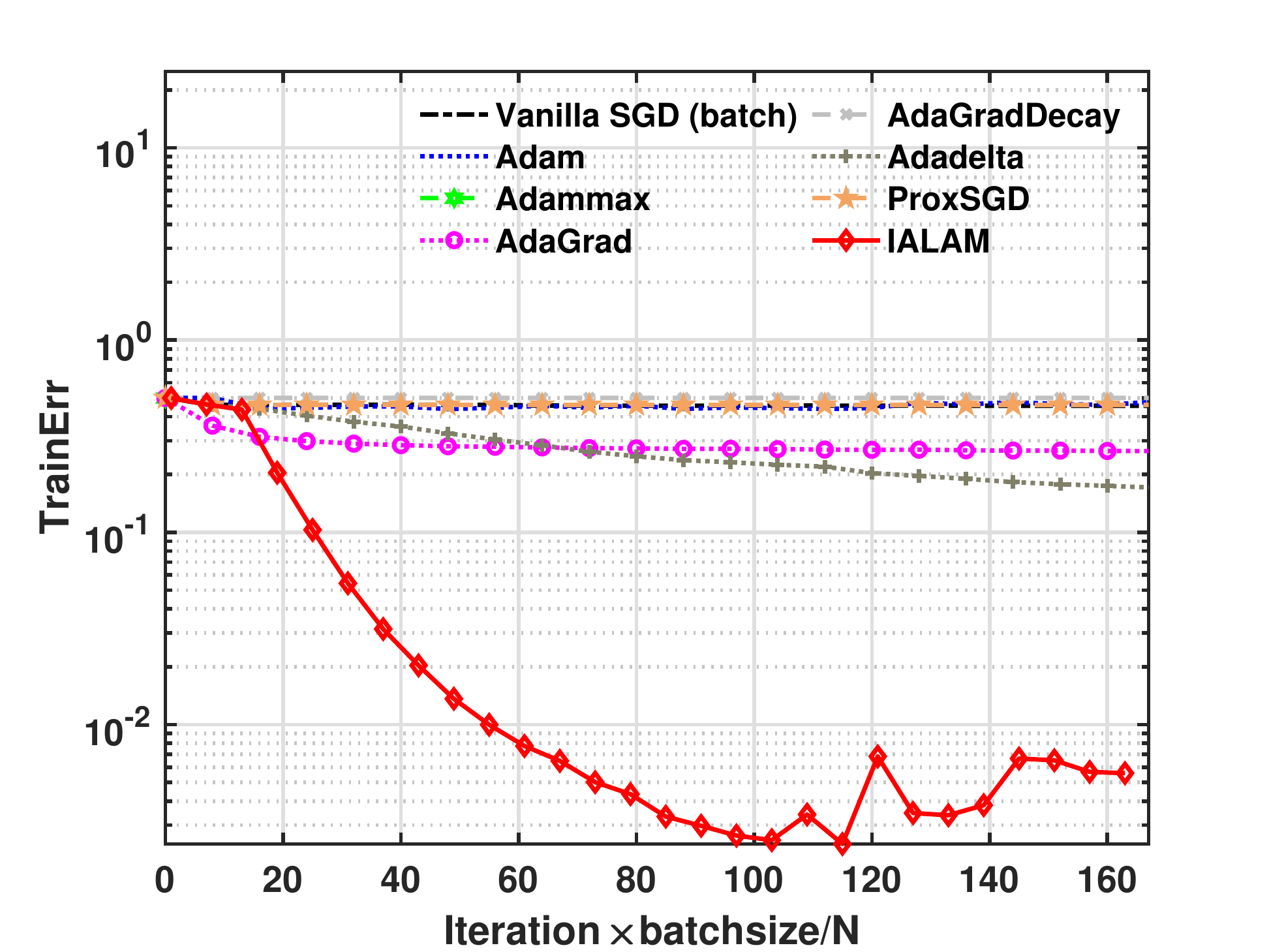}}
	\subfloat[Accuracy]{\includegraphics[width=50mm]{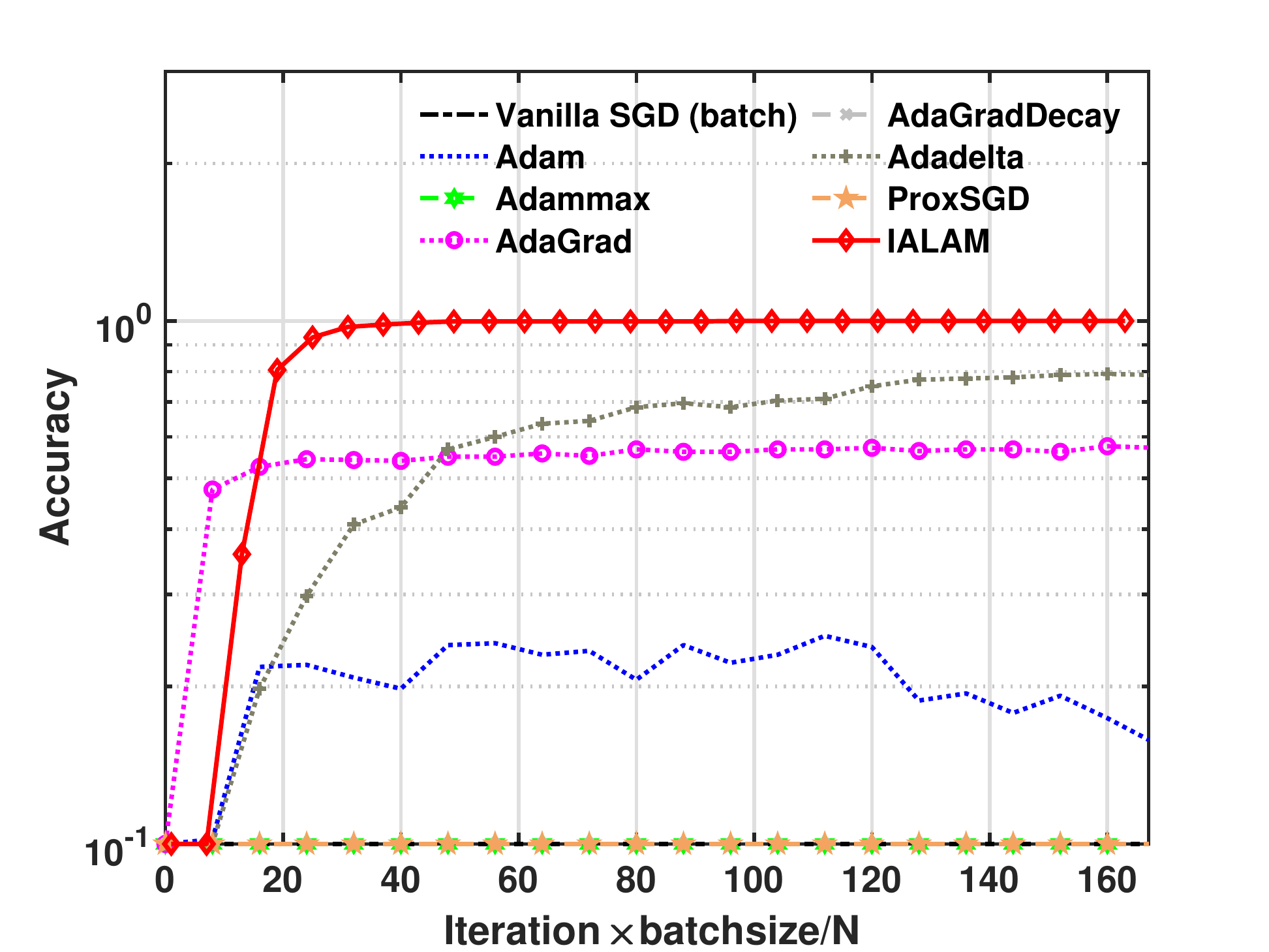}}
	\subfloat[Column Sparse Ratio]{\includegraphics[width=50mm]{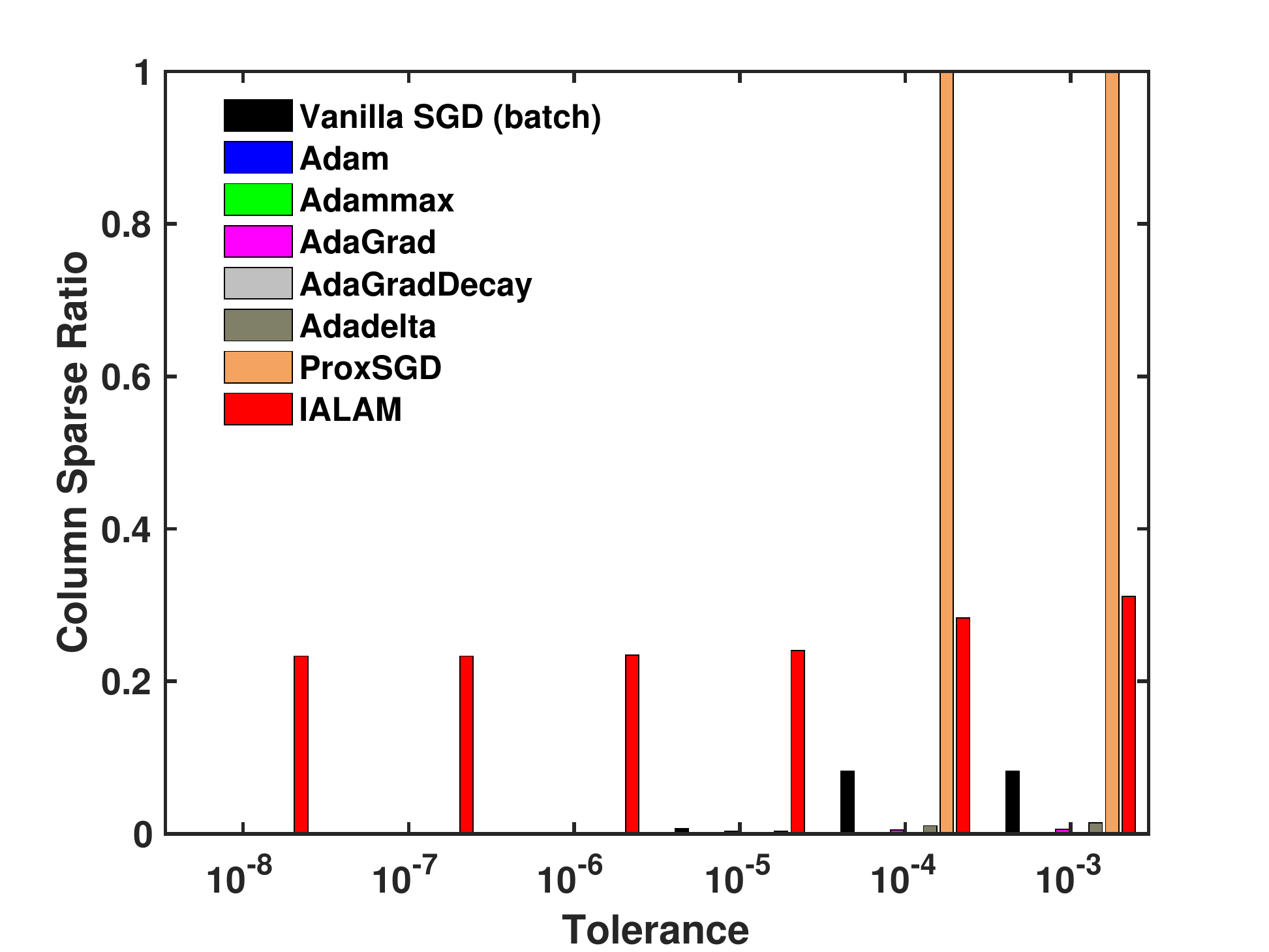}}\\
	\subfloat[TrainErr]{\includegraphics[width=50mm]{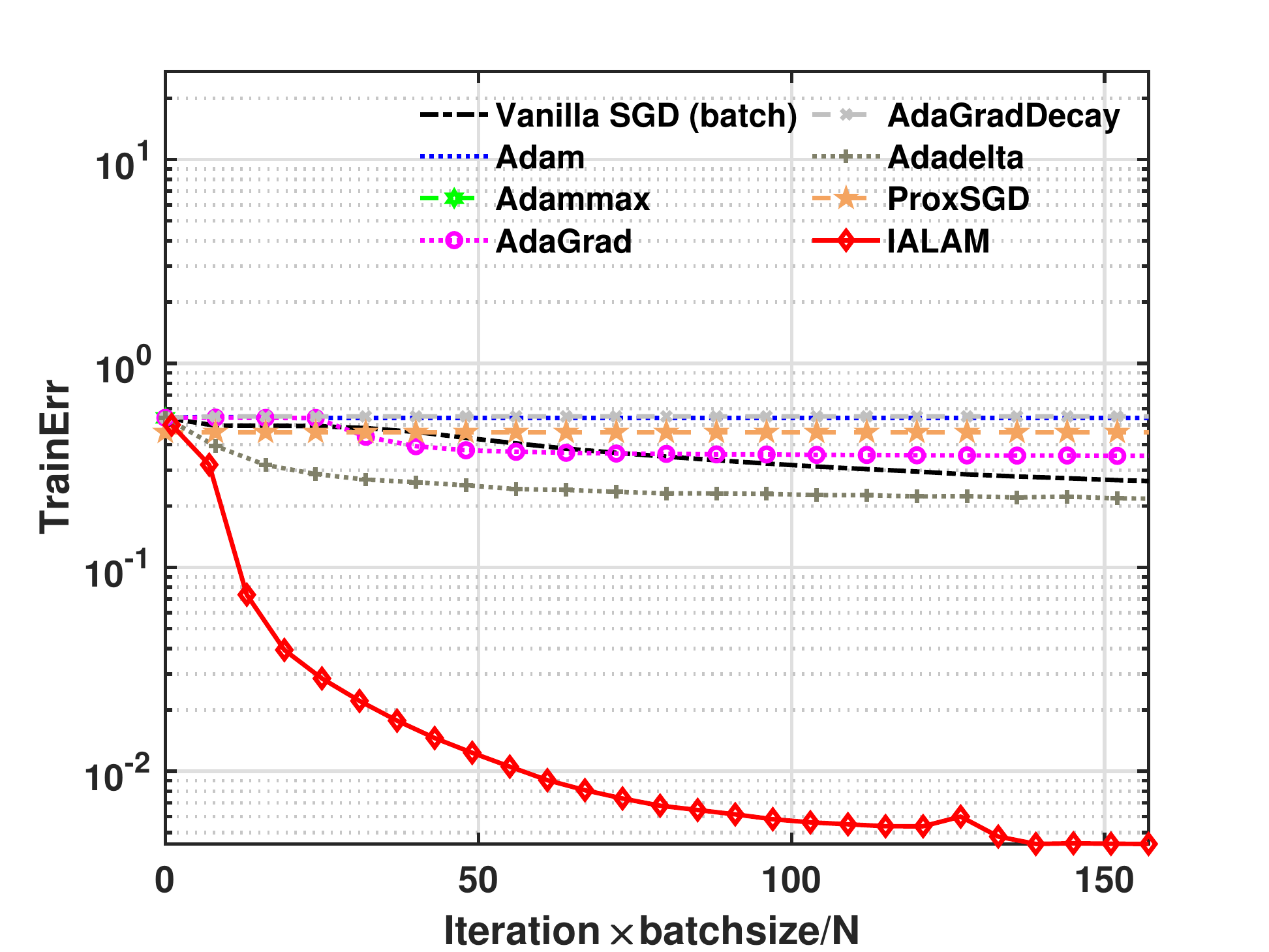}}
	\subfloat[Accuracy]{\includegraphics[width=50mm]{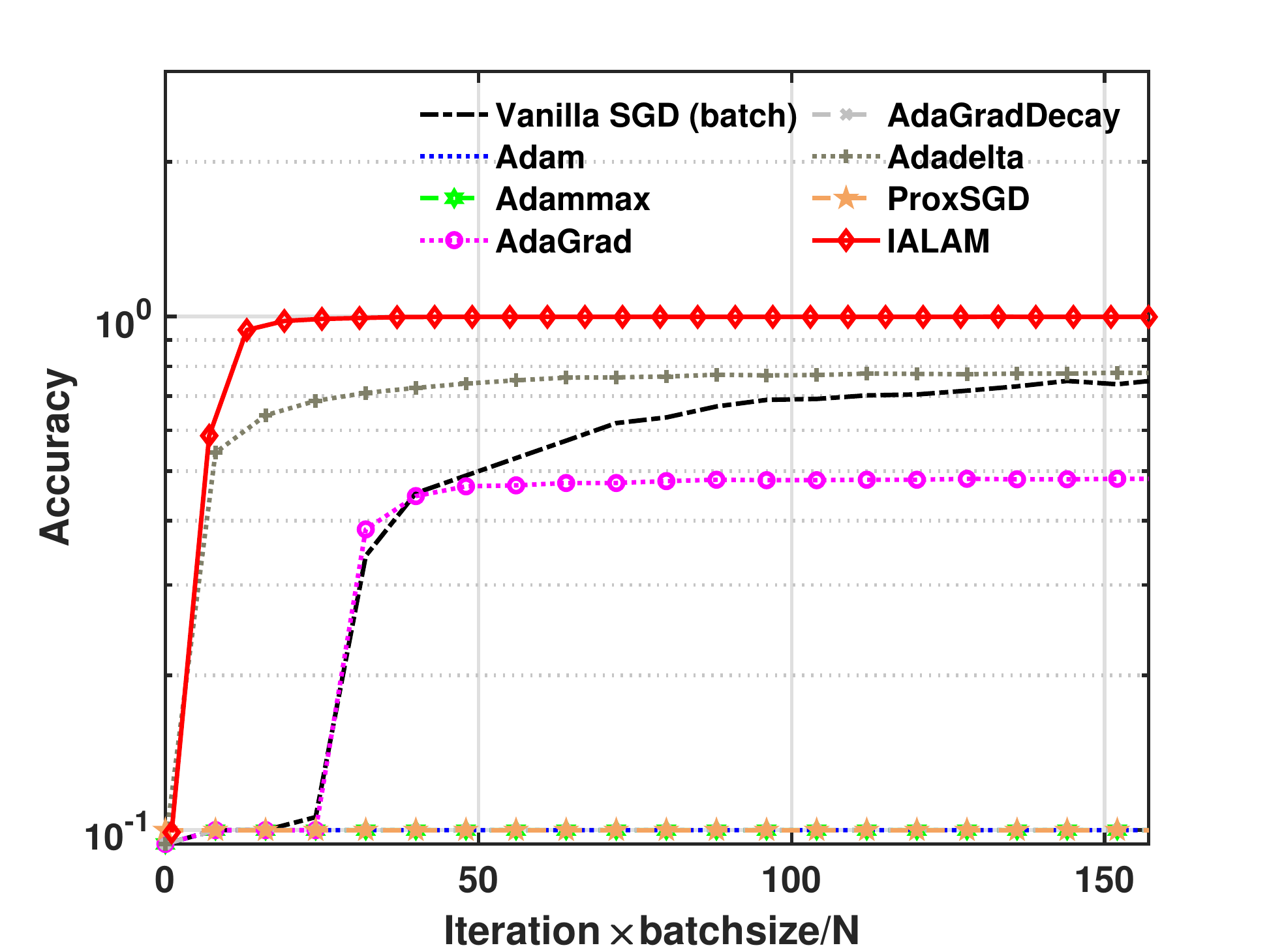}}
	\subfloat[Column Sparse Ratio]{\includegraphics[width=50mm]{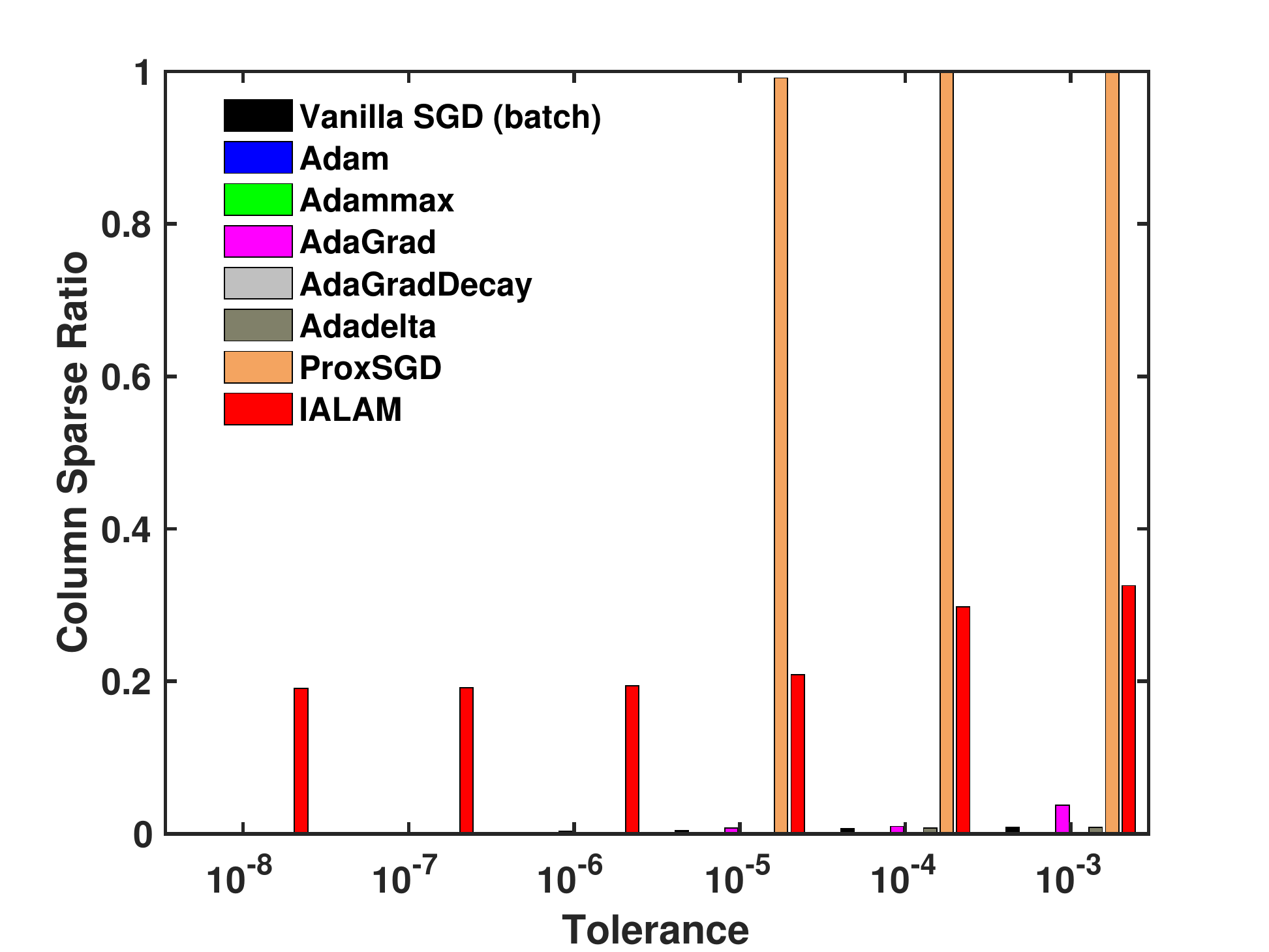}}\\
	\subfloat[TrainErr]{\includegraphics[width=50mm]{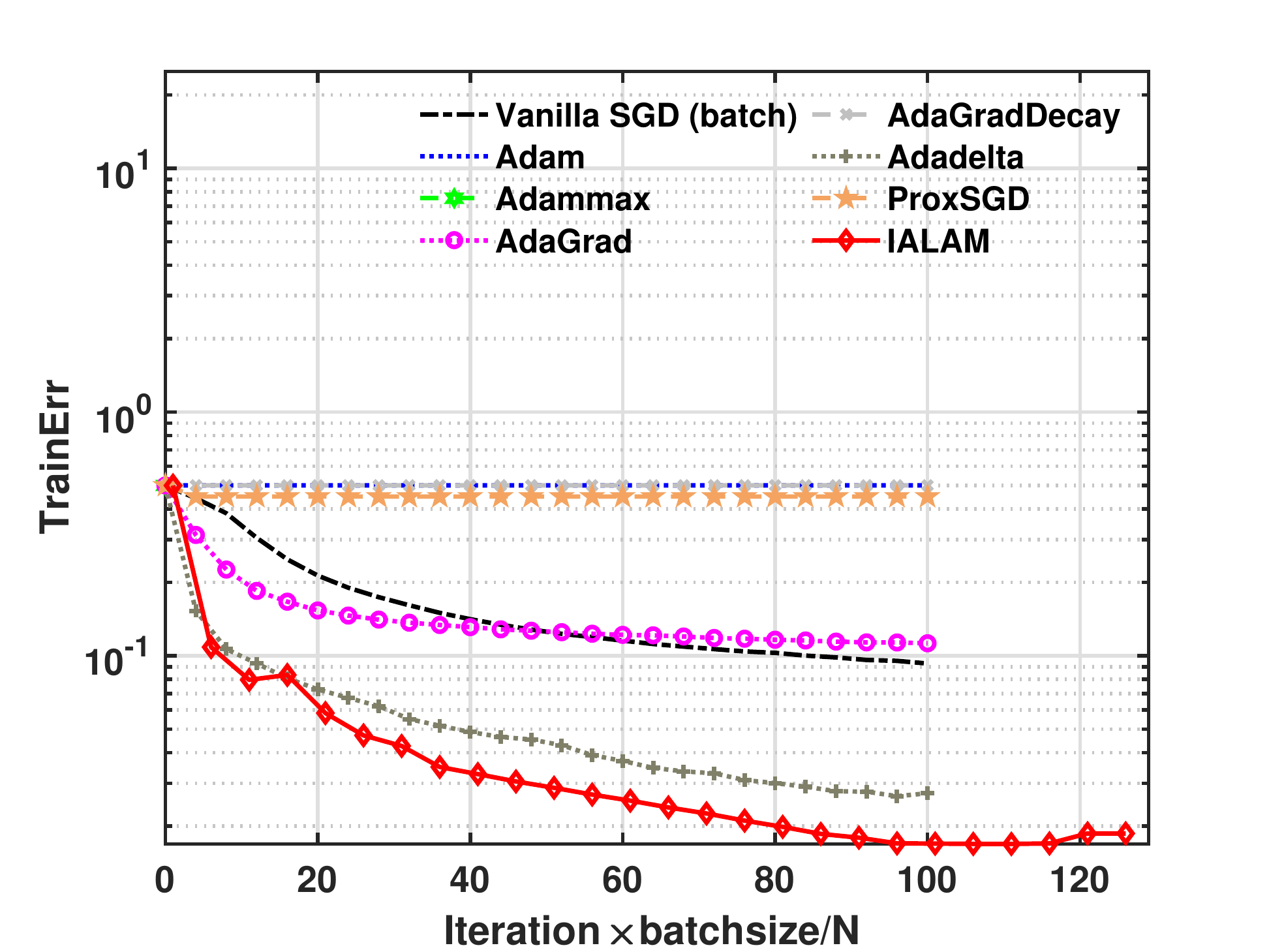}}
	\subfloat[Accuracy]{\includegraphics[width=50mm]{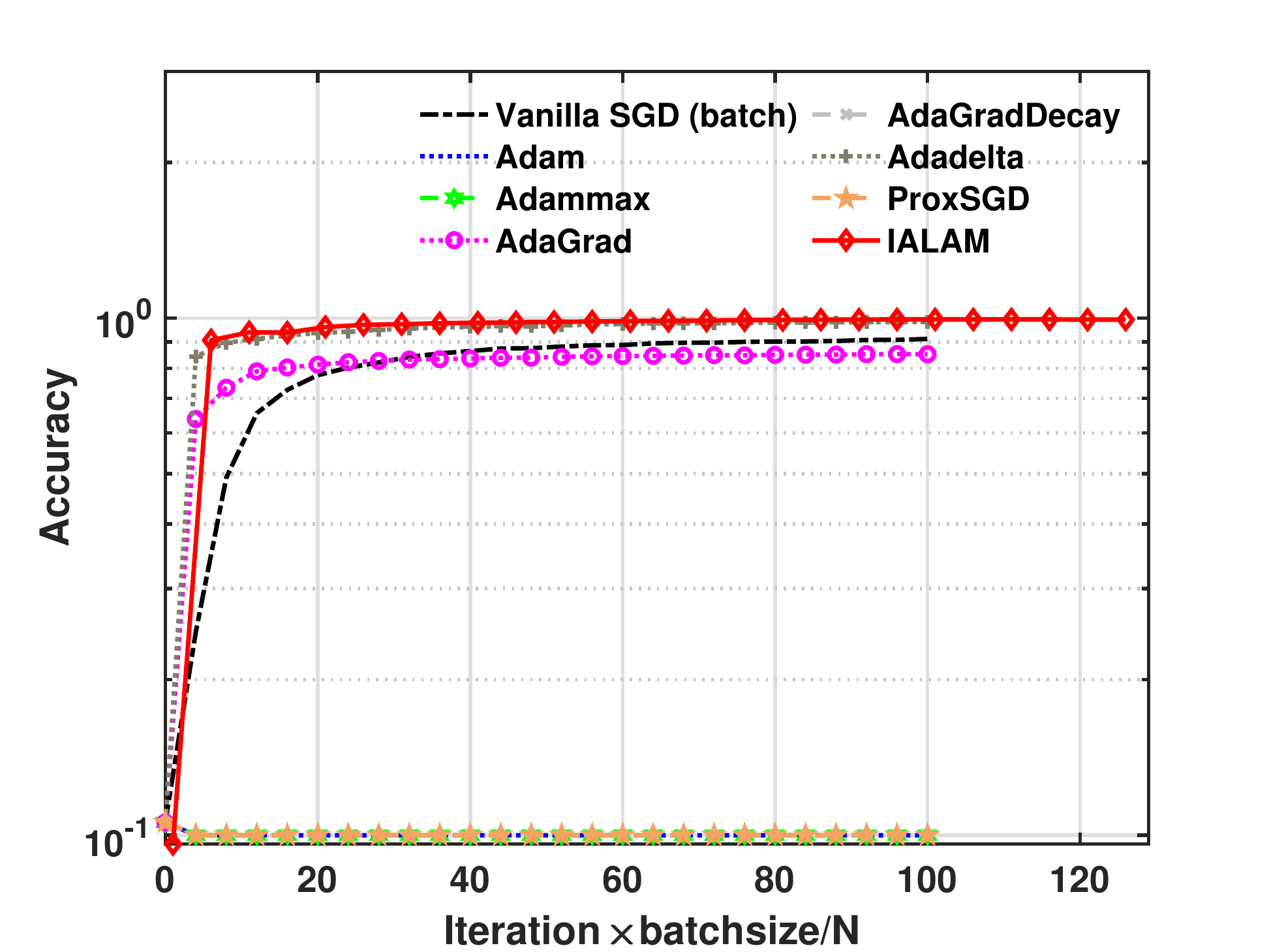}}
	\subfloat[Column Sparse Ratio]{\includegraphics[width=50mm]{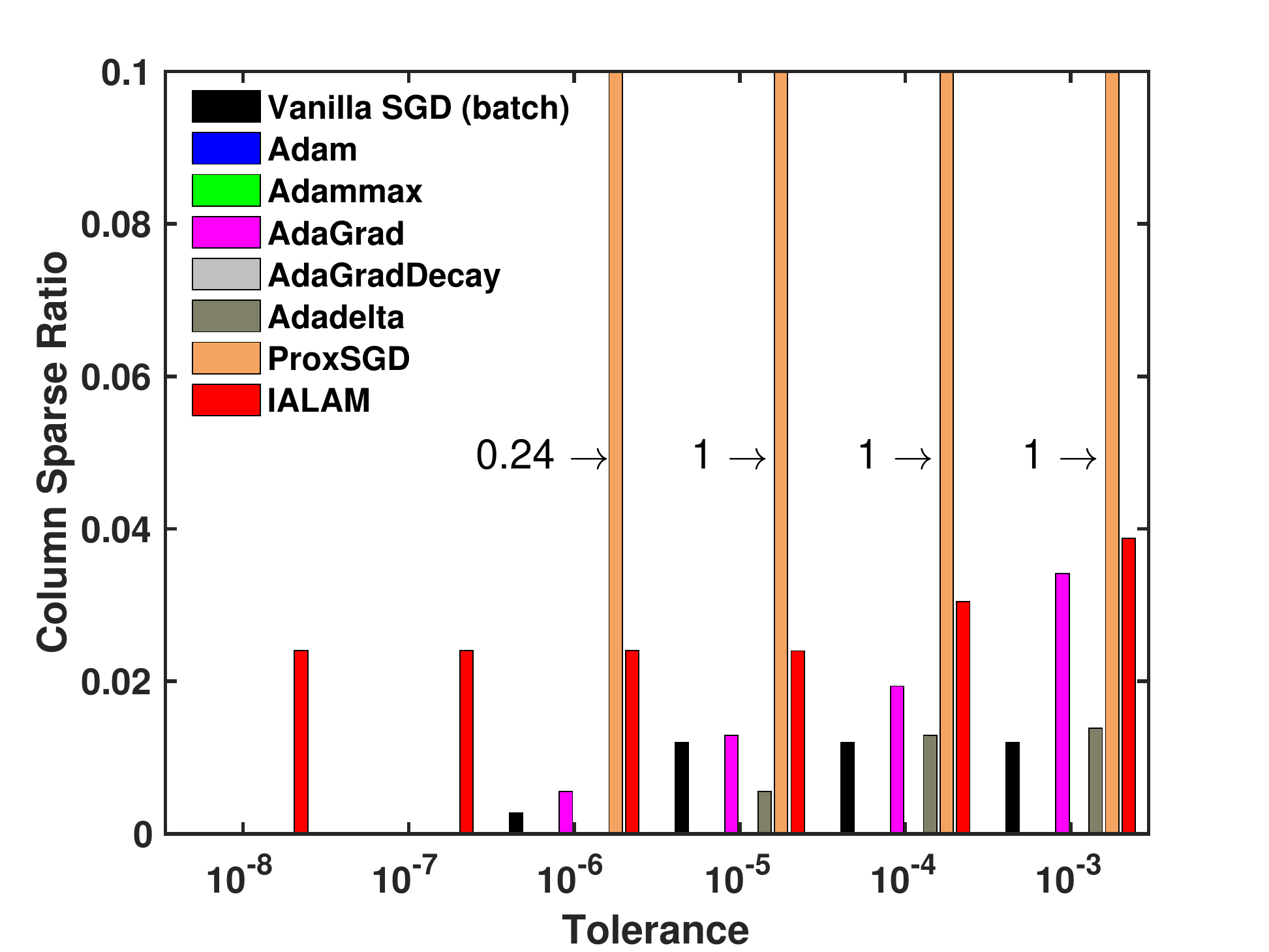}}
	\caption{Comparisons among IALAM and SGD-based approaches on MNIST with (a)--(c): $N=100$, $N_1=5$, $L=2$; (d)--(f): $N=500,N_1=50,N_2=20,L=3$; (g)--(i): $N=1000,N_1=100,N_2=50,L=3$; (j)--(l) $N=5000,N_1=200,N_2=100,L=3$.}
	\label{fig:mnistcompare}
\end{figure}

\begin{figure}[htbp]	
	\centering
	\setcounter{subfigure}{0}
	\subfloat[TrainErr]{\includegraphics[width=50mm]{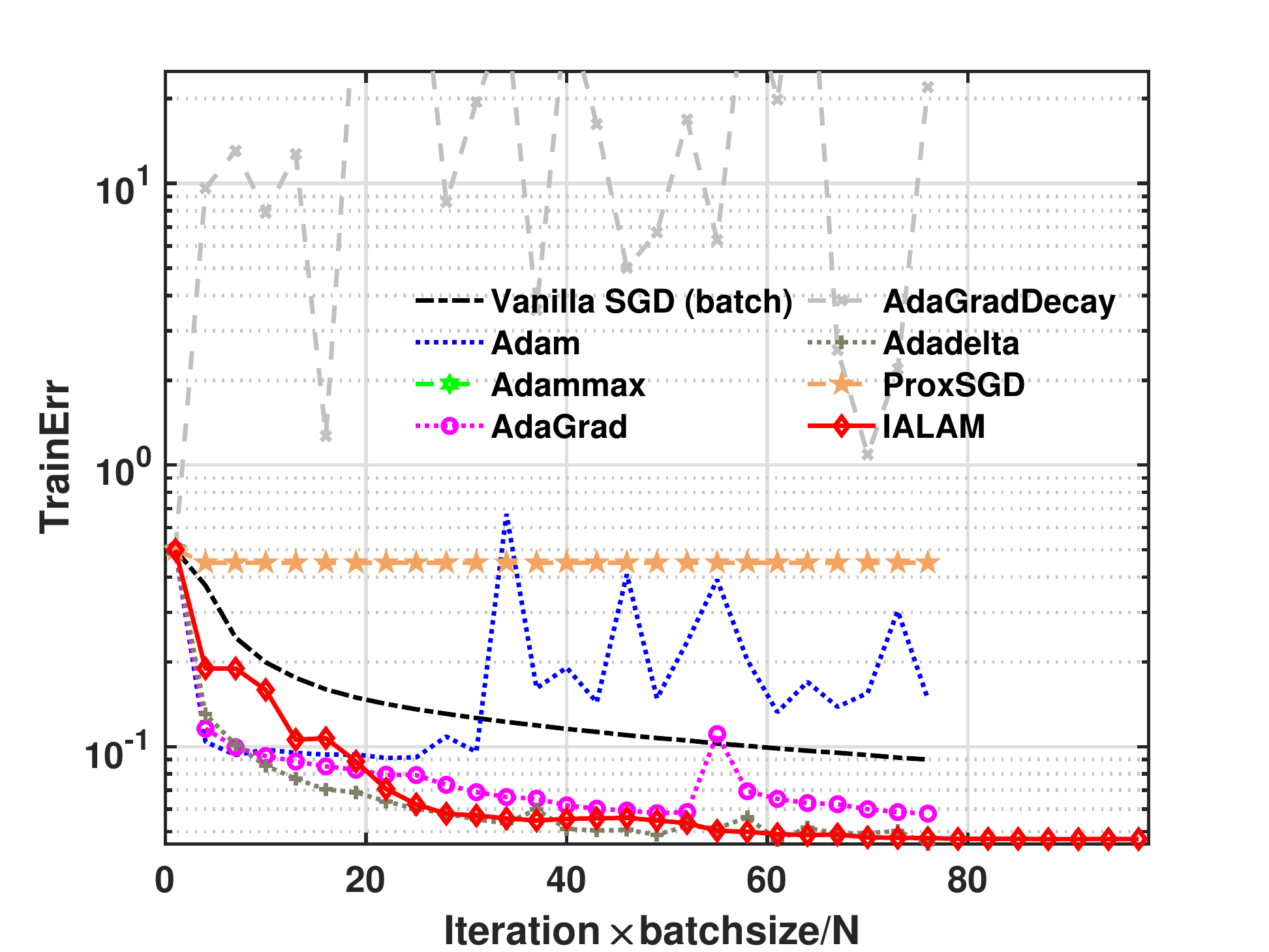}}
	\subfloat[Accuracy]{\includegraphics[width=50mm]{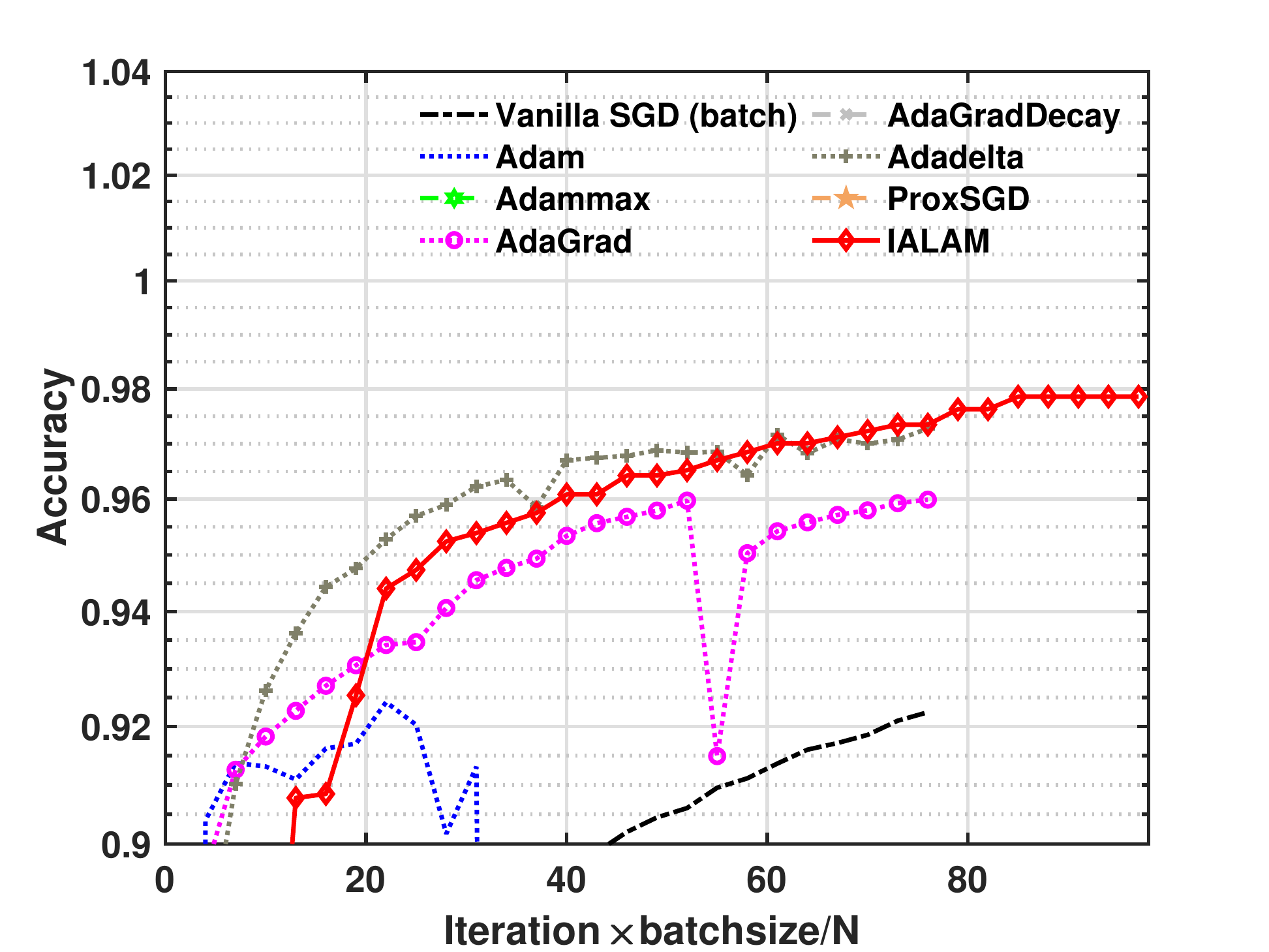}}
	\subfloat[Column Sparse Ratio]{\includegraphics[width=50mm]{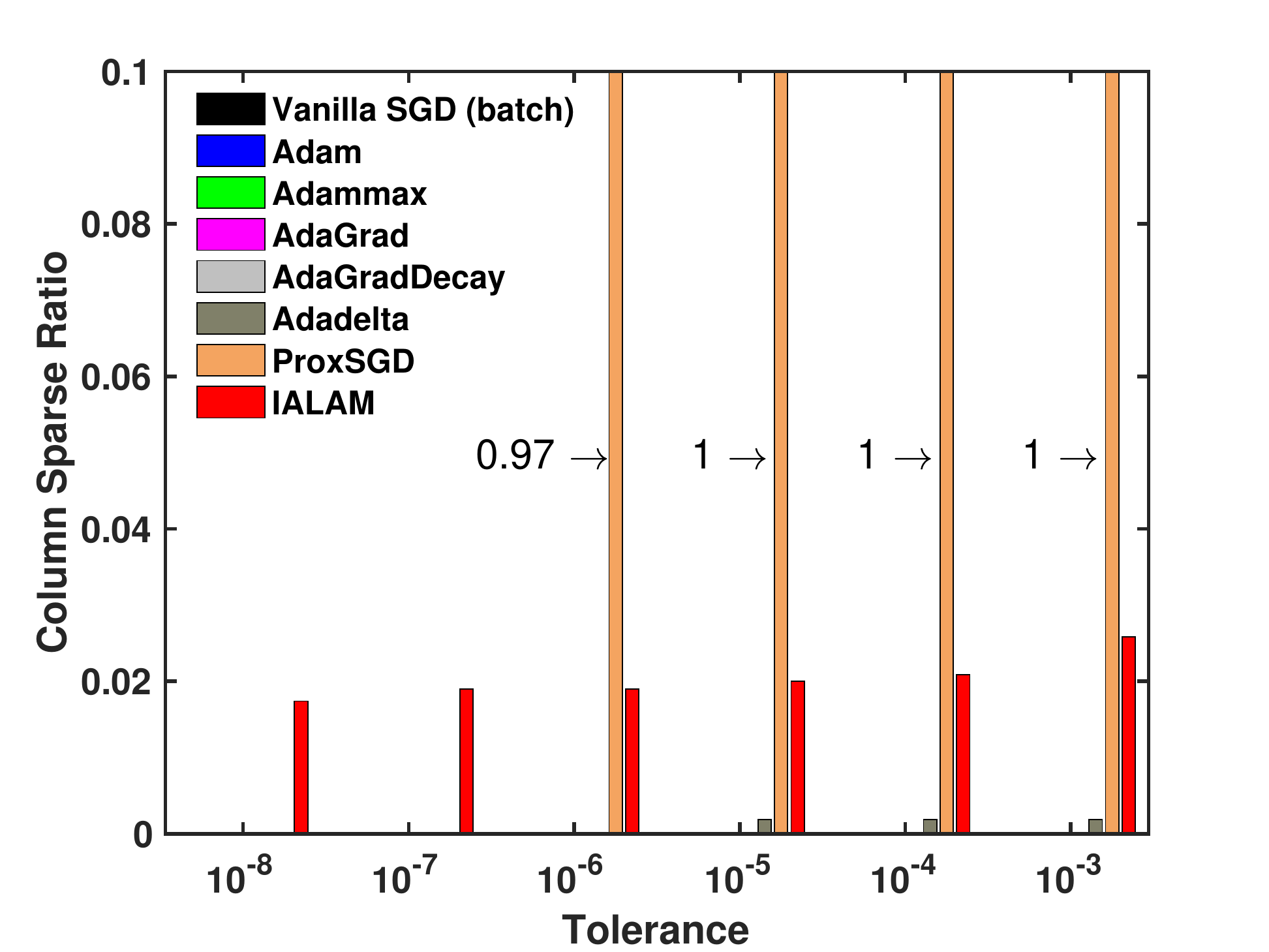}}
	\caption{Comparisons among IALAM and SGD-based approaches on MNIST with $N=60000,N_1=200,N_2=100,L=3$.}
	\label{fig:mnistcompare2}
	
\end{figure}

Finally, we select $720$ test problems based on MNIST data set with
different network parameter combinations $\{(L=2, N_1=20), (L=2, N_1=50), (L=3, N_1=20, N_2=10), (L=3, N_1=50, N_2=20),  (L=4, N_1=10, N_2=10, N_3=10), (L=4, N_1=40, N_2=20, N_3=10)\}$, $\alpha\in\{0,0.01,0.05,0.1\}$, $\lambda_w=i/10N$, $i\in\{1,2,\ldots,10\}$, $N\in\{100,500,2000\}$. We investigate the performance profiles \citep{dolan2002benchmarking} of Vanilla SGD, Adam,
Adamdelta and our IALAM through three measurements ``TrainErr",  {and ``TestErr"}.
We terminate Valinna SGD, Adadelta and Adam whenever the epoch reaches $100$.
We describe how to plot the performance profiles.
For problem $p$ and solver $s$, we use $t_{p}^s$ to represent the output meansurement (``TrainErr" {or ``TestErr"}).  
	Performance ratio is defined as $r_{p}^s:=t_{p}^s / \min _{s}\left\{t_{p}^s\right\}.$ If solver $s$ fails to solve problem $p$, the ratio $r_{p}^s$ is set to 10000. Finally, the overall performance of solver $s$ is defined by
	$$
	\pi_{s}(\omega):=\sum_{p=1}^{720}\delta (r_{p}^s \leq \omega)\bigg/720.
	$$

Clearly, the closer $\pi_{s}$ is to 1, the better performance the solver $s$ has.
The performance profiles with respect to ``TrainErr" and ``TestErr" are given in Figure \ref{fig:perpro}.
We can conclude that IALAM outperforms the others with respect to {both ``TrainErr" and ``TestErr"}. It is worth noting that we use the same batch-size for Valinna SGD, Adadelta, and Adam, and stop them whenever the {the number of epochs reaches $100$.} 

\begin{figure}[htbp]	
	\centering
	\setcounter{subfigure}{0}
	\subfloat[TrainErr]{\includegraphics[width=75mm]{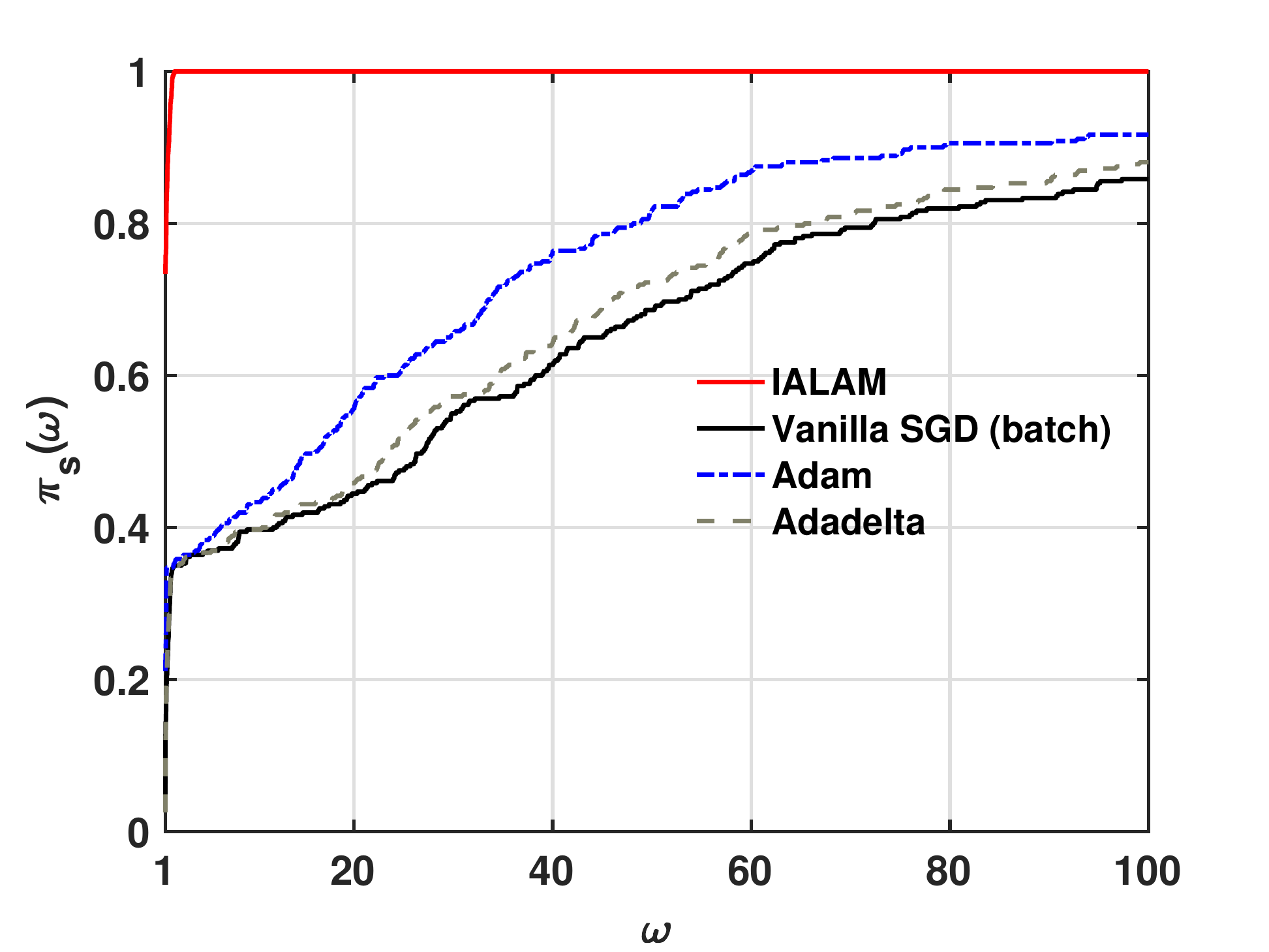}}
	\subfloat[TestErr]{\includegraphics[width=75mm]{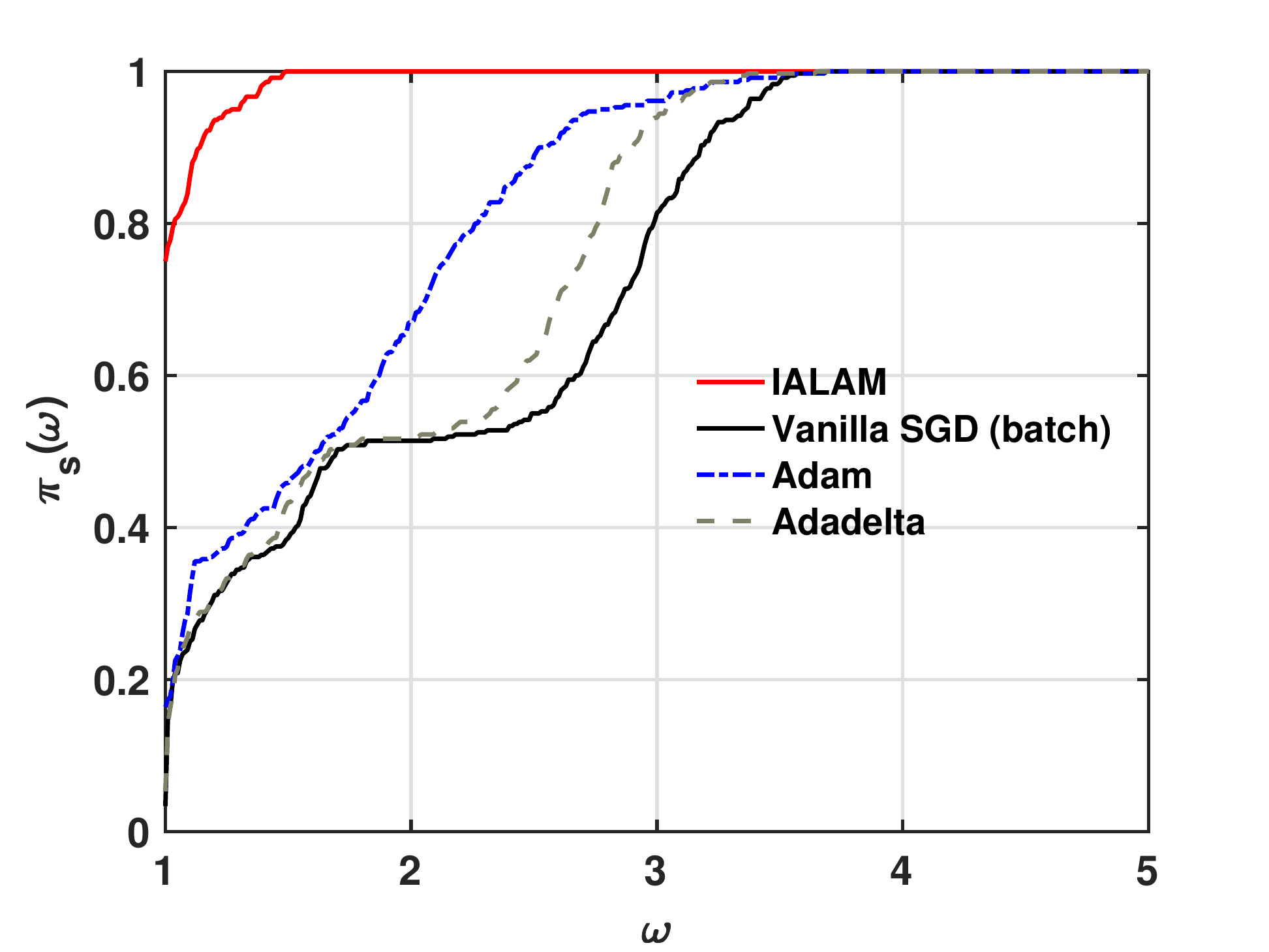}}
	\caption{Performance profile for IALAM, Valinna SGD, Adadelta and Adam on TrainErr and TestErr.}\label{fig:perpro}
\end{figure}

\section{Conclusion}

{We focus on the regularized minimization model (P) for training
	leaky ReLU with group sparsity. We first present an $l_{1}$-norm penalty model (named PP) for problem (P) and then theoretically demonstrate that these two models share the same global minimizers, local minimizers and limiting stationary points under mild conditions. In addition, we prove} that problem (PP) has a nonempty and bounded solution set and its feasible set satisfies the MFCQ, under which the KKT point of (PP) is also an MPCC W-stationary point of problem (P).	
We {propose} an inexact augmented Lagrangian algorithm with the alternating minimization (IALAM) to solve problem (PP).
{The global convergence to the KKT point has been established. Comprehensive numerical experiments have illustrated the efficiency of IALAM as well as its ability to seek sparse solution.} 

\acks{We would like to acknowledge support for this project
	from the National Natural Science Foundation of China (No. 12125108, 11971466, 12288201, 12021001 and 11991021),
	Hong Kong Research Grants Council grant PolyU15300021,
	Key Research Program of Frontier Sciences,
	Chinese Academy of Sciences (No. ZDBS-LY-7022) and the CAS AMSS-PolyU Joint Laboratory in Applied Mathematics. 
}

\vskip 0.2in
\bibliography{dnnref}

\end{document}